\documentclass[11pt]{amsart}
\usepackage{amssymb,amsmath,amsthm,amsfonts} 
\usepackage{enumerate}
\usepackage {latexsym}
\usepackage{graphicx}
\usepackage{epsfig}
\usepackage{bbm}
\include{srctex}
\usepackage{latexsym}
\usepackage{bbm,enumerate}
\usepackage{mathtools}
\usepackage{hyperref}
\usepackage{marginnote}
\usepackage[letterpaper, top=1.2in, bottom=1.2in, outer=1.1in, inner=1.1in, heightrounded, marginparwidth=2.5cm, marginparsep=2mm]{geometry}

\allowdisplaybreaks
\newcommand{\nc}{\newcommand}
\nc{\les}{\lesssim}
\nc{\nit}{\noindent}
\nc{\nn}{\nonumber}
\nc{\D}{\partial}
\nc{\diff}[2]{\frac{d #1}{d #2}}
\nc{\diffn}[3]{\frac{d^{#3} #1}{d {#2}^{#3}}}
\nc{\pdiff}[2]{\frac{\partial #1}{\partial #2}}
\nc{\pdiffn}[3]{\frac{\partial^{#3} #1}{\partial{#2}^{#3}}}
\nc{\abs}[1] {\lvert #1 \rvert}
\nc{\cAc}{{\cal A}_c}
\nc{\cE}{{\cal E}}
\nc{\cF}{{\cal F}}
\nc{\cP}{{\cal P}}
\nc{\cV}{{\cal V}}
\nc{\cQ}{{\cal Q}}
\nc{\cGin}{{\cal G}_{\rm in}}
\nc{\cGout}{{\cal G}_{\rm out}}
\nc{\cO}{{\cal O}}
\nc{\Lav}{{\cal L}_{\rm av}}
\nc{\cL}{{\cal L}}
\nc{\cB}{{\cal B}}
\nc{\cZ}{{\cal Z}}
\nc{\mR}{{\mathcal R}}
\nc{\mG}{{\mathcal G}}
\nc{\cT}{{\cal T}}
\nc{\cY}{{\cal Y}}
\nc{\cX}{{\cal X}}
\nc{\cXT}{{{\cal X}(T)}}
\nc{\cBT}{{{\cal B}(T)}}
\nc{\vD}{{\vec \mathcal{D}}}
\nc{\efield}{\mathcal{E}}
\nc{\vE}{{\vec \efield}}
\nc{\vB}{{\vec \mathcal{B}}}
\nc{\vH}{{\vec \mathcal{H}}}
\nc{\F}{  \mathcal{F} }
\nc{\ty}{{\tilde y}}
\nc{\tu}{{\tilde u}}
\nc{\tV}{{\tilde V}}
\nc{\Pc}{{\bf P_c}}
\nc{\bx}{{\bf x}}
\nc{\bX}{{\bf X}}
\nc{\bXYZ}{{\bf XYZ}}
\nc{\bY}{{\bf Y}}
\nc{\bF}{{\bf F}}
\nc{\bS}{{\bf S}}
\nc{\dV}{{\delta V}}
\nc{\dE}{{\delta E}}
\nc{\TT}{{\Theta}}
\nc{\dPsi}{{\delta\Psi}}
\nc{\order}{{\cal O}}
\nc{\Rout}{R_{\rm out}}
\nc{\eplus}{e_+}
\nc{\eminus}{e_-}
\nc{\epm}{e_\pm}
\nc{\eps}{\varepsilon}
\nc{\vnabla}{{\vec\nabla}}
\nc{\G}{\Gamma}
\nc{\w}{\omega}
\nc{\mh}{h}
\nc{\mg}{g}
\nc{\vphi}{\varphi}
\nc{\tlambda}{\tilde\lambda}
\nc{\be}{\begin{equation}}
	\nc{\ee}{\end{equation}}
\nc{\ba}{\begin{eqnarray}}
	\nc{\ea}{\end{eqnarray}}

\nc{\g}{\gamma}
\nc{\ol}{\overline}

\newtheorem{theorem}{Theorem}[section]
\newtheorem{lemma}[theorem]{Lemma}
\newtheorem{prop}[theorem]{Proposition}
\newtheorem{corollary}[theorem]{Corollary}
\newtheorem{defin}[theorem]{Definition}
\newtheorem{rmk}[theorem]{Remark}

\nc{\pT}{\partial_T}
\nc{\pz}{\partial_z}
\nc{\pt}{\partial_t}
\nc{\la}{\langle}
\nc{\ra}{\rangle}
\nc{\infint}{\int_{-\infty}^{\infty}}
\nc{\halfwidth}{6.5cm}
\nc{\figwidth}{10cm}
\newcommand{\f}{\frac}

\nc{\nlayers}{L} \nc{\nsectors}{M}
\nc{\indicator}{\mathbf{1}}
\nc{\Rhole}{R_{\rm hole}}
\nc{\Rring}{R_{\rm ring}}
\nc{\neff}{n_{\rm eff}}
\nc{\Frem}{F_{\rm rem}}
\nc{\R}{\mathbb R}
\nc{\C}{\mathbb C}
\nc{\Z}{\mathbb Z}
\nc{\DD}{\Delta}
\nc{\cD}{\mathcal D}
\nc{\cR}{\mathcal R}
\nc{\lnorm}{\left\|}
\nc{\rnorm}{\right\|}
\nc{\rnormp}{\right\|_{\ell^{p,\eps}}}
\nc{\rar}{\rightarrow}
\newcommand{\ang}[1]{\left\langle#1\right\rangle}

\usepackage{xcolor}

\begin{document}

	\begin{abstract}
		
		We investigate $L^1\to L^\infty$ dispersive estimates for the Dirac equation with a potential in four spatial dimensions.  We classify the structure of the obstructions at the thresholds as being composed of an at most two dimensional space of resonances per threshold,  and finitely many eigenfunctions.  Similar to the Schr\"odinger evolution, we prove the natural $t^{-2}$ decay rate when the thresholds are regular.  When there is a threshold resonance or eigenvalue, we show that there is a time dependent, finite rank operator satisfying $\|F_t\|_{L^1\to L^\infty}\les (\log t)^{-1}$ for $t>2$ such that
		$$
		\|e^{it\mathcal H}P(\mathcal H)-F_t\|_{L^1\to L^\infty}\les t^{-1} \quad \text{for } t>2,
		$$
		with $P$ a projection onto a subspace of the absolutely continuous spectrum in a small neighborhood of the thresholds.
		We further show that the operator $F_t=0$ if there is a threshold eigenvalue but no threshold resonance.  We pair this with high energy bounds for the evolution and provide a complete description of the dispersive bounds.  
	\end{abstract}

	\title[Dispersive estimates for Dirac Operators]{\textit{Dispersive estimates for Dirac Operators in dimension four with 
			obstructions at threshold energies}}
	
	\author[W.~R. Green, C.~Lane, B.~Lyons]{William~R. Green, Connor Lane, Benjamin Lyons}
	\thanks{The first author was partially supported by Simons Foundation Grant 511825.}
	\address{Department of Mathematics\\
		Rose-Hulman Institute of Technology \\
		Terre Haute, IN 47803, U.S.A.}
	\email{green@rose-hulman.edu, lanecf@rose-hulman.edu, lyonsba1@rose-hulman.edu}

	\maketitle

	\section{Introduction}
	
	We consider the linear Dirac equation with a potential in four spatial dimensions:
	\begin{align*}
		i\partial_t \psi(x, t) = (D_m + V(x))\psi(x, t), \qquad
		\psi(x, 0) = \psi_0(x).
	\end{align*}
	Here, $x \in \mathbb R^4$ is the spatial variable, $\psi(x,t) \in \mathbb C^{4}$ and the potential $V$ is a $4\times 4$ matrix-valued function. The free Dirac operator
	$D_m$ is defined by
	\begin{align*}
		D_m = -i \alpha \cdot \nabla + m \beta = -i \sum_{k = 1}^{4} \alpha_k \partial_{x_k} + m \beta,
	\end{align*}
	where $m\geq0$ is a constant, and the $4\times 4$ Hermitian matrices $\alpha_0 := \beta$ and $\alpha_j$ satisfy
	\begin{align}\label{eqn:anticomm}
		\alpha_j \alpha_k + \alpha_k\alpha_j 
		= 2\delta_{jk} \mathbb I_{\mathbb C^{4}}, 
		\qquad 
		\text{for all } j, k \in \{0, 1, 2, 3, 4\}.
	\end{align}
	We consider the massive case, when $m > 0$.  The Dirac equation can be thought of as a relativistic correction to the Schr\"odinger equation to account for relativistic dynamics of quantum particles.  A more detailed introduction to the Dirac equation may be found in Thaller's detailed monograph, \cite{Thaller}.  For a briefer introduction, see \cite{EGTwhat}.  For concreteness, in four dimensions we use the following matrices
	\begin{align*}
		\beta = \left[\begin{array}{cc} I_{\mathbb{C}^2} & 0\\ 0 & -I_{\mathbb{C}^2}
		\end{array}\right], \text{ for } 1\leq i\leq 3, \, \alpha_i=\left[\begin{array}{cc} 0 & \sigma_i \\ \sigma_i & 0
		\end{array}\right],  \text{ where }
	\end{align*}
	\begin{align*}
		\sigma_1=\left[\begin{array}{cc} 0 & -i\\ i & 0
		\end{array}\right],\
		\sigma_2=\left[\begin{array}{cc} 0 & 1 \\ 1 & 0
		\end{array}\right],\
		\sigma_3=\left[\begin{array}{cc} 1 & 0 \\ 0 & -1
		\end{array}\right], \, \text{ and } \alpha_4=i\left[ \begin{array}{cc} 0 &   I_{\mathbb{C}^2}\\ - I_{\mathbb{C}^2} & 0\end{array}\right].
	\end{align*}

	The following identity,\footnote{When discussing  scalar operators such as $-\Delta+m^2-\lambda^2$ in the context of the Dirac equation they are to be understood as matrix-valued, that is as $(-\Delta+m^2-\lambda^2)\mathbbm 1_{\mathbb C^{4}}$, a diagonal system of operators.} which follows from  the anti-commutations relations  \eqref{eqn:anticomm},
	\be  \label{dirac_schro_free}
	(D_m-\lambda \mathbbm 1)(D_m+\lambda \mathbbm 1) =(-i \alpha \cdot \nabla +m\beta -\lambda \mathbbm 1)
	(-i\alpha\cdot \nabla+m\beta+\lambda \mathbbm 1)   =(-\Delta+m^2-\lambda^2) 
	\ee
	allows us to formally define the free Dirac resolvent
	operator $\mathcal R_0(\lambda)=(D_m-\lambda)^{-1}$ in terms of the
	free resolvent $R_0(\lambda)=(-\Delta-\lambda)^{-1}$ of  the Schr\"{o}dinger operator for $\lambda$ in the resolvent set:
	\begin{align}\label{eqn:resolvdef}
		\mathcal R_0(\lambda)=(D_m+\lambda) R_0(\lambda^2-m^2).
	\end{align}
	This relates the more complicated Dirac resolvent to the more well-studied Schr\"odinger resolvent.
	
	Throughout the paper we utilize the following notation:  $a-:a-\epsilon$ for an arbitrarily small, but fixed $\epsilon>0$.  Similarly, $a+:=a+\epsilon$.  When there is an absolute constant $C>0$ with $A\leq CB$, we write $A\les B$.  We denote $\la x\ra:=(1+|x|^2)^{1/2}$, and for matrix-valued functions we write $|V(x)|\les \la x\ra^{-\delta}$ to mean that $|V_{ij}(x)|\les \la x\ra^{-\delta}$ for every entry.  Similarly, any function space $X$ used in the paper denotes the space of $\mathbb C^4$-valued functions with all entries in $X$.  That is, by $f\in L^p(\R^4)$ we mean $f(x)=(f_j(x))_{j=1}^4$ with each component an $L^p$ function. We call the threshold energies $\pm m$ regular if there are no resonances or eigenfunctions of the Dirac operator $\mathcal H \coloneqq D_m + V$, which we characterize in terms of non-trivial distributional solutions to $\mathcal H\psi=\pm m\psi$ in Section~\ref{sec:esa} below.
	Our main results control the evolution of the Dirac operator.  Without loss of generality, we focus on the positive threshold at energy $m$.
	Define a smooth cut-off function $\chi(y)$
	with $\chi(y)=1$ if $|y-m|<z_1/2$ and
	$\chi(y)=0$ if $|y-m|>z_1$, for a
	sufficiently small $0<z_1\ll 1$.  We define $P_{ac}(\mathcal H)$ to be the projection onto the absolutely continuous subspace of $L^2$ associated to $\mathcal H$.
	Our main technical result are the following low energy bounds.  
	\begin{theorem}\label{thm:main}
		
		Assume that $V$ is self-adjoint satisfying $|V(x)|\les \la x\ra ^{-\delta}$, and that the thresholds are regular points of the spectrum of
		$\mathcal H$ on $\R^4$.  Then if $\delta>5$, we have 
		$$
		\lnorm  e^{it\mathcal H}\chi(\mathcal H)P_{ac}(\mathcal H)  \rnorm_{L^1 \to L^\infty} \les \la t\ra^{-2}.
		$$
		If either threshold is not a regular point of the spectrum of
		$\mathcal H$ on $\R^4$, then there exists a time dependent, finite rank operator $F_t$  satisfying
		$\lnorm F_t\rnorm_{L^1\to L^\infty} \les 1/\log t$ such that, for $t>2$,
		$$
		\lnorm  e^{it\mathcal H}\chi(\mathcal H)P_{ac}(\mathcal H) - F_t \rnorm_{L^1 \to L^\infty} \les t^{-1}.
		$$
		\begin{enumerate}[i)]
			\item If there is a resonance at zero but no eigenvalue, $F_t$ is rank at most four provided
			$\delta>4$.
			\item If there is an eigenvalue at zero but no resonance, then $F_t=0$ provided
			$\delta>8$.
			\item If there is an eigenvalue and a resonance at zero, $F_t$ is finite rank provided 		$\delta>8$.
		\end{enumerate}
		
	\end{theorem}
	The operator $F_t$ may be explicitly constructed as in \eqref{eqn:Ft def} below at the threshold $m$.  We note that $F_t$ has rank at most two per threshold; at energies $\lambda=m$ and $\lambda=-m$.  We note that these low energy results mirror those known for the Schr\"odinger operator $-\Delta+V$ in four dimensions, \cite{EGG4}.  Heuristically, this is due to \eqref{eqn:resolvdef} which relates the Dirac and Schr\"odinger resolvent operators.

	The class of self-adjoint, decaying potentials we consider here are of the form that naturally arise when linearizing about a standing wave solution to a nonlinear Dirac equation, we refer the reader to the Boussa\"id and Comech's monograph \cite{BCbook} for a detailed treatment.
	As usual, the dispersive bounds follow by studying the solution operator as an element of the functional calculus.  Under the assumptions of Theorem~\ref{thm:main}, $\mathcal H=D_0+V$ is self-adjoint but unbounded both above and below.  Namely, one has the Weyl criterion, $\sigma_{c}(\mathcal H)=\sigma_c(D_m)=(-\infty,m]\cup[m,\infty)$. Further, there is no singularly continuous spectrum or embedded eigenvalues, \cite{GM,BC}.  The spectral measure may be constructed in terms of the limiting resolvent operators
	$$
	\mR_V^\pm(\lambda)=\lim_{\epsilon\to 0^+} [D_m+V-(\lambda\pm i\epsilon)]^{-1}.
	$$
	Agmon's limiting absorption principle ensures that these limiting operators exist as operators on weighted $L^2$ spaces, \cite{agmon}, and their relationship to the Schr\"{o}dinger resolvents \eqref{eqn:resolvdef}.  As usual, the Stone's formula allows us to express the evolution of the solution operator as
	\begin{align}\label{eqn:stone}
		e^{-it\mathcal H}P_{ac}(\mathcal H)f=\frac{1}{2\pi i} \int_{\R} e^{-it\lambda} [\mR_V^+-\mR_V^-](\lambda) f\, d\lambda.
	\end{align}
	We study the spectral measure by creating expansions of the perturbed Dirac resolvent operators $\mR_V^\pm(\lambda)$ as perturbations of the free Dirac resolvents
	$$
	\mR_0^\pm(\lambda) = \lim_{\epsilon\to 0^+} [D_m-(\lambda\pm i\epsilon)]^{-1}.
	$$
	Using \eqref{eqn:resolvdef}, we obtain the identity $\mR_0^\pm (\lambda) =(D_m+\lambda)R_0^\pm(\lambda^2-m^2)$ where $R_0^\pm(\lambda)=(-\Delta-(\lambda\pm i0))^{-1}$ are the limiting resolvent operators for the Schr\"{o}dinger operator.  As is expected, the case of four spatial dimensions has more technical challenges as the free resolvent operators cannot be expressed in a closed form as in three dimensions.  Further, there are logarithmic terms that arise in the expansions as the spectral parameter approaches the thresholds.  We overcome these technical challenges by carefully adapting the four dimensional Schr\"odinger resolvents in \cite{EGG4} to the richer Dirac structure.  We note that our analysis uses and builds on techniques developed in the study of other dispersive equations, notably the Schr\"odinger equation \cite{Mur,Jen,Jen2,GS,ES,KS,eg2,EGG4}. 
	
	Global dispersive bounds that control the $L^\infty(\R^n)$ size of solutions are well-studied in the Schr\"odinger, wave, and Klein-Gordon contexts.  Rather than try to give a full history of these related problem, we refer the interested reader to Schlag's recent survey paper \cite{Scsurvey} and the references therein.  The literature on global dispersive estimates for the Dirac equation has developed in the last two decades; the first works were done by Boussa\"id \cite{Bouss1}, and D'Ancona and Fanelli, \cite{DF} in the massive $m>0$ in the weighted-$L^2$ setting when $n=3$.  
	
	The effect of threshold obstructions on the global dispersive bounds, as well as classifying resonances and eigenvalues in terms of distributional solutions to $\mathcal H\psi=\pm m\psi$ where the exact properties of $\psi$ depend on the spatial dimension, $n$, is more recent. Global estimates for the Dirac evolution in dimensions $n=2,3$ have been established by  Erdo\smash{\u{g}}an and the first author in \cite{egd}, along with Toprak, \cite{EGT,EGT2} in the massive case, and with Goldberg in the massless case \cite{EGG2}.  The case of $n=1$ was studied by Erdo\smash{\u{g}}an and the first author in \cite{egd1}.    More recently the authors, with Ravishankar and Shaw in \cite{GLLSS}, considered the massless case when $n=3$. Danesi, \cite{Dan}, and separately Cacciafesta, S\'er\'e and Zhang, \cite{CSZ}  consider related estimates in the massless case when $n=3$ with more singular potentials.
	Recent work of Kraisler, Sagiv and Weinstein, \cite{KSW}, considered non-compact perturbations, and time-periodic perturbations, \cite{KSW2}, of the Dirac operator in one dimension.

	Nonlinear Dirac equations are of independent interest, we again refer the reader to  Boussa\"id and Comech's excellent monograph, \cite{BCbook}, and \cite{EV,PS,BH3,BH,CTS,BC2,BF,Bor} for example.  There is a longer history in the study of spectral properties of Dirac operators, \cite{BG1,GM,MY,CGetal,BC} and limiting absorption principles \cite{Yam,GM,BG,EGG,CGetal}.
	
	For completeness, we pair the bounds of Theorem~\ref{thm:main} with an argument that controls the high energy portion of the evolution.   
	\begin{theorem}\label{thm:full results}
		
		Assume that $V$ is self-adjoint with continuous entries satisfying $|V(x)|\les \la x\ra^{-\delta}$ for some $\delta>5$.   
		If the thresholds are regular points of the spectrum we have 
		$$
		\lnorm  e^{it\mathcal H}P_{ac}(\mathcal H)\la \mathcal H\ra^{-\frac{11}{2}-}  \rnorm_{L^1 \to L^\infty} \les \la t\ra^{-2}.
		$$
		If either threshold is not a regular point of the spectrum of
		$\mathcal H$ on $\R^4$, then there exists a time dependent, finite rank operator $F_t$  satisfying
		$\lnorm F_t\rnorm_{L^1\to L^\infty} \les 1/\log t$ such that, for $t>2$,
		$$
		\lnorm  e^{it\mathcal H}P_{ac}(\mathcal H)\la \mathcal H\ra^{-\frac{11}{2}-} - F_t \rnorm_{L^1 \to L^\infty} \les t^{-1}.
		$$
		\begin{enumerate}[i)]
			\item If there is a resonance at zero but no eigenvalue, $F_t$ is rank at most four provided
			$\delta>5$.
			\item If there is an eigenvalue at zero but no resonance, then $F_t=0$ provided
			$\delta>8$.
			\item If there is an eigenvalue and a resonance at zero, $F_t$ is finite rank provided 		$\delta>8$.
		\end{enumerate}

	\end{theorem}
	The negative powers of $\la \mathcal H\ra$ can be interpreted as the need for smoothness on the initial data.  The amount of smoothness required here is not optimal; due to the relationship between the massive Dirac and Klein-Gordon equation one expects to require at least $\frac{n+3}{2}$ derivatives of the initial data should be in $L^1$ to obtain a $|t|^{-\frac{n-1}{2}}$ decay rate.  In $n=4$ spatial dimensions one has the added complication that in general smoothness on the potential is required to obtain dispersive bounds even for the Schr\"odinger equation due to the counterexamples constructed by Goldberg and Visan, \cite{GV}.   We prove high energy estimates, see Proposition~\ref{prop:high energy} below, by restricting to dyadic frequency intervals to utilize the non-zero curvature of the phase to provide a lower bound for the Hessian of the phase when localized to these dyadic intervals.
	
	We define the complementary cut-off $\widetilde \chi(z)=1-\chi(z)$ to be a smooth, even cut-off away from the threshold energies.  The effect of threshold obstructions only occurs in the low energy evolution.
	In light of Theorem~\ref{thm:main}, we need only prove high energy bounds for the perturbed Dirac evolution.   By decomposing the spectrum into dyadic intervals when $z\approx 2^j$, then summing the bounds in Proposition~\ref{prop:high energy} in $j$, we obtain high energy bounds of the form: 
	\begin{align*}
		\|e^{-it\mathcal H} P_{ac}(\mathcal H)\widetilde \chi(\mathcal H) \la \mathcal H\ra^{-\frac{11}{2}-}\|_{L^{1}\to L^{\infty}} &\les \la t\ra^{-2},\\
		\|e^{-it\mathcal H} P_{ac}(\mathcal H)\widetilde \chi(\mathcal H) \la \mathcal H\ra^{-4-}\|_{L^{1}\to L^{\infty}} &\les 1.
	\end{align*}
	Here we note that, as in the massless three-dimensional case \cite{GLLSS}, we can adapt our expansions to prove boundedness of the solution operator with weaker, nearly optimal decay and smoothness assumptions.
	\begin{theorem}\label{thm:min bd}
		
		Assume that $V$ is self-adjoint with continuous entries satisfying $|V(x)|\les \la x\ra^{-\delta}$ for some $\delta>\f52$.   
		If the thresholds are regular points, then
		$$
		\|e^{-it\mathcal H} P_{ac}(\mathcal H) \la \mathcal H\ra^{-4-}\|_{L^{1}\to L^{\infty}} \les 1.
		$$	
	\end{theorem}
	The assumptions here are nearly optimal with respect to differentiability of the initial data and decay of $V$ at infinity.  In Lemma~\ref{lem:free hi}, one sees that in the case of $V=0$, one need $\la D_m\ra^{-4-}$ for the free operator to be bounded for all $t$.  In light of the limiting absorption principle, \cite{EGG}, and scaling considerations, one expects $\delta>2$ to be optimal for high energy.  
	
	We note that, as in previous analyses of the Dirac equation \cite{egd,EGT,egd1,GLLSS}, one can obtain Strichartz estimates as a corollary of the dispersive bounds via a $T^*T$ argument as in \cite{GV}.  Furthermore, as in Remark~4.7 of \cite{egd}, our resolvent expansions for $\mR_V^\pm(\lambda)$ in the regular case imply a limiting absorption principle in a neighborhood of the threshold and consequently that there are only finitely many eigenvalues in the spectral gap $[-m,m]$.
	
	The paper is organized as follows.  First in Section~\ref{sec:reg}  we develop expansions of the limiting free resolvent operators in a neighborhood of the threshold.  We further use these expansions to derive expansions of the spectral measure and its derivatives near the threshold when they are regular.  In Section~\ref{sec:reg dis} we use these expansions in the Stone's formula, \eqref{eqn:stone}, to reduce to oscillatory integral bound that suffice to prove Theorem~\ref{thm:main} in the case when the thresholds are regular.   In Section~\ref{sec:non reg} we develop expansions of $\mR_V^\pm$ in a neighborhood of the threshold when the threshold is not regular, from which we prove the claims in Theorem~\ref{thm:main} when at least one threshold is not regular.  In Section~\ref{sec:esa} we characterize the existence of threshold obstructions as resonances and eigenvalues while also relating the obstructions to the spectral measure constructed in Sections~\ref{sec:non reg}.  Finally in Section~\ref{sec:hi} we control the high energy portion of the evolution.  We further establish Theorem~\ref{thm:min bd} here as the argument for this bound on low energy is inspired by the selective iteration scheme we use in the high energy regime.

	\section{Resolvent Expansions}\label{sec:reg}
	In light of the Stone's formula, \eqref{eqn:stone}, our analysis relies on detailed expansions of the limiting Dirac resolvent operators.  To construct and analyze the differentiability properties of the spectral measure $[\mR_V^+-\mR_V^-](\lambda)$ in a sufficiently small neighborhood of the threshold(s) at $\lambda =\pm m$, we construct the perturbed resolvents $\mR_V^\pm(\lambda)$ as a perturbation from the free Dirac resolvents $\mR_0^\pm$.  We consider the positive branch of the spectrum $[m,\infty)$.  The negative branch $(-\infty,-m]$ may be analyzed similarly with fairly straightforward modifications, see Remark~\ref{rmk:neg branch} below.
	We first consider the free Dirac resolvent operators $\mR_0^\pm(\lambda)$ in a neighborhood of the threshold to establish the needed expansions to prove the low-energy bounds time decay bounds of Theorem~\ref{thm:main} when the thresholds are regular.
	
	The identities \eqref{dirac_schro_free} and \eqref{eqn:resolvdef}, lead us to the change of variables $\lambda=\sqrt{z^2+m^2}$ to rewrite the limiting Dirac resolvent operators
	\begin{align}\label{eq:Dmpm}
		\mR_0^\pm(z):=(D_m+\sqrt{z^2+m^2})R_0^\pm(z^2).
	\end{align}
	This is a slight abuse of notation, but greatly clarifies the exposition below.  This allows us to utilize the more well-studied Schr\"odinger resolvent operators to develop the needed expansions for the Dirac resolvents.  Using the expansions of the Schr\"odinger resolvent in four dimensions from \cite{EGG4}, we have
	$$
	R_0^\pm (z^2)=G_0+g_1^\pm(z)+z^2 G_1+g_2^\pm (z)G_2+\dots
	$$
	where
	\begin{align}
		G_0f(x)&=-\frac{1}{4\pi^2}\int_{\R^4} 
		\frac{f(y)}{|x-y|^2}\,dy=(-\Delta)^{-1} f(x) , \label{G0 def}\\
		G_1f(x)&=-\frac{1}{8\pi^2}\int_{\R^4} \log(|x-y|) f(y)\, dy,\label{G1 def}\\ 
		\label{G2 def} G_2f(x)&=c_2 \int_{\R^4} |x-y|^2 f(y)
		\, dy\\
		\label{G3 def} G_3f(x)&=c_3 \int_{\R^4}|x-y|^2
		\log(|x-y|)f(y)\, dy
	\end{align}
	Here $c_2,c_3$ are certain real-valued constants, the
	exact values are unimportant for our analysis.  These were obtained using the relationship between the Schr\"odinger resolvent and the Bessel functions of order one,
	\begin{align}\label{eqn:res bessel}
		R_0^\pm(z^2)=\pm \frac{iz}{8\pi}[J_1(z|x-y|)\pm iY_1(z|x-y|)].
	\end{align}
	We will use $G_j(x,y)$ to denote the integral kernel of the operator $G_j$.  In 
	addition, the following functions appear  naturally,
	\begin{align}
		\label{g1 def}
		g^+_1(z)=\overline{g_1^-(z)}&=z^2(a_1\log(z)+b_1)\\
		\label{g2 def}
		g^+_2(z)=\overline{g_2^-(z)}&=z^4(a_2\log(z)+b_2).
	\end{align}
	Here  $a_j\in \R\setminus\{0\}$ and $b_j\in \mathbb C
	\setminus \R$.  Then, writing $\lambda =\sqrt{z^2+m^2}=m+(\sqrt{z^2+m^2}-m)$ and
	\begin{align*}
		\mR_0(z)&=(D_m+mI)R_0(z^2)+(\sqrt{z^2+m^2}-m)R_0(z^2)\\
		&=(-i\alpha \cdot \nabla +2mI_{uc})R_0(z^2)+\frac{z^2}{2m}R_0(z^2)+ O_4(z^4)R_0(z^2).
	\end{align*}
	Here we use that $\sqrt{z^2+m^2}=m+\frac{z^2}{2m}+ O_4(z^4)$ where $ O_4$ indicates a remainder that has four derivatives where differentiation in $z$ is bounded by division by $z$, i.e. $\partial_z  O_4(z^k)= O_3(z^{k-1})$. 
	To account for the matrix-valued structure of the Dirac operator, we define the following matrices and related operators.  We define the matrices
	\begin{align}
		I_{uc}=\frac12(\beta+I_{4\times 4})=\begin{pmatrix} 1 & 0 & 0 &0 \\ 0& 1 &0 &0\\ 0 & 0 & 0 &0\\ 0 & 0 & 0 &0\end{pmatrix}, \qquad I_{lc}=\frac12(I_{4\times 4}-\beta)=\begin{pmatrix} 0 & 0 & 0 &0 \\ 0& 0 &0 &0\\ 0 & 0 & 1 &0\\ 0 & 0 & 0 &1\end{pmatrix}.
	\end{align}
	We define the operator $M_{uc}$ and $M_{lc}$ as the operators with integral kernel $I_{uc}$ and $I_{lc}$ respectively, so that
	$$
	M_{uc}f(x)= \int_{\R^4} I_{uc}f(y)\, dy=(c_1,c_2,0,0)^T, \quad \text{and} \quad M_{lc}f(x)=(0,0,c_3,c_4)^T,
	$$
	where $c_j=\int_{\R^4} f_j(y)\, dy$ where $f_j$ is the entry in the $j^{th}$ component of the $\C^4$-valued function $f$.
	
	We define the operators $\mG_j$ by their integral kernels
	\begin{align}
		\mathcal{G}_0 (x,y) &= (D_m+mI)G_0(x,y)= [ -i \alpha \cdot \nabla + 2m I_{uc} ] G_0(x,y)=\frac{i\alpha \cdot (x-y)}{2\pi^2|x-y|^4} - \frac{mI_{uc}}{2\pi^2 |x-y|^2},
		\label{eq:G0 def}  \\
		\mathcal{G}_1(x,y)& = \frac{  m}{2 \pi^2}M_{uc}(x,y), \label{eq:G1 def} \\
		\mathcal{G}_2(x,y)&=  [ - i \alpha \cdot \nabla +2m I_{uc} ] G_1(x,y) - \frac{1}{2m} G_0 (x,y)\label{eq:G2 def}, \\
		\mG_3(x,y)&=[ - i \alpha \cdot \nabla +2m I_{uc} ] G_2(x,y).
	\end{align}
	Here the scalar-valued $G_0(x,y)$ in \eqref{eq:G1 def} is understood to be a $G_0(x,y)\mathbbm 1_{\mathbb C^{4}}$
	So that, when $z|x-y|\les 1$ we have the expansions
	$$
	\mR_0^\pm(z)(x,y)=\mG_0(x,y)+g_1^\pm(z)\mG_1(x,y)+z^2\mG_0+g_2^\pm(z)\mG_3(x,y)+\dots
	$$
	These expansions can be extended in terms of larger powers of $z$ and $g_j^{\pm}(z)=z^{2j}(a_j\log \lambda +b_j)$ when $z|x-y|\les 1$.  To clarify that these expansions are valid when $z|x-y|\les 1$, we write $\mR_L^\pm(z)(x,y):=\mR_0^\pm(z)(x,y)\chi(z|x-y|)$.  Then, we have
	\begin{align}\label{eqn:RL low def}
		\mR_L^\pm(z)(x,y)=\mG_0(x,y)+g_1^\pm(z)\mG_1(x,y)+z^2 \mG_2(x,y)+ O_2(z^{4-}(|x-y|^2+|x-y|^{-2})).
	\end{align}
	In addition, we note that $g_1^+(z)\mG_1+z^2 \mG_2=\overline{g_1^-(z)\mG_1+z^2 \mG_2}=[-i\alpha \cdot \nabla+2mI_{uc}] z^2(a_1\log(z|x-y|)+b_1)-\frac{z^2}{2m}G_0(x,y)$ so that (for some constants $c_1,c_2\neq 0$)
	\begin{align*}
		\partial_z\mR_L^\pm(z)(x,y)&=[2z(a_1\log(z|x-y|)+b_1)+a_1z]I_{uc}+z\bigg[
		\frac{c_1I}{|x-y|^2}+c_2 \frac{\alpha\cdot (x-y)}{|x-y|}\bigg]\\
		&\qquad\qquad+ O_1(z^{3-}(|x-y|^2+|x-y|^{-2})),\\
		\partial_z^2\mR_L^\pm(z)(x,y)&=[2(a_1\log(z|x-y|)+b_1) +3a_1]I_{uc}+\bigg[
		\frac{c_1I}{|x-y|^2}+c_2 \frac{\alpha\cdot (x-y)}{|x-y|}\bigg]\\
		&\qquad\qquad+O(z^{2-}(|x-y|^2+|x-y|^{-2})).
	\end{align*}
	In particular, we have
	\begin{align}\label{eqn:RL low bds}
		|\mR_L^\pm(z)(x,y)|&\les \frac{1}{|x-y|^3}+\frac{1}{|x-y|^2}, \\
		|\partial_z^k \mR_L^\pm(z)(x,y)|&\les z^{2-k}(1+\log(z|x-y|)+|x-y|^{-2}), \quad k=1,2.\nn
	\end{align}
	Furthermore,
	$$
	(\mR_L^+-\mR_L^-)(z)(x,y)=2\Im(g_1^+(z))\mG_1+ O_2(z^4|x-y|^2).
	$$
	This follows from the relationship $\mR_0^+-\mR_0^+=(D_m+\sqrt{z^2+m^2})[R_0^+-R_0^-]$, along with the identity
	\eqref{eqn:res bessel}.  In the difference of + and - resolvents, only the contribution of $J_1$ remains and (c.f \cite{AS}) has a power series expansion near zero.
	In particular $2\Im(g_1^+(z))=cz^2$ for some constant $c$ and for $k=0,1,2$,
	\begin{align}\label{eqn:RL diff bds}
		|\partial_z^k\big[ (\mR_L^+-\mR_L^-)(z)(x,y)-cz^2\mG_1(x,y) \big] |\les z^{4-k}|x-y|^2.
	\end{align}

	When $z|x-y|\gtrsim 1$ we use the asymptotic expansions of the Bessel functions, \cite{AS}, and \eqref{eqn:res bessel} to see 
	\begin{align}
		\mR_H^\pm(z)(x,y)&:=\mR_0^\pm (z)(x,y)\widetilde{\chi}(z|x-y|)=(D_m+\sqrt{z^2+m^2}I)\bigg( \frac{\pm iz}{8\pi |x-y|}e^{\pm iz|x-y|}\omega_{\pm}(z|x-y|)\bigg)\nn \\
		&= e^{\pm i z|x-y|} \widetilde \omega_{\pm}(z|x-y|), \qquad |\partial_z^k  \widetilde\omega_{\pm}(z|x-y|)|\les \frac{z^{\frac12-k}}{r^{\frac32}}.\label{eqn:RH defn}
	\end{align}
	Combining these we have
	\begin{lemma}\label{lem:R0exp}
		
		We have the following expansions of the Dirac resolvent when $0<z<1$:
		\begin{align*}
			\mR_0^\pm (z)(x,y)&=\mG_0(x,y)+\mathcal E_0(z)(x,y),\\
			&=\mG_0(x,y)+g_1^\pm(z)\mG_1(x,y)+z^2\mG_2(x,y)+\mathcal E_1(z)(x,y),\\
			&=\mG_0(x,y)+g_1^\pm(z)\mG_1(x,y)+z^2\mG_2(x,y)+g_2^\pm(z)\mG_3(x,y),\\
			&\qquad\qquad		+\frac{z^2g_1^\pm(z)}{2m}G_1(x,y) +z^4\mG_4(x,y)			+\mathcal E_2(z)(x,y),
		\end{align*}
		where
		\begin{align*}
			\sup_{0<z<z_1} \big[z^{-2+}|\mathcal E_0(z)|+z^{-1+}|\partial_z \mathcal E_0(z)|\big] \les \frac{1}{|x-y|^2}+1,\\
			\sup_{0<z<z_1} \big[z^{-2+}|\mathcal E_0(z)|+z^{-1+}|\partial_z \mathcal E_0(z)|+z^{0+}|\partial_z^2 \mathcal E_0(z)|\big] \les \frac{1}{|x-y|^2}+ |x-y|^{\frac12},\\
			\sup_{0<z<z_1} \big[z^{-2-}|\mathcal E_1(z)|+z^{-1-}|\partial_z \mathcal E_1(z)|
			\big] \les \frac{1}{|x-y|^2}+1,\\
			\sup_{0<z<z_1} \big[z^{-4+}|\mathcal E_2(z)|+z^{-3+}|\partial_z \mathcal E_2(z)|\big] \les |x-y|^2+\frac{1}{|x-y|^2}.
		\end{align*}
		
	\end{lemma}
	We recall the following definition, which is standard in the analysis of dispersive estimates.  An operator $K:L^2(\R^4)\to L^2(\R^4)$ with kernel
	$K(\cdot,\cdot)$ is absolutely bounded if the operator with kernel
	$|K(\cdot,\cdot)|$ is bounded from $L^2(\R^4)$ to $L^2(\R^4)$.
	In particular, Hilbert-Schmidt and
	finite rank operators are absolutely bounded.
	
	To obtain expansions for the perturbed resolvent operators $\mathcal{R}^{\pm}_V(z)$,  we recall the symmetric resolvent identity. As in previous analyses of the Dirac equation, \cite{egd,EGT,EGT2,EGG2,GLLSS}, the assumption that $ V \colon \R^4 \to \C^{4\times 4}$ is self-adjoint allows us to use the symmetric resolvent identity for $\mR_V^\pm$.  First, by the spectral theorem for self-adjoint matrices we may write 
	\begin{gather*}
		V= B^{*} \Lambda B= B^{*}  |\Lambda|^{\f12} U |\Lambda|^{\f12} B =: v^{*} U v,\,\,\,\text{ where} \\
		\Lambda=\text{diag}(\lambda_1,\lambda_2,\lambda_3,\lambda_4), \text{ with }\lambda_j \in \R, \\
		|\Lambda|^{\f12}= \text{diag}(|\lambda_1|^{\f12},|\lambda_2|^{\f12},|\lambda_3|^{\f12},|\lambda_4|^{\f12}),\\
		U= \text{diag}(\text{sign}(\lambda_1), \text{sign}(\lambda_2),\text{sign}(\lambda_3),\text{sign}(\lambda_4)).
	\end{gather*}
	We note that if the entries of $V(x)$ are all bounded by $\la x\ra^{-\delta}$, then the entries of $v(x)$ and $v^*(x)$ are bounded by $\la x\ra^{-\delta/2}$.
	Defining $ M^\pm(z) = U + v\mathcal{R}^\pm_0(z) v^*$,  the symmetric resolvent identity yields
	\begin{align}\label{eq:pert1}
		\mathcal{R}^{\pm}_V(z) =\mathcal{R}^{\pm}_0(z) - \mathcal{R}^{\pm}_0(z) v^{*} M_{\pm}^{-1} (z) v \mathcal{R}^{\pm}_0(z). 
	\end{align} 
	For notational convenience, we write $M_\pm^{-1}(z)$ in place of $(M^\pm)^{-1}(z)$.
	We consider $M^\pm(z)$ as a perturbation of $M^\pm(0) = U + v\mR_0^\pm(0)v^* = U + v\mG_0v^*$.  We denote this operator by $T_0 \coloneqq M^\pm(0)$, and define $T_j=v\mG_jv^*$ for $j\geq 1$.
	The expansions for $\mR_0^\pm$ in Lemma~\ref{lem:R0exp} immediately yield  the expansions below.
	\begin{lemma}\label{lem:M exp}
		If $|V(x)|\les \la x\ra^{-\delta}$ for some $\delta>0$ then we have the expansions when $0<z<1$.
		\begin{align*}
			M^\pm (z)&=T_0+\mathcal M_0^\pm(z)\\
			&=T_0 +g_1^\pm(z)T_1+z^2T_2+\mathcal M_1^\pm(z)\\
			&=T_0 +g_1^\pm(z)T_1+z^2T_2+g_2^\pm(z)T_3+z^4T_4+\mathcal M_2^\pm(z),
		\end{align*}
		where if $\delta >4$ we have (for a sufficiently small $0<z_1\ll 1$)
		\begin{align*}
			\sup_{0<z<z_1} \big[z^{-2+}|\mathcal M_0^\pm(z)|+z^{-1+}|\partial_z \mathcal M_0^\pm(z)|\big],\\
			\sup_{0<z<z_1} \big[z^{-2-}|\mathcal M_1^\pm(z)|+z^{-1-}|\partial_z \mathcal M_1^\pm(z)|\big] ,
		\end{align*}
		are absolutely bounded operators on $L^2(\R^4)$.  For control over two derivative we require $\delta>5$, in which case we have 
		\begin{align*}
			\sup_{0<z<z_1} \big[z^{-2+}|\mathcal M_0^\pm(z)|+z^{-1+}|\partial_z \mathcal M_0^\pm(z)|+z^{0+}|\partial_z^2 \mathcal M_0^\pm(z)|\big],\\
			\sup_{0<z<z_1} \big[z^{-2-}|\mathcal M_1^\pm(z)|+z^{-1-}|\partial_z \mathcal M_1^\pm(z)|+z^{-0-}|\partial_z^2 \mathcal M_1^\pm(z)|\big].
		\end{align*}	
		If we further assume $\delta>8$, then
		\begin{align*}
			\sup_{0<z<z_1} \big[z^{-4-}|\mathcal M_2^\pm (z)|+z^{-3-}|\partial_z \mathcal M_2^\pm (z)|\big] 
		\end{align*}
		is an absolutely bounded operator on $L^2(\R^4)$.
		
	\end{lemma}

	The proof follows immediately from the expansions in Lemma~\ref{lem:R0exp} and the decay of $V$.  For the contribution of $|x-y|^{-2}$ we view this as a multiple of  the fractional integral operator $\mathcal I_2(x,y)$ is the integral kernel of the fractional integral operator $\mathcal I_2$.  From Lemma 2.3 in \cite{Jen}, $\mathcal I_2:L^{2,1+}\to L^{2,-1}$, so $v\mathcal I_2v^*$ is an absolutely bounded operator on $L^2$.
	
	We utilize the following notation,  for an absolutely bounded operator $T(z)$ on $L^2(\R^n)$ we write $T(z)=O_N(z^j)$ to mean that
	\be
	\big\|\sup_{0<z<z_1}  \sup_{0\leq \ell \leq N} z^{\ell-j}|\partial_z^\ell T(z)|\,\big\|_{L^2\to L^2}\les 1.
	\ee
	In particular, we note that $\mathcal M_0^\pm(z)=O_2(z^{2-})$ and $\mathcal M_j^\pm(z)=O_1(z^{2j+})$ for $j=1,2$.

	Due to the similarity of the resonance structure of the Dirac operator to that of the Schr\"odinger operator, we utilize the terminology from \cite{EGG4} to characterize threshold resonances and associated subspaces of $L^2(\R^4)$.   We define regularity for the threshold $\lambda=m$ (which corresponds to $z=0$), as in Remark~\ref{rmk:neg branch} this may be adapted to the negative threshold at $\lambda=-m$
	\begin{defin}\label{resondef}
		\begin{enumerate}
			\item We say the threshold $\lambda=m$ is a regular point of the spectrum
			of $\mathcal H = D_m+ V$ provided $T_0$ is invertible on $ L^2(\mathbb R^4)$.
			
			\item Assume that threshold $\lambda=m$ is not a regular point of the spectrum. Let $S_1$ be the Riesz projection
			onto the kernel of $ T_0 $ as an operator on $ L^2(\mathbb R^4)$.
			Then $ T_0 +S_1$ is invertible on $ L^2(\mathbb R^4)$.  Accordingly, we define $D_0=( T_0 +S_1)^{-1}$ as an operator
			on $ L^2(\R^4)$.
			We say there is a resonance of the first kind if the operator $T_1:= S_1 T_1 S_1$ is invertible on
			$S_1L^2(\mathbb R^4)$.
			
			\item  Assume that $T_1$ is not invertible on
			$S_1L^2(\mathbb R^4)$. Let $S_2$ be the Riesz projection onto the kernel of $T_1$ as an operator on $S_1L^2(\R^4)$.   Then $T_1+S_2$ is invertible on
			$S_1 L^2(\R^4)$.
			We say there is a resonance of the second kind   if $S_2=S_1$.
			If $S_1-S_2\neq 0$,  we say there is a resonance of the third kind. 
		\end{enumerate}
	\end{defin}
	
	\noindent
	{\bf Remarks.  i)} We note that $S_1-S_2\neq 0$ corresponds to the existence of a resonance
	at the threshold energy, and $S_2\neq 0$ corresponds to the existence of an
	eigenvalue (see 
	Section~\ref{sec:esa} below).  That is, a resonance of
	the first kind means that there is a resonance at the threshold $\lambda=m$
	only, a resonance of the second kind means that there is
	an eigenvalue only, and a resonance of the third
	kind means that   there is both a resonance and an 
	eigenvalue at the threshold.  We utilize a streamlined technique to allow us to invert
	$M^{\pm}(z)$ for the different types of resonances based on the techniques developed in the massless two dimensional case, \cite{EGG2}. Some care is needed since $S_1-S_2$ can have rank greater than one. 
	\\
	{ \bf ii)} Since $ T_0 $ is self-adjoint, $S_1$ is the orthogonal projection onto the kernel of $ T $, and we have
	(with $D_0=( T_0 +S_1)^{-1}$) $$S_1D_0=D_0S_1=S_1.$$
	A similar statement also valid for $S_2$ and $(T_1+S_2)^{-1}$. \\
	{\bf iii)} Since $T_0$ is a compact perturbation of the invertible operator $U$, the Fredholm alternative
	guarantees that $S_1$ and $S_2$ are finite-rank projections in all cases.
	
	See Section~\ref{sec:esa} below for a
	full characterization of the spectral subspaces of $L^2$ associated
	to $\mathcal H=D_m+V$ at energy $\lambda=m$.

	\begin{lemma}\label{d0bounded}
		The operator $ D_0 $ is absolutely bounded  provided $|V(x)|\les \la x\ra^{-2-}$.
	\end{lemma}
	The proof is standard and similar to previous analysis of inverse operators for both Dirac and Schr\"odinger operators in various dimensions, see \cite{EGG,eg2,EGT} for example.  Dominating the integral kernel of $|v\mG_0v^*|$ pointwise by a constant multiple of $\la x\ra^{-1-}(\mathcal I_1+\mathcal I_2)(x,y)\la y\ra^{-1-}$ shows that it is absolutely bounded on $L^2$ since the fractional integral operators $\mathcal I_j:L^{2,1+}(\R^4)\to L^{2,-1-}(\R^4)$ when $j=1,2$, \cite{Jen}.  By 
	iterating the resolvent identity
	$D_0=U-U(v\mG_0v^*+S_1)D_0$ sufficiently to write $D_0$ as a sum of finite rank and Hilbert-Schmidt operators, from which the claim follows. One can easily show $[Uv\mG_0v^*]^k$ is Hilbert-Schmidt for large enough $k$ using \eqref{eqn:RL low bds} and Lemma~\ref{lem:spatial estimates} below, we leave the details to the reader.

	For the convenience of the reader, when the exact formulation of an operator is not needed we write $\Gamma_j$ to indicate an absolutely bounded operator that is independent of $z$ and the choice of limiting values $+$ or $-$.  The exact operator is allowed to change from line to line, when exact formulations are required, we  note them explicitly.

	In the case when $\lambda=m$ is regular, $T_0$ is an invertible operator on $L^2(\R^4)$. In this case, we have
	\begin{lemma}\label{lem:Minv reg}
		
		If $m$ is a regular point of the spectrum and $|V(x)|\les \la x\ra^{-\delta}$ for some $\delta>5$, then the operators $M^\pm(z)$ are invertible in a sufficiently small neighborhood of $z=0$.  Furthermore we have the expansion
		\be\label{eqn:Minv reg}
		M_\pm^{-1}(z)=D_0+g_1^\pm(z)D_0T_1D_0+z^2D_0T_2D_0+O_2(z^{2+})=O_2(z^0).
		\ee
		If $m$ is not a regular point of the spectrum, then the operators $M^\pm(z)+S_1$ are invertible provided $\delta>4$.  Furthermore, we have the bounds
		$$
		(M^\pm(z)+S_1)^{-1}=D_0+g_1^\pm(z)D_0T_1D_0+z^2D_0T_2D_0+O_1(z^{2+})=O_1(z^0),
		$$
		in a sufficiently small neighborhood of zero. 
		If $\delta>8$, there are are absolutely bounded, $z$-independent operators $\Gamma_{i,j}$ so that
		\begin{multline}\label{eqn:M+S inv long}
			(M^\pm(z)+S_1)^{-1}=D_0+g_1^\pm(z)D_0T_1D_0+z^2 D_0T_2D_0+(g_1^\pm(z))^2\Gamma_{4,2}\\
			+[g_2^\pm(z)+z^2g_1^\pm(z)]\Gamma_{4,1}
			+z^4\Gamma_{4,0}+O_1(z^{4+}).
		\end{multline}

	\end{lemma}
	
	\begin{proof}
		
		The bounds on the error term in Lemma~\ref{lem:M exp} may be rephrased as $\mathcal M_0^\pm(z)=O_2(z^{2-})$, so that
		$M^\pm(z) =T_0+O_2(z^{2-})$ with $T_0$ an invertible operator with absolutely bounded inverse $D_0:=T_0^{-1}$.  A Neumann series expansion immediately shows that $M_\pm^{-1}(z)=D_0+O_2(z^{2-})$.   To obtain a more detailed expansion we apply the resolvent identity $A^{-1}=B^{-1}+B^{-1}(B-A)A^{-1}$  with $A=M^\pm (z)$ and $B=T_0$ twice.  Noting that $B-A=M^\pm(z)-T_0=g_1^\pm(z)T_1+z^2T_2+\mathcal M_1^\pm(z)$,  we have
		\begin{multline*}
			(M^\pm(z)+S_1)^{-1}=D_0-D_0(M^\pm(z)-T_0)D_0
			+[D_0(M^\pm(z)-T_0)]^2(M^\pm(z)+S_1)^{-1}\\
			=D_0+g_1^\pm(z)D_0T_1D_0+z^2D_0T_2D_0+O_2(z^{2+}).
		\end{multline*}	
		One can obtain further powers of $z$ at the cost of requiring larger $\delta$ by noting that $\mathcal M_1^\pm(z)=g_2^\pm(z)T_3+z^2T_4+M_{2}^\pm(z)$ and choosing not to truncate the expansion as soon.  The extra decay requirement from the potential is needed to control the growth in the kernels of $T_3$, $T_4$.
		
		In the case that $T_0$ is not invertible, the operator $T_0+S_1$ is invertible.  With a small abuse of notation we denote $D_0:=(T_0+S_1)^{-1}$, this agrees with the previous definition since $S_1=0$ if $T_0$ is invertible.  The previous expansion is valid and proves the second claim.
		Iterating the resolvent identity three times noting that $B-A=g_1^\pm(z)T_1+z^2T_2+g_2^\pm(z)T_3+z^4T_4+O_1(z^{4+})$ in this case yields the longer expansion provided $\delta>8$.
		\begin{multline*}
			(M^\pm(z)+S_1)^{-1}=D_0-D_0(M^\pm(z)-T_0)D_0
			+[D_0(M^\pm(z)-T_0)]^2D_0\\
			-[D_0(M^\pm(z)-T_0)]^3(M^\pm(z)+S_1)^{-1}\\
			=D_0+g_1^\pm(z)D_0T_1D_0+z^2 D_0T_2D_0+(g_1^\pm(z))^2\Gamma_{4,2}+g_2^\pm(z)\Gamma_{4,1}
			+z^4\Gamma_{4,0}+O_1(z^{4+})
		\end{multline*}
		Where $\Gamma_{4,2}=D_0T_1D_0T_1D_0$, $\Gamma_{4,1}=D_0T_3D_0+D_0T_1D_0T_2D_0+D_0T_2D_0T_1D_0$, and 
		$\Gamma_{4,0}=D_0T_2D_0T_2D_0$.
		We needed to iterate three times to get an error term of size $z^{4+}$ with respect to the spectral parameter.
		
	\end{proof}
	
	In particular, we note that, from \eqref{eqn:Minv reg}, and the discussion after \eqref{eqn:RL low bds}, for some $z$-independent absolutely bounded operator $\Gamma_0$,  we have
	\begin{align}\label{eqn:Minv pm diff}
		M_+^{-1}(z)-M_-^{-1}(z)=z^2 \Gamma_0+O_2(z^{2+})
	\end{align}
	
	\begin{rmk}\label{rmk:neg branch}
		
		As mentioned in the introduction of this section, we complete our analysis on the positive branch of the spectrum $[m,\infty)$.  To adapt to the negative branch, when $\lambda$ is in a neighborhood of $-m$, we utilize the change of variables $\lambda=-\sqrt{m^2+z^2}$ and make the required adaptations; replacing $I_{uc}$ and $M_{uc}$ with $I_{lc}$ and $M_{lc}$ and adjusting the signs to account for  expanding $-\sqrt{z^2+m^2}=-m-\frac{z^2}{2m}+  O_4(z^4)$.  The proofs follow analogously in this case.
		
	\end{rmk}

	\section{Dispersive Bounds When Zero Is Regular}\label{sec:reg dis}
	Having established expansions for $M_{\pm}^{-1}(z)$ as an absolutely bounded operator on $L^2(\R^4)$ in Lemma~\ref{lem:Minv reg} in the previous section, we now show that the integral kernel of the solution operator satisfies the claimed pointwise bounds in Theorem~\ref{thm:main}.  That is, we show that the integral kernel of the low energy evolution satsifies
	$$
		\sup_{x,y\in\R^4}|e^{-it\mathcal H}\chi(\mathcal H)P_{ac}(\mathcal H)(x,y)|\les \la t\ra^{-2}.
	$$
	We utilize the Stone's formula, \eqref{eqn:stone}, to reduce to oscillatory integrals in the spectral parameter $z$. 
	We iterate the resolvent identity and the symmetry resolvent identity to arrive at the following expansion for the perturbed resolvents:
	\begin{align}\label{eqn:BS}
		\mR_V^\pm(z)=\sum_{k=0}^{2M+2} \mR_0^\pm(z)(-V\mR_0^\pm(z))^k+(\mR_0^\pm(z)V)^M\mR_0^\pm(z) v^*M_\pm^{-1}(z)v\mR_0^\pm(z)(V\mR_0^\pm(z))^{M}.
	\end{align}
	In light of the spatial singularities of the resovlents, \eqref{eqn:RL low bds}, we select $M$ large enough to ensure the iterated resolvents on either side of $M_{\pm}^{-1}(z)$ are locally $L^2$.
	One needs a different argument for the terms of the finite Born Series expansion, and the tail of the Born series.  Heuristically, one would like to integrate by parts twice to obtain the desired time decay.  The oscillatory behavior of the free resolvents presents a challenge, \eqref{eqn:RH defn} shows that the second derivative of the integral kernel of the resolvent grows in the spatial variables.  Accordingly, we recall the following oscillatory integral lemma from \cite{eg2}. 
	\begin{lemma}\label{lem:high stat phase}   If 
		$$
		|a(z)|\les 	\frac{z\chi(z) \widetilde \chi(zr)}{(1+zr)^{\f12}} ,
		\qquad |\partial_z a(z)|\les 
		\frac{\chi(z) \widetilde \chi(zr)}{(1+zr)^{\f12}},
		$$
		then we have the bound
		\begin{align*}
			\bigg|\int_0^\infty e^{-it \phi_\pm(z) } a(z)\, dz\bigg| \les \la t\ra^{-1},
		\end{align*}
		where $\phi_\pm(z)=\sqrt{z^2+m^2}\mp \frac{z r}{t}$.
	\end{lemma}
	This will allow us to obtain the second power of time decay without the need for spatial weights.  To illustrate, we prove the low energy  dispersive estimates for the free Dirac evolution.
	
	\begin{lemma}\label{lem:free low disp}
		
		We have the bounds
		\begin{align*}
			\sup_{x,y,\in\R^4}\bigg|\int_{0}^{\infty}e^{-it\sqrt{z^2 + m^2}}\frac{z\chi(z)}{\sqrt{z^2 + m^2}}[\mR_0^+-\mR_0^-](x,y)\,dz
			\bigg|\les \la t\ra^{-2}.
		\end{align*}
		Consequently, the free Dirac solution operator satisfies the bound
		$$
		\|e^{-itD_m}\chi(D_m)\|_{L^1\to L^\infty}\les \la t\ra^{-2}.
		$$
		
	\end{lemma}
	
	\begin{proof}
		We consider when $z|x-y|\les 1 $ and $z|x-y|\gtrsim 1$ separately, that is we write $\mR_0^\pm=\mR_L^\pm+\mR_H^\pm$ and consider two integrals.  We consider $z|x-y|\les 1$ first, and seek to control
		\begin{align}\label{eqn:RL free low}
			\int_{0}^{\infty}e^{-it\sqrt{z^2 + m^2}}\frac{z\chi(z)}{\sqrt{z^2 + m^2}}[\mR_L^+-\mR_L^-](z)(x,y)\,dz.
		\end{align}
		By \eqref{eqn:RL diff bds}, the integrand is bounded uniformly in $x,y$ on the support of $\chi(z)$, hence the integral is bounded.  By  \eqref{eqn:RL diff bds}, we may integrate by parts once with no boundary terms due to the smallness of $\mR_L^+-\mR_L^-$ at $z=0$ and the support of the cut-off. Accordingly, we may rewrite \eqref{eqn:RL free low} as
		\begin{multline*}
			\bigg(\frac{e^{-it\sqrt{z^2+m^2}} \sqrt{z^2+m^2}}{(-it)^2z}\partial_z[\chi(z)[\mR_L^+-\mR_L^-])(z)(x,y)]\bigg)\bigg|_{z=0}\\
			-\frac{1}{(-it)^2}
			\int_{0}^{\infty}e^{-it\sqrt{z^2 + m^2}}\partial_z \bigg(\frac{\sqrt{z^2+m^2}}{z}\partial_z[\chi(z)[\mR_L^+-\mR_L^-])(z)(x,y)\bigg)\,dz.
		\end{multline*}
		The boundary term from the second integration by parts is bounded by $|t|^{-2}$ uniformly in $x,y$ by \eqref{eqn:RL diff bds}, the division by $z$ is canceled exactly by the smallness of $\mR_L^+-\mR_L^-$.  We do not explicitly consider when $\partial_z$ acts on a cut off as $|\chi'(z)|\approx 1$ on the set $z\approx 1$, or may be bounded by division by $z$.  In all cases, if the derivatives act on a cut-off, we may bound by the arguments when the derivative acts on other pieces of the integrand.  By the triangle inequality, we need only bound
		\begin{multline*}
			\frac{1}{|t|^2}\int_{0}^\infty \bigg|\partial_z \bigg(\frac{\sqrt{z^2+m^2}}{z}\partial_z(\chi(z)[\mR_L^+-\mR_L^-])\bigg)(z)(x,y)\bigg|\, dz\\
			\les \frac{1}{|t|^2} \int_0^\infty \bigg|\partial_z \bigg(\frac{\sqrt{z^2+m^2}}{z}\partial_z[\chi(z) \chi(z|x-y|) (z^2\mG_1(x,y) + O_2( z^{4}(|x-y|^2))]\bigg|\, dz\\
			\les \frac{1}{|t|^2}\int_0^\infty \bigg| \partial_z[ \sqrt{z^2+m^2}\chi(z)\chi(z|x-y|)(\mG_1(x,y)+ z^{2}|x-y|^2)]\bigg|\, dz\ \les \frac{1}{|t|^2}.
		\end{multline*}
		By \eqref{eq:G1 def}, we see that the first term is clearly bounded, while for the second term we note
		that 
		$$
		\int_0^\infty z|x-y|^2 \chi(z|x-y|)\, dz\les \int_0^{|x-y|^{-1}} z|x-y|^2\, dz \les 1.
		$$
		We now turn our attention to when $z|x-y|\gtrsim 1$, here we do not use the cancellation between the + and - terms, but estimate them individually.  In this regime we recall \eqref{eqn:RH defn}, multiplying by $(z|x-y|)^{\f32}\gtrsim 1$,   we have $|\mR_H^\pm(z)(x,y)|\les z^2$, so the integral is again bounded uniformly in $x,y$.  We now may integrate by parts once without boundary terms and we seek to bound
		\begin{multline*}
			\frac{1}{-it} \int_0^\infty e^{-it\sqrt{z^2+m^2}}\partial_z \big( e^{\pm iz|x-y|}\chi(z)\widetilde \chi(z|x-y|)\widetilde \omega_{\pm}(z|x-y|) \big)\, dz\\
			=\frac{1}{-it} \int_0^\infty e^{-it\sqrt{z^2+m^2}\pm iz|x-y|}\big( \pm i|x-y|\chi(z)\widetilde \chi(z|x-y|)\widetilde \omega_{\pm}(z|x-y|)\\
			+\partial_z( \chi(z)\widetilde \chi(z|x-y|)\widetilde \omega_{\pm}(z|x-y|)) \big)\, dz:=\frac{1}{-it} \int_0^\infty e^{-it\sqrt{z^2+m^2}\pm iz|x-y|}a(z)\, dz.
		\end{multline*}
		Where by \eqref{eqn:RH defn}, we see that $a(z)$ satisfies the bounds of Lemma~\ref{lem:high stat phase} with $r=|x-y|$.  Hence, one obtains that $\la t\ra^{-2}$ bound uniformly in $x,y\in \R^4$.
	\end{proof}
	
	For the remaining terms of the Born series, those summands in \eqref{eqn:BS} where $k>0$, we utilize the integral kernels of the resolvents and the iterated integrals over $\R^{4k}$.  This necessitates control over the spatial variables in the iterated resolvents to ensure that our estimates are uniform in $x,y\in \mathbb R^4$.
	We recall the following integral estimates (c.f. Lemma 6.3 in \cite{EG1} and Lemma 3.8 in \cite{goldVis}) :
	\begin{lemma}\label{lem:spatial estimates}
		Fix $x,y\in \R^n$, with $0 \leq k, \ell < n$, $\delta > 0$, $k + \ell + \delta \geq n$, $k + \ell \neq n$:
		\begin{align*}
			\int_{\R^n} \frac{\ang{z}^{-\delta-}}{|z-x|^k |y-z|^\ell} dz 
			\lesssim
			\left\{ \begin{array}{ll}
				|x - y|^{-\max \{0, k + \ell - n\}}                & \text{if } |x-y|   <  1; \\
				|x - y|^{-\min \{k, \ell, k + \ell + \delta - n\}} & \text{if } |x-y| \geq 1.
			\end{array}\right..
		\end{align*}
		Consequently, we have that
		\begin{align}\label{eqn:R3kell}
			\int_{\R^n} \frac{\ang{z}^{-\delta-}}{|z-x|^k |y-z|^\ell} dz 
			\lesssim \frac{1}{|x-y|^p},
		\end{align}
		where we may take any $p\in [\max \{0, k + \ell - n\},\min \{k, \ell, k + \ell + \delta - n\}]$ as desired. 
		
		Further, if $0\leq \delta,k<n$ with $k+\delta>n$, then
		$$
		\int_{\R^n}\frac{\la x\ra^{-\delta}}{|x-y|^k}\, dx \les \la y\ra^{n-k-\delta}.
		$$
		
	\end{lemma}
	In the case that $k+\ell=n$, one can use that $a^{-k}b^{-\ell}\les a^{-k+}b^{-\ell}+a^{-k-}b^{-\ell}$ to avoid logarithmic bounds.  We further need to collect appropriate bounds on derivatives of the free resolvents.  Combining \eqref{eqn:RL low bds} and \eqref{eqn:RH defn} we have (for $k=1,2$)
	\begin{align}\label{eqn:R0 for bs}
		|\partial_z^k \mR_0^\pm(z)(x,y)|\les 	z^{2-k}(1+\log(z|x-y|)+|x-y|^{-2})\chi(z|x-y|)
		+\frac{z^\f12 \widetilde\chi(z|x-y|)}{|x-y|^{\f32-k}}.
	\end{align}
	While for $k=0,1,2$ we have
	\begin{align}\label{eqn:R0 diff for bs}
		|\partial_z^k [\mR_0^+-\mR_0^-](z)(x,y)|\les 	z^{2-k}(1+O(z^2|x-y|^2))\chi(z|x-y|)
		+\frac{z^\f12\widetilde\chi(z|x-y|)}{|x-y|^{\f32-k}}.
	\end{align}
	With these tools in hand, we now control the contribution of the finite Born series terms to the Stone's formula. We recall the algebraic identity
	\begin{align}\label{eqn:alg identity}
		\prod_{k = 0}^M A_k^+ - \prod_{k = 0}^M A_k^-
		= \sum_{\ell=0}^M \bigg(\prod_{k = 0}^{\ell - 1}A_k^-\bigg)
		\big(A_\ell^+ - A_\ell^-\big)\bigg(
		\prod_{k = \ell + 1}^M A_k^+\bigg).
	\end{align} 
	In light of this, we need to control the contribution of
	\begin{lemma}\label{lem:born low}
		Fix $k\in \mathbb N\cup\{0\}$. Assume that $V(x) \lesssim \langle x\rangle^{-\delta-}$ for some $\delta>\f72$, then
		\begin{align*}
			\sup_{x,y,\in\R^4}\bigg|\int_{0}^{\infty}e^{-it\sqrt{z^2 + m^2}}\frac{z\chi(z)}{\sqrt{z^2 + m^2}}(\cR_0^+V)^{k} \cR_0^+ - (\cR_0^-V)^{k} \cR_0^-(z)(x,y)\,dz
			\bigg|\les \la t\ra^{-2}.
		\end{align*}
	\end{lemma}
	\begin{proof}
		When $k=0$, this is the evolution of the free Dirac operator considered in Lemma~\ref{lem:free low disp}. 
		Fix $k>0$ and
		let $k_1, k_2 \in \mathbb N\cup\{0\}$ be such that $k_1 + k_2 = k$.  Then, it suffices to control the contribution of
		\begin{align*}
			\int_{0}^{\infty}e^{-it\sqrt{z^2 + m^2}}\frac{z\chi(z)}{\sqrt{z^2 + m^2}}(\cR_0^-V)^{k_1}(\cR_0^+ - \cR_0^-)(V\cR_0^+)^{k_2}(z)(x,y)\,dz.
		\end{align*}
		Combining \eqref{eqn:RL diff bds}, \eqref{eqn:RH defn}, and \eqref{eqn:RL low bds}, when $0<z<1$, we have
		\begin{align}\label{eqn:R0 bds bs helpful}
			|[\mR_0^+-\mR_0^-](z)(x,y)|\les z^2, \quad
			|\mR_0^\pm(z)(x,y)|\les \frac{1}{|x-y|^3}+\frac{1}{|x-y|^{\f32}}.
		\end{align}
		So that, with $x_0=x$ and $x_{k+1}=y$, we have
		\begin{multline*}
			\bigg|\int_{0}^{\infty}e^{-it\sqrt{z^2 + m^2}}\frac{z\chi(z)}{\sqrt{z^2 + m^2}}(\cR_0^-V)^{k_1}(\cR_0^+ - \cR_0^-)(V\cR_0^+)^{k_2}(z)(x,y)\,dz\bigg|\\
			=\bigg|\int_{0}^{\infty}e^{-it\sqrt{z^2 + m^2}}\frac{z\chi(z)}{\sqrt{z^2 + m^2}}\int_{\R^{4k}}
			\prod_{j=1}^{k_1}(\cR_0^-(z)(x_{j-1},x_j)V(x_j))\\
			\qquad\qquad\qquad\qquad\qquad(\cR_0^+ - \cR_0^-)(z)(x_{k_1},x_{k_1+1})
			\prod_{i=k_1+1}^{k_2}(V(x_i)\cR_0^+(z)(x_i,x_{i+1})\,d\vec x \, dz\bigg|\\
			\les \int_0^\infty \chi(z) \int_{\R^{4k}} 
			\prod_{j=1}^{k_1}\frac{|V(x_j)|\la x_{j-1}-x_{j}\ra^{\f32}}{|x_{j-1}-x_{j}|^3}
			\prod_{i=k_1+1}^{k_2}\frac{|V(x_i)|\la x_i-x_{i+1}\ra^{\f32}}{|x_i-x_{i+1}|^3}\, d\vec x.
		\end{multline*}
		All terms in the integral are positive, we may integrate in any order.  We consider the integral in $k_1$ first,
		$$
		\int_{\R^4}\frac{|V(x_{k_1})|\la x_{k_1-1}-x_{k_1}\ra^{\f32}}{|x_{k_1-1}-x_{k_1}|^3}\, dx_{k_1}\les \int_{\R^4}\bigg(\frac{\la x_{k_1}\ra^{-\delta-}}{|x_{k_1-1}-x_{k_1}|^3}
		+\frac{\la x_{k_1}\ra^{-\delta-} }{|x_{k_1-1}-x_{k_1}|^{\f32}}\bigg) \, dx_{k_1} \les 1.
		$$
		We note that this bound requires only $|V(x)|\les \la x\ra^{-\delta}$ for some $\delta>\f52$.
		By integrating in $x_{k_1}$ first we remove any singularities in the spatial variables, which are a concern for small $k$.  We may then integrate `outward'; integrating in $x_{k_1\pm1}$, then $x_{k_1\pm2}$ and so on.  This separates the singularities and the remaining integrals in the spatial variables are all now dominated by the integral above, using Lemma~\ref{lem:spatial estimates} repeatedly shows that 
		\begin{align*}
			\sup_{x,y,\in\R^4}\bigg|\int_{0}^{\infty}e^{-it\sqrt{z^2 + m^2}}\frac{z\chi(z)}{\sqrt{z^2 + m^2}}(\cR_0^+V)^{k} \cR_0^+ - (\cR_0^-V)^{k} \cR_0^-(z)(x,y)\,dz
			\bigg|\les 1,
		\end{align*}
		showing the boundedness, provided that $\delta>\f52$.
		
		To establish the time decay we integrate by parts once, then have to consider cases.
		We first consider the case of $k = 1$. By symmetry, it suffices to consider
		\begin{multline*}
			\int_{0}^{\infty}e^{-it\sqrt{z^2 + m^2}}\frac{z}{\sqrt{z^2 + m^2}}\chi(z)[\cR_0^+ - \cR_0^-]V\cR_0^+(z)(x,y)dz\\
			=\frac{e^{-it\sqrt{z^2+m^2}}\chi(z)[\cR_0^+ - \cR_0^-]V\cR_0^+(z)(x,y) }{-it}\bigg|_{z=0}\\
			+\frac{1}{it}\int_0^\infty e^{-it\sqrt{z^2 + m^2}}
			\partial_z\big[\chi(z)[\cR_0^+ - \cR_0^-]V\cR_0^+ \big](z)(x,y)\, dz.
		\end{multline*}
		As in the proof of Lemma~\ref{lem:free low disp}, we don't explicitly consider when the derivatives act on cut-offs.
		By \eqref{eqn:R0 bds bs helpful}, we see that there are no boundary terms.  We consider cases based on where the derivative acts.  If the derivative acts on the difference of resolvents, we decompose them into $\mR_L^\pm$ and $\mR_H^\pm$.  When we have the difference of $\mR_L^\pm$, we use \eqref{eqn:RL diff bds} to write
		\begin{multline*}
			\frac{1}{it}\int_0^\infty e^{-it\sqrt{z^2 + m^2}}
			\big[\chi(z)\partial_z[\cR_L^+ - \cR_L^-]V\cR_0^+ \big](z)(x,y)\, dz\\
			=\frac{1}{it}\int_0^\infty \int_{\R^4} e^{-it\sqrt{z^2 + m^2}}
			\big[\chi(z)(2cz\mG_1(x,x_1)V(x_1)\cR_0^+ (z)(x_1,y)\\
			+O_1(z^3|x-x_1|^2)\chi(z|x-x_1|) V(x_1)\cR_0^+ (z)(x_1,y)\big]\, dx_1 dz.
		\end{multline*}
		By \eqref{eqn:R0 for bs}, when $0<z<1$ and dividing by powers of $z|x-y|$ when $z|x-y|\les 1$, we have
		\begin{align}\label{eqn:R0 deriv in bs}
			|\partial_z \mR_0^\pm(z)(x,y)|\les \frac{z}{|x-y|^{2}}+\frac{z^{\f12}}{|x-y|^{\f12}},
		\end{align}
		which suffices to allow us to integrate by parts again.   The only boundary terms that appear arise when the second derivative in $z$ acts on the $2cz\mG_1$, so we may bound the contribution by
		\begin{multline*}
			\frac{1}{|t|^2} \int_{\R^4} 
			\big|2c\mG_1(x,x_1) V(x_1)\cR_0^+ (0)(x_1,y)\big|\, dx_1\\
			+\frac{1}{|t|^2}\int_0^\infty \int_{\R^4} \bigg|
			\partial_z\big[\sqrt{z^2+m^2}
			\chi(z)(1+O_1(z^2|x-x_1|^2)\chi(z|x-x_1|) V(x_1)\cR_0^+ (z)(x_1,y)\big]\, dx_1 dz.
		\end{multline*}
		The first term is controlled by \eqref{eq:G1 def}, \eqref{eqn:R0 bds bs helpful},
		and Lemma~\ref{lem:spatial estimates}, which shows the contribution of the boundary term is bounded by $|t|^{-2}$ uniformly in $x,y$.  For the remaining piece, using \eqref{eqn:R0 bds bs helpful} its contribution is dominated by
		\begin{multline*}
			\frac{1}{|t|^2}\int_0^\infty \int_{\R^4} \la x_1\ra^{-\f72-}
			(1+z|x-x_1|^2\chi(z|x-x_1|))\bigg(\frac{1}{|x_1-y|^3}+\frac{1}{|x_1-y|^{\f32}}\bigg)\, dx_1\, dz\\
			+\frac{1}{|t|^2}\int_0^\infty \int_{\R^4} \la x_1\ra^{-\f72-}\bigg(\frac{1}{|x_1-y|^3}+\frac{1}{|x_1-y|^{\f12}}\bigg) \, dz\, dx_1.
		\end{multline*}
		The $z$ integral is bounded, the resulting spatial integral in $x_1$ is bounded uniformly in $x,y$ by Lemma~\ref{lem:spatial estimates}.  The spatial integrals in the second term are similarly bounded uniformly in $x,y$ by Lemma~\ref{lem:spatial estimates}.
		
		We next consider the case when the resolvents on the left are $\mR_H$.  Here we cannot use the cancellation, but instead bound each individually.  Here we note that by \eqref{eqn:RH defn},
		$$
		\partial_z \mR_H^\pm(z)(x,y)=e^{\pm iz|x-y|}\omega_\pm(z|x-y|), \qquad |\partial_z^k \omega_{\pm}(z|x-y|)|\les z^{\f12-k}\widetilde \chi(z|x-y|) \la z|x-y|\ra^{-\f12}.
		$$
		Accordingly, we seek to control (ignoring the spatial integral for the moment)
		$$
		\int_0^\infty e^{-it\sqrt{z^2+m^2}\pm iz|x-y|} \chi(z) \omega_\pm(z|x-x_1|)V(x_1) \mR_0^+(z|x_1-y|) \, dz
		$$
		As in the proof of Lemma~\ref{lem:free low disp}, we employ Lemma~\ref{lem:high stat phase}. this time with
		$$
		a(z)=\chi(z)\omega_\pm(z|x-x_1|)V(x_1)
		\mR_0^+(z|x_1-y|).
		$$
		By \eqref{eqn:RH defn}, \eqref{eqn:R0 bds bs helpful}, and \eqref{eqn:R0 deriv in bs} we have the bounds
		\begin{align*}
			|a(z)|\les \frac{z\chi(z)\widetilde \chi(z|x-x_1|)}{(1+z|x-x_1|)^{\f12}} \la x_1\ra^{-\f72-}\bigg(\frac{1}{|x_1-y|^3}+\frac{1}{|x_1-y|^{\f32}} \bigg),\\
			|a'(z)|\les \frac{\chi(z)\widetilde \chi(z|x-x_1|)}{(1+z|x-x_1|)^{\f12}} \la x_1\ra^{-\f72-}\bigg(\frac{1}{|x_1-y|^2}+\frac{1}{|x_1-y|^{\f12}} \bigg).
		\end{align*}
	 	The desired $|t|^{-2}$ bound follows from applying Lemma~\ref{lem:high stat phase} for the $z$ integral and then Lemma~\ref{lem:spatial estimates} for the $x_1$ integral, that is
		\begin{multline*}
			\sup_{x,y\in \R^4}\bigg|\int_0^\infty e^{-it\sqrt{z^2+m^2}\pm iz|x-y|} \chi(z) \omega_\pm(z|x-x_1|)V(x_1) \mR_0^+(z|x_1-y|) \, dz|\\
			\les \la t\ra^{-1} \sup_{x,y\in\R^4}\int_{\R^4} \la x_1\ra^{-\f72 -}\bigg(\frac{1}{|x_1-y|^3}+\frac{1}{|x_1-y|^2} \bigg)\, dx_1 \les \la t\ra^{-1}.
		\end{multline*}
		
		The remaining possibilities when $k=1$,  controlling the contribution of
		$$
		\frac{1}{it}\int_0^\infty e^{-it\sqrt{z^2 + m^2}}
		\chi(z)[\cR_L^+ - \cR_L^-]V\big[\partial_z\cR_0^+ \big](z)(x,y)\, dz
		$$
		or when the derivative acts on the cut-off $\chi(z)$
		are controlled by slight variations of the analysis above.  We note that in when $\partial_z\cR_0^+$ is supported on `low' argument, one obtains further smallness in the spectral parameter $z$ that  simplifies the analysis.  Whereas if $\partial_z \mR_0^+$ is supported on `high' argument, one applies Lemma~\ref{lem:high stat phase} as above.
		
		We now argue that $k>1$ may be reduced to the arguments presented above.  In all cases, we begin by one integration by parts, there are no boundary terms due to the $\mR_0^+-\mR_0^-$, and we seek to bound
		$$
		\frac{1}{it}\int_{0}^{\infty}e^{-it\sqrt{z^2 + m^2}}\partial_z \big[\chi(z)(\cR_0^-V)^{k_1}(\cR_0^+ - \cR_0^-)(V\cR_0^+)^{k_2}\big](z)(x,y)\,dz\bigg|\\
		$$
		For iterated resolvents, it suffices to show that iterated resolvents may be bounded by the same upper bounds used for single resolvents in the argument for $k=1$.  Notice that by \eqref{eqn:R0 for bs} and Lemma~\ref{lem:spatial estimates} we have
		\begin{multline*}
			|[\mR_0^\pm V\mR_0^\pm](z)(x,y)|=\bigg|
			\int_{\R^4} \mR_0^\pm(z)(x,x_1)V(x_1)\mR_0^\pm(z)(x_1,y)\, dx_1\bigg|\\
			\les \int_{\R^4} \frac{\la x_1\ra^{-\f72-}}{|x-x_1|^3|x_1-y|^3}+\frac{\la x_1\ra^{-\f72-}}{|x-x_1|^{\f32}|x_1-y|^{\f32}}\, dx_1 \les \frac{1}{|x-y|^3}+\frac{1}{|x-y|^{\f32}}.
		\end{multline*}
		If one derivative acts on a free resolvent, using \eqref{eqn:R0 deriv in bs} we have (when $0<z<1$)
		\begin{multline*}
			|[\mR_0^\pm V\partial_z\mR_0^\pm](z)(x,y)|=\bigg|
			\int_{\R^4} \mR_0^\pm(z)(x,x_1)V(x_1)\partial_z\mR_0^\pm(z)(x_1,y)\, dx_1\bigg|\\
			\les \int_{\R^4} \frac{z\la x_1\ra^{-\f72-}}{|x-x_1|^3|x_1-y|^2}+\frac{z^{\f12}\la x_1\ra^{-\f72-}}{|x-x_1|^{\f32}|x_1-y|^{\f12}}\, dx_1 \les \frac{z}{|x-y|^2}+\frac{z^{\f12}}{|x-y|^{\f12}}.
		\end{multline*}
		That is, we can bound the iterated resolvents by the same bounds we used for a single resolvent in the $k=1$ case.
		This allows us to implement the $k=1$ argument whenever the derivative from the first integration by parts acts on either the leading or lagging free resolvent.
		
		Consider when the derivative acts on the leading resolvent.  We again decompose into $\mR_L$ and $\mR_H$.  Consider first $\mR_L$, here we may integrate by parts again by \eqref{eqn:RL low bds} and \eqref{eqn:R0 diff for bs}, the integrand is $O_1(z^{\f52})$ with respect to the spectral variable, so no boundary terms arise and we bound
		\begin{multline*}
			\frac{1}{|t|^2}\int_{0}^{\infty}e^{-it\sqrt{z^2 + m^2}}\partial_z \bigg[\frac{\sqrt{z^2+m^2}}{z} \big[\chi(z)(\partial_z\mR_L^-)(\cR_0^-V)^{k_1-1}(\cR_0^+ - \cR_0^-)(V\cR_0^+)^{k_2}\big]\bigg](z)(x,y)\,dz\bigg|
		\end{multline*}
		Utilizing \eqref{eqn:RL low bds} and \eqref{eqn:R0 diff for bs} we may dominate this by integrals of the form
		$$
		\frac{1}{|t|^2}\int_0^\infty \chi(z) \int_{\R^{4k}} 
		(1+|x-x_1|^{-2})\prod_{j=2}^{k_1}\frac{|V(x_j)|\la x_{j-1}-x_{j}\ra^{\f52}}{|x_{j-1}-x_{j}|^3}
		\prod_{i=k_1+1}^{k}\frac{|V(x_i)|\la x_i-x_{i+1}\ra^{\f52}}{|x_i-x_{i+1}|^3}\, d\vec x
		$$
		Here we used that $|\partial_z^k [\mR_0^+-\mR_0^-](z)(x,y)|\les z^{2-k}$ for $k=0,1,2$.  As in the argument for boundedness, we may integrate in $x_{k_1}$ first, the assumed decay of $V$ more than suffices for repeated use of Lemma~\ref{lem:spatial estimates} to show the spatial integrals are bounded uniformly in $x,y$.  Here we note that the powers of $\la x_{j-1}-x_j\ra^{\f52}$ are not a sharp bound, as at most one of these resolvents can be differentiated.  The argument for a leading $\mR_H$ follows similarly, with the iterated resolvents contributions again bounded by repeated use of Lemma~\ref{lem:spatial estimates}.

		Here we note that if the derivative acts on an inner resolvent, we may safely integrate by parts again without considering cases.  This case necessitates the assumed decay on $V$ since
		\begin{align*}
			|V(x_j)\partial_z^2 \mR_0^\pm(z)(x_j,x_{j+1}) V(x_{j+1})|&\les  z^{0-} \la x_j\ra^{-\f72-}\bigg[|x_j-x_{j+1}|^{\f12}+|x_j-x_{j+1}|^{-2}\bigg] \la x_{j+1}\ra^{-\f72-}\\
			&\les z^{0-}\la x_j\ra^{-3-}(1+|x_j-x_{j+1}|^{-2})\la x_{j+1}\ra^{-3-}.
		\end{align*}
		The decay on $V$ is most constrained by the contribution of
		$$
		\frac{1}{(it)^2}\int_0^\infty  e^{-it\sqrt{z^2+m^2}} \frac{\sqrt{z^+m^2}}{z}\chi(z)\bigg[[\mR_L^+-\mR_L^-]V[\partial_z^2\mR_0^+] V\mR_H^+\bigg](z)(x,y)\,   dz
		$$
		By the bounds above, we dominate by
		$$
		\frac{1}{|t|^2}\int_0^\infty \int_{\R^8} z^{1-} \chi(z)(z|x-x_1|)^{-\ell} \la x_1\ra^{-3-}\frac{\la x_2 \ra^{-3-}}{|x_2-y|^{\f32}} \,dx_1\, dx_2\, dz.
		$$
		Here we take advantage of the fact that $z|x-x_1|\les 1$ and select $\ell=1-$ to dominate by
		$$
		\frac{1}{|t|^2}\int_0^\infty \int_{\R^8} z^{-1+}\chi(z) \frac{ \la x_1\ra^{-3-}}{|x-x_1|^{1-}}\frac{\la x_2 \ra^{-3-}}{|x_2-y|^{\f32}} \,dx_1\, dx_2\, dz \les \frac1{|t|^2},
		$$
		uniformly in $x,y$.
		Here we see that we must have $\delta>\f72$ for the $x_1$ integral to be bounded.
	\end{proof}

	We now work towards bounding the tail of the Born series, the summand in \eqref{eqn:BS} containing the operators $M_{\pm}^{-1}(z)$.  In the symmetric resolvent identity we bound $M_{\pm}^{-1}(z)$ and its derivatives as absolutely bounded operators on $L^2(\R^4)$. The following consequences of Lemma~\ref{lem:spatial estimates} allow us to control the $L^2$ norms of iterated resolvents.
	\begin{lemma}\label{lem:abstractSpatialIntegrals}
		Let $0\leq a_0 < 4$, $-\f12 \leq a_1 < 4$, $-\f12\leq a_2 < 4$ satisfy $a_0 + a_1 + a_2 < 10$, and  $\delta>5$, then
		\begin{equation*}
			\sup_{x\in \mathbb R^4}\left\|\int_{\R^8} \frac{\langle x_1\rangle^{-\delta}\langle x_2\rangle^{-\delta}\langle x_3\rangle^{-\delta/2}}{|x - x_1|^{a_0}|x_1 - x_2|^{a_1}|x_2 - x_3|^{a_2}}dx_1dx_2\right\|_{L^2_{x_3}}\lesssim 1.
		\end{equation*}
		Further, if $0\leq a_0 < 4$, $-1/2 \leq a_1,a_2,a_3 < 4$, satisfy $a_0 + a_1 + a_2 + a_3 < 14$, and   $\delta>5$, then
		\begin{equation*}
			\sup_{x\in \R^4}\left\|\int_{\R^{12}} \frac{\langle x_1\rangle^{-\delta}\langle x_2\rangle^{-\delta}\langle x_3\rangle^{-\delta}\langle x_4\rangle^{-\delta/2}}{|x - x_1|^{a_0}|x_1 - x_2|^{a_1}|x_2 - x_3|^{a_2}|x_3 - x_4|^{a_3}}dx_1dx_2dx_3\right\|_{L^2_{x_4}}\lesssim 1.
		\end{equation*}
	\end{lemma}
	
	\begin{proof}
		We consider the first claim and consider the most extreme cases, when $a_0,a_1,a_2>3$, and $a_0=0,a_1=a_2=-\f12$.  The remaining cases may be bounded by superpositions of these cases. First, in the $a_0,a_1,a_2>3$ case,
		applying Lemma~\ref{lem:spatial estimates} in the $x_1$ integral, since $a_0+a_1>6$, then in the $x_2$ integral, $1<a_0 + a_1 + a_2 - 8<2$,   we have the $L^2$ norm is dominated by
		\begin{align*}
			\left\|\int_{\R^4} \frac{\langle x_2\rangle^{-\delta}\langle x_3\rangle^{-\delta/2}}{|x - x_2|^{a_0 + a_1 - 4}|x_2 - x_3|^{a_2}} dx_2\right\|_{L^2_{x_3}} 
			\lesssim\left\|\frac{\langle x_3\rangle^{-\delta/2}}{|x - x_3|^{a_0 + a_1 + a_2 - 8}}\right\|_{L^2_{x_3}} 
			\lesssim 1,
		\end{align*}
		uniformly in $x$.  Here one makes obvious adjustments if $a_0+a_1=4$ and consider a sum of terms replacing $a_0$ with $a_0+\epsilon$ and $a_0-\epsilon$ for a sufficiently small $\epsilon$.
		In the last bound we use that $a_0 + a_1 + a_2 < 10$ to see that the integrand is locally $L^2$.
		
		Now we consider the case where $a_0=0$ and $a_1=a_2=\f12$, here there is growth in $|x_1-x_2|$ and $|x_2-x_3|$. Using that $|x-y|^{\f12} \lesssim \langle x\rangle^{\f12}\langle y\rangle^{\f12}$. Since $\delta > 5$, we dominate by
		\begin{multline*}
			\left\|\int_{\R^8} \langle x_1\rangle^{-\delta}|x_1 - x_2|^{\f12}\langle x_2\rangle^{-\delta}|x_2 - x_3|^{\f12}\langle x_3\rangle^{-\delta/2}dx_1dx_2\right\|_{L^2_{x_3}} \\
			\lesssim \left\|\int_{\R^8}\langle x_1\rangle^{-\delta+\f12}\langle x_2\rangle^{-\delta+1}\langle x_3\rangle^{-\delta/2+\f12}dx_1 dx_2\right\|_{L^2_{x_3}} 
			\lesssim \left\|\langle x_3\rangle^{-(\delta-1)/2}\right\|_{L^2_{x_3}} 
			\lesssim 1.
		\end{multline*}
		This bound is uniform in $x\in \R^4$ provided $\delta-1>4$.  The remaining combinations of $a_0,a_1,a_2$ are dominated by these, hence the first claim follows.
		
		The second claim follows along similar lines.  We again consider the extreme cases, first if $a_0,a_1,a_2,a_3>3$ case, by repeated use of Lemma~\ref{lem:spatial estimates}, we have
		\begin{align*}
			\int_{\R^{12}} \frac{\langle x_1\rangle^{-\delta}\langle x_2\rangle^{-\delta}\langle x_3\rangle^{-\delta}\langle x_4\rangle^{-\delta/2}}{|x - x_1|^{a_0}|x_1 - x_2|^{a_1}|x_2 - x_3|^{a_2}|x_3 - x_4|^{a_3}}dx_1dx_2dx_3\les \frac{\langle x_4\rangle^{-\delta/2}}{|x - x_4|^{a_0 + a_1 + a_2 - 12}}.
		\end{align*}
		By the assumption that $a_0 + a_1 + a_2 + a_4 < 14$, this is in $L^2$ uniformly in $x$.  Similarly when $a_0=1$ and the remaining $a_j=-\f12$, repeated use of $|x-y|^{\f12}\lesssim  \langle x\rangle^{\f12}\langle y\rangle^{\f12}$ and Lemma~\ref{lem:spatial estimates} bounds the integral by $\la x_4\ra^{-(\delta-1)/2}$ uniformly in $x$, which is in $L^2$ provided $\delta>5$.  As before, we make the obvious adjustments to avoid logarithmic terms.
	\end{proof}

	We now seek to prove the following theorem to control the tail of the Born series.  For notational convenience, we denote
	\begin{align}\label{eq:Omega def}
		\Omega^\pm:= [(V\cR_0^\pm )^3  v^* M_\pm^{-1}v (\cR_0^\pm V )^3](z).
	\end{align}
	\begin{prop}\label{thm:RegularCaseMainTheorem}
		Suppose the thresholds of $\mathcal H$ are regular, and that $|V(x)| \lesssim \ang{x}^{-5-}$. Then
		\begin{align*}
			\sup_{x, y \in \R^4}\bigg|\int_{0}^{\infty}e^{-it\sqrt{z^2+m^2}}\frac{z\chi(z)}{\sqrt{z^2+m^2}}&\left(\cR_0^+ \Omega^+ \cR_0^+
			- \cR_0^- \Omega^- \cR_0^-\right)(z)(x, y)\, dz\bigg| \lesssim \frac{1}{\la t\ra^2}.
		\end{align*}
	\end{prop}
	We prove this claim through a series of lemmas.  As might be expected in even spatial dimensions, due to the different behavior of the free resolvent on `large' versus `small' arguments, we decompose the leading and lagging free resolvents into the high and low components, $\mR_H$ and $\mR_L$ respectively, and consider the resulting cases one at a time.  We note that one can generalize the result of Proposition~\ref{thm:RegularCaseMainTheorem} to replace $M_{\pm}^{-1}(z)$ with any absolutely bounded operators on $L^2$ that satisfy the same bounds in $z$ in \eqref{eqn:Minv reg} and \eqref{eqn:Minv pm diff}.
	
	We first consider the `low-low' case when the leading and lagging resolvents are $\mR_L$, for which we have the following bound.
	\begin{lemma}\label{lem:NotAllLowLemma1}
		Under the assumptions of Proposition~\ref{thm:RegularCaseMainTheorem}, we have the bound
		\begin{align*}
			\sup_{x, y \in \R^4}\bigg|\int_{0}^{\infty}e^{-it\sqrt{z^2+m^2}}\frac{z\chi(z)}{\sqrt{z^2+m^2}}&\left(\cR_L^+\Omega^+\cR_L^+- \cR_L^- \Omega^-\cR_L^-\right)(z)(x, y)dz\bigg| \lesssim \frac{1}{\la t\ra^2}.
		\end{align*}
	\end{lemma}
	There is a distinction between when all of the resolvents in $\Omega^\pm$ are supported on `low' argument versus when there is (at least) on that is not.  When all resolvents are supported on low argument, when $z|x_j-x_{j+1}|\les 1$, the cancellation between the `+' and `-' terms is crucial and must be captured to ensure integrability in $z$ and avoid premature boundary terms when integrating by parts.  Here we need to capture the leading terms in the expansions of the resolvents and $M_{\pm}^{-1}(z)$ exactly, as the exact behavior is needed whereas in other cases one can use coarser bounds.  For the convenience of the reader, we provide the following lemma that considers this case in detail before proving the more general result of Lemma~\ref{lem:NotAllLowLemma1}.  This is the most delicate case and allows us to develop the tools needed to more easily consider the cases when $\mR_L$ are replaced by $\mR_H$ and/or $\mR_0$ in the proofs of Lemmas~\ref{lem:NotAllLowLemma1} and \ref{HighOnEdgeLemma} below. 
	\begin{lemma}\label{lem:allLow}
		Under the assumptions of Proposition~\ref{thm:RegularCaseMainTheorem}, we have the bound
		\begin{multline*}
			\sup_{x, y \in \R^4}\bigg|\int_0^\infty  e^{-it\sqrt{z^2+m^2}}\frac{z\chi(z)}{\sqrt{z^2+m^2}}\left(\cR_L^+(V\cR_L^+ )^3  v^* M_+^{-1}v (\cR_L^+ V)^3  \cR_L^+\right . \\
			-\left .\mR_L^- (V\cR_L^- )^3  v^* M_-^{-1}v (\cR_L^- V)^3  \cR_L^-\right)(z)(x,y) dz  \bigg| \lesssim \frac{1}{\la t\ra^2}.
		\end{multline*}
	\end{lemma}
	
	\begin{proof}
		Noting \eqref{eqn:RL low bds}, Lemma~\ref{lem:abstractSpatialIntegrals}, and Lemma~\ref{lem:Minv reg}, we see that 
		\begin{multline*}
			\bigg|\int_0^\infty  e^{-it\sqrt{z^2+m^2}}\frac{z\chi(z)}{\sqrt{z^2+m^2}}
			\bigg[\mR_L^\pm(V\cR_L^\pm )^3 v^* M_\pm^{-1}v (\cR_L^\pm V)^3  \cR_L^\pm \bigg](z)(x,y)\, dz\bigg|\\
			\int_0^\infty z\chi(z) \|\mR_L^\pm(V \mR_L^\pm)^3\mR_L^\pm v^*\|_{L^2} \| |M_{\pm}^{-1}(z)|\|_{L^2\to L^2} \|v(\mR_L^\pm V)^3\mR_L^\pm\|_{L^2}\, dz \les 1,
		\end{multline*}
		uniformly in $x,y\in \mathbb R^4$.

		To establish time decay, we take utilize cancellation between the `+' and `-' terms as $z\to 0$.  For convenience, we denote
		$$
			\Omega_L^\pm:=[(V\cR_L^\pm )^3  v^* M_\pm^{-1}v (\cR_L^\pm V )^3](z).
		$$
		From the algebraic identity \eqref{eqn:alg identity}, we note that
		\begin{align}\label{eqn:RLOmegaRL diff}
			\mR_L^+\Omega_L^+\mR_L^+-\mR_L \Omega_L^-\mR_L^-
			=(\mR_L^+-\mR_L^-)\Omega_L^+\mR_L^+
			+\mR_L^-[\Omega_L^+-\Omega_L^-]\mR_L^+
			+\mR_L^-\Omega_L^-[\mR_L^+-\mR_L^-].
		\end{align}
		By \eqref{eqn:RL diff bds} and \eqref{eqn:Minv pm diff}, this quantity is bounded by $z^2$ with respect to the spectral parameter.  By Lemma~\ref{lem:abstractSpatialIntegrals} and \eqref{eqn:RL low bds}, the iterated resolvent $\mR_L(V\mR_L)^3v^*$ and $v(\mR_LV)^3\mR_L$ are in $L^2$, Lemma~\ref{lem:Minv reg} ensures that $M_{\pm}$ are absolutely bounded, which ensures that the accompanying the spatial integrals converge.  Further, \eqref{eqn:RLOmegaRL diff} is differentiable in $z$, though we expect to have a boundary term from the second integration by parts.
		
		We consider only the first and second terms in \eqref{eqn:RLOmegaRL diff}, the third term follows by symmetry considerations.
		We now consider the contribution of the first term in \eqref{eqn:RLOmegaRL diff} to the Stone's formula,
		\begin{equation*}
			\int_0^\infty e^{-it\sqrt{z^2+m^2}}\frac{z\chi(z)}{\sqrt{z^2+m^2}}\left[(\cR_L^+-\cR_L^-) \Omega^+_L \cR_L^+\right](z)(x,y) \,dz.
		\end{equation*}	
		Here, we integrate by parts twice, as noted above there is no boundary term from the first integration by parts, to obtain
		\begin{multline}\label{eqn:Tail RLdiff front}
			 -\frac{1}{(2it)^2} e^{-it\sqrt{z^2+m^2}}\frac{\sqrt{z^2 + m^2}}{z}\partial_z\left(\chi(z)(\cR_L^+-\cR_L^-) \Omega_L^+\cR_L^+\right) \, \bigg\vert_{z = 0} \\
			+ \frac{1}{(2it)^2}\int_0^\infty e^{-it\sqrt{z^2+m^2}}\partial_z\bigg(\frac{\sqrt{z^2 + m^2}}{z}\partial_z\left(\chi(z)\left((\cR_L^+-\cR_L^-) \Omega_L^+\cR_L^+\right)\right)\bigg)\, dz.
		\end{multline}	
		Here the boundary term is understood as the limiting value as $z\to 0^+$.  The division by $z$ after the derivative in $z$ ensures a non-zero boundary term.  We note that
		\begin{multline}
			(\cR_L^+-\cR_L^-) \Omega^+_L \cR_L^+(z)(x,y)\\
			=\int_{\R^{32}}[z^2+O_1(z^2|x-x_1|)]\chi(z|x-x_1|)
			\prod_{j=1}^3 V(x_j)\mR_L^+(z)(x_j,x_{j+1})\\ v^*(x_4)M_+^{-1}(z)(x_3,x_5)v(x_5)\prod_{j=5}^7\mR_L^+(z)(x_j,x_{j+1})V(x_{j+1})\mR_L^+(z)(x_8,y)\, d\vec x
		\end{multline}
		We note, from \eqref{eqn:RL low bds}, \eqref{eqn:RL diff bds} and Lemma~\ref{lem:Minv reg}, any term where the derivative acts on $\mR_L$ or $M_+^{-1}$ creates extra $z$ smallness so that the resulting function is of order $z^{1+}$ as $z\to 0$.  The only way to obtain non-zero boundary terms is when the derivative acts on the leading $z^2$ portion of $\mR_L^+-\mR_L^-$.  In all computations we ignore the case when the derivative acts on cut-offs.  In this case, we have
		\begin{multline*}
			\bigg|\lim_{z\to 0^+}\frac{1}{z}\partial_z(\cR_L^+-\cR_L^-) \Omega^+_L \cR_L^+(z)(x,y)\bigg|\\
			=\lim_{z\to 0^+}\bigg|\int_{\R^{32}}\chi(z|x-x_1|)
			\prod_{j=1}^3 V(x_j)\mR_L^+(z)(x_j,x_{j+1}) v^*(x_4)M_+^{-1}(z)(x_3,x_5)\\
			v(x_5)\prod_{j=5}^7\mR_L^+(z)(x_j,x_{j+1})V(x_{j+1})\mR_L^+(z)(x_7,y)\, d\vec x\bigg|\\
			\les \sum \int_{\R^{32}} \prod_{j=1}^3 \frac{\la x_j\ra^{-\delta}}{|x_j-x_{j+1}|^{k_j}} \la x_4\ra^{-\delta/2} |D_0(x_4,x_5)| \la x_5\ra^{-\delta/2}
			\prod_{j=5}^7 \frac{\la x_{j+1}\ra^{-\delta}}{|x_j-x_{j+1}|^{k_j}} \frac{1}{|x_8-y|^{k_8}}\, d\vec x,
		\end{multline*}
		where the sum is taken over $k_j\in \{2,3\}$.  By Lemmas~\ref{lem:abstractSpatialIntegrals} and \ref{d0bounded}, these integral converge uniformly in $x,y\in \mathbb R^4$.  This suffices to show that the boundary term in \eqref{eqn:Tail RLdiff front} satisfies the $|t|^{-2}$ bound.
		
		We note that $\Omega_L^+\mR_L^+$ is a product of integral kernels of $\mR_L^+$, potentials, and $M_+^{-1}$.  Let $\mathcal I=\{1,2,3,5,6,7\}$, then the integral kernel of $\Omega_L^+\mR_L^+$ is of the form
		\begin{align}\label{eqn:OmRL prod def}
			\Omega_L^+\mR_L^+(z)(\vec x)=
			\prod_{j\in \mathcal I} \mR_L^+(z)(x_j,x_{j+1})
			[v^*(x_4)M_+^{-1}(z)(x_4,x_5)v(x_5)] \prod_{i\in\mathcal I} V(x_i)
		\end{align}
		As in the preceding argument by \eqref{eqn:RL low bds}, Lemma~\ref{lem:Minv reg} and Lemma~\ref{lem:spatial estimates}, we see that that
		\begin{align}\label{eqn:OmRL prod 0der}
			\sup_{y\in \mathbb R^4}  \sup_{0<z<z_1}|\Omega_L^+\mR_L^+(z)(x_1,y)| \les |V(x_1)|.
		\end{align}
		We now seek to control the first two derivatives in $z$ of this term.  We note that a consequence of \eqref{eqn:Minv reg} in Lemma~\ref{lem:Minv reg} is the following
		\begin{multline}\label{eqn:Minv sup}
			\left\|\sup_{0<z<z_1} |M_\pm^{-1}(z)|\,\right\|_{L^2\to L^2}+
			\left\|\sup_{0<z<z_1} z^{-1+} |\partial_z M_\pm^{-1}(z)|\,\right\|_{L^2\to L^2} \\
			+\left\|\sup_{0<z<z_1} z^{0+} |\partial_z^2 M_\pm^{-1}(z)|\,\right\|_{L^2\to L^2}<\infty.
		\end{multline}
		Denote $\mathcal I_k=\mathcal I\setminus \{k\}$.
		From \eqref{eqn:OmRL prod def} and the product rule, we have
		\begin{multline*}
			\partial_z[\Omega_L^+\mR_L^+](z)(\vec x)\\
			= \sum_{k\in \mathcal I} \partial_z \mR_L^+(z)(x_k,x_{k+1})
			\prod_{j\in \mathcal I_k} \mR_L^+(z)(x_j,x_{j+1})
			[v^*(x_4)M_+^{-1}(z)(x_4,x_5)v(x_5)] \prod_{i\in\mathcal I}V(x_i)\\
			+\prod_{j\in \mathcal I} \mR_L^+(z)(x_j,x_{j+1})
			[v^*(x_4)\partial_z M_+^{-1}(z)(x_4,x_5)v(x_5)] \prod_{i\in\mathcal I}V(x_i).
		\end{multline*}
		Combining \eqref{eqn:RL low bds}, \eqref{eqn:Minv reg}, and Lemma~\ref{lem:abstractSpatialIntegrals} we see that
		\begin{align}\label{eqn:OmRL prod der}
			\sup_{y\in \mathbb R^4}  \sup_{0<z<z_1}|z^{-1+} \partial_z[\Omega_L^+\mR_L^+](z)(x_1,y)| \bigg|\les |V(x_1)|.
		\end{align}
		Similarly, with $\mathcal I'=\mathcal I\setminus \{k_1,k_2\}$ we have \begin{multline*}
			\partial_z^2[\Omega_L^+\mR_L^+](z)(\vec x)\\
			= \bigg(\sum_{k_1,k_2\in \mathcal I} \partial_z \mR_L^+(z)(x_{k_1},x_{k_1+1}) \partial_z \mR_L^+(z)(x_{k_2},x_{k_2+1})
			\prod_{j\in \mathcal I'} \mR_L^+(z)(x_j,x_{j+1})
			[v^*M_+^{-1}v](z)(x_4,x_5)  \\
			+\sum_{k\in \mathcal I} \partial_z \mR_L^+(z)(x_k,x_{k+1})\prod_{j\in \mathcal I_k} \mR_L^+(z)(x_j,x_{j+1})
			[v^*(x_4)\partial_z M_+^{-1}(z)(x_4,x_5)v(x_5)] \\
			+\prod_{j\in \mathcal I} \mR_L^+(z)(x_j,x_{j+1})
			[v^*(x_4)\partial_z^2 M_+^{-1}(z)(x_4,x_5)v(x_5)]\bigg) \prod_{i\in\mathcal I},
		\end{multline*}
		where if $k_1=k_2$ instead of a product of $\partial_z \mR_L^+(z)(x_{k_1},x_{k_1+1}) \partial_z \mR_L^+(z)(x_{k_2},x_{k_2+1})$, we have $\partial_z^2 \mR_L^+(z)(x_{k_1},x_{k_1+1})$.  From this and \eqref{eqn:RL low bds}, \eqref{eqn:Minv reg}, and Lemma~\ref{lem:abstractSpatialIntegrals}, we have
		\begin{align}\label{eqn:OmRL prod 2der}
			\sup_{y\in \mathbb R^4}  \sup_{0<z<z_1}[z^{0+}| \partial_z^2\Omega_L^+\mR_L^+(z)(x_1,y)|  ]\les |V(x_1)|.
		\end{align}
		We now consider the integral term in \eqref{eqn:Tail RLdiff front}.  As in the boundary term, more care is required when the derivative acts on the leading $\mR_L^+-\mR_L^-$ to avoid a logarithmic singularity in the $z$ integral.  That is, we consider two cases
		$$
			\partial_z \big( [\mR_L^+-\mR_L^-]\Omega_L^+\mR_L^+\big)
			=\partial_z  [\mR_L^+-\mR_L^-]\big(\Omega_L^+\mR_L^+\big)
			+  [\mR_L^+-\mR_L^-]\partial_z\big(\Omega_L^+\mR_L^+\big)
		$$
		By \eqref{eqn:RL diff bds}, we have
		$$
			\frac{1}{z}\partial_z[\mR_L^+-\mR_L^-](z)(x,x_1)=\mG_1(x,x_1)+O_1(z^2|x-x_1|^2)\chi(z|x-x_1|)
		$$
		In particular, this shows that
		$$
			\bigg|\partial_z\bigg(\frac{1}{z}\partial_z[\mR_L^+-\mR_L^-]\bigg)(z)(x,x_1)\bigg|\les z|x-x_1|^2\chi(z|x-x_1|)
		$$
		The claimed time decay for this term follows by showing the following is bounded uniformly in $x,y\in \R^4$:
		\begin{multline}\label{eqn:OmegaL RLdiff int 2ibp}
			\bigg|\int_0^\infty e^{-it\sqrt{z^2+m^2}}\partial_z\bigg(\frac{\sqrt{z^2 + m^2}}{z}\partial_z\left(\chi(z)\left((\cR_L^+-\cR_L^-) \Omega_L^+\cR_L^+\right)\right)\bigg)\, dz\bigg|\\
			\les \int_0^\infty \chi(z)\bigg|
			\partial_z\bigg(\frac{\partial_z[\mR_L^+-\mR_L^-]}{z}\bigg)
			[\Omega_L^+\mR_L^+](z)\bigg| +\bigg|\frac{[\mR_L^+-\mR_L^-]}{z}\partial_z^2
			[\Omega_L^+ \mR_L^+](z)\bigg|\, dz\\
			+\int_0^\infty \bigg| \frac{1}{\sqrt{z^2+m^2}} \partial_z\left(\chi(z)\left((\cR_L^+-\cR_L^-) \Omega_L^+\cR_L^+\right)\right)\bigg|\, dz.
		\end{multline}
		By \eqref{eqn:OmRL prod 0der}, \eqref{eqn:OmRL prod der}, \eqref{eqn:OmRL prod 2der}, the computations above, and H\"older's inequality we see that
		\begin{multline*}
			|\eqref{eqn:OmegaL RLdiff int 2ibp}|\les \sup_{x,y\in \mathbb R^4} \int_0^\infty \int_{\R^{4}}\bigg[\sup_{0<z<z_1} |\Omega_L^+\mR_L^+(z)(x_1,y)|\bigg] 
			 z|x-x_1|^2 \chi(z|x-x_1|)\,dz \, dx_1 \\
			+\sup_{y\in \mathbb R^4}\int_0^\infty z^{1-}\chi(z) \int_{\R^{4}} \bigg[\sup_{0<z<z_1}|z^{-1+}\partial_z[\Omega_L^+\mR_L^+](z)(x_1,y)\bigg] \, dx_1
			\, dz\\
			+\sup_{x,y\in \mathbb R^4} \int_{\R^{4}}\int_0^\infty z^{0-}\chi(z) \bigg[ \sup_{0<z<z_1} z^{0+}|\partial_z^2\Omega_L^+\mR_L^+(z)(x_1,y)|\bigg] 
			 \, dz\, dx_1\les \|V\|_{L^1_{x_1}}\les 1,
		\end{multline*}
		uniformly in $x$.  Hence, we have $|\eqref{eqn:Tail RLdiff front}|\les |t|^{-2}$ uniformly in $x,y$.
		
		We note turn to the contribution of  $\mR_L^-[\Omega_L^+-\Omega_L^-]\mR_L^+$.  
		By \eqref{eqn:alg identity} and \eqref{eqn:OmRL prod def}, the argument above and Fubini's Theorem suffice to control all possibilities except when the +/- difference acts on the operators $M_{\pm}^{-1}(z)$.  In this case, we need to bound the contribution of		
		\begin{align}\label{eqn:tail Mdiff}
			\int_0^\infty e^{-it\sqrt{z^2+m^2}}\frac{z\chi(z)}{\sqrt{z^2+m^2}}\left(\cR_L^- (V\cR_L^-)^3 v^* (M_+^{-1}-M_-^{-1})v (\cR_L^+ V)^3\cR_L^+\right)(z)(x,y)\,dz.
		\end{align}
		As before, we may integrate by parts once without a boundary term, since \eqref{eqn:Minv pm diff} gives us that $[M_+^{-1}-M_-^{-1}](z)=z^2\Gamma_0+O_2(z^2+)$ as an absolutely bounded operator, with $\Gamma_0$ $z$-independent.
		\begin{multline*}
			-\frac{e^{-it\sqrt{z^2+m^2}}\sqrt{z^2 + m^2}}{(it)^2z}\partial_z\left(\chi(z)\left((\cR_L^- V)^3 \cR_L^- v^* (M_+^{-1}-M_-^{-1})v \cR_L^+ (V \cR_L^+)^3 \right)\right)(z)(x,y)\bigg\vert_{z = 0} \\
			+ \frac{1}{(it)^2}\int_0^\infty e^{-it\sqrt{z^2+m^2}}
			\partial_z\bigg(\frac{\sqrt{z^2 + m^2}}{z}
			\partial_z\big(\chi(z)\big((\cR_L^- V)^3\cR_L^- v^*\\ 
			(M_+^{-1}-M_-^{-1})v \cR_L^+ (V \cR_L^+)^3\big)\big)\bigg)(z)(x,y)\,  dz .
		\end{multline*}
		First, we notice that \eqref{eqn:RL low bds}, along with Lemma~\ref{lem:abstractSpatialIntegrals} shows that (for $k=1,2$)
		\begin{align}\label{eqn:iter RL L2}
			\sup_{y\in \R^4}\|[v \cR_L^\pm (V \cR_L^\pm)^3(z)(\cdot,y)]\|_{L^2} \les 1,\quad
			\sup_{y\in \R^4}\| \partial_z^k [v \cR_L^\pm (V \cR_L^\pm)^3(z)(\cdot,y)]\|_{L^2} \les z^{2-k-}.
		\end{align}
		We first handle the boundary term. By \eqref{eqn:iter RL L2} and \eqref{eqn:Minv pm diff}, we see the only contribution occurs when the derivative acts on $(M_+^{-1} - M_-^{-1})(z)$.  Accordingly, we need only show the limiting value as $z\to 0^+$ of the boundary term is bounded.
		\begin{align*}
			\int_{\R^{32}}e^{-it\sqrt{z^2+m^2}}\frac{\sqrt{z^2 + m^2}}{z}\chi(z)\left((\cR_L^- V)^3 \cR_L^- v^* \partial_z(M_+^{-1}-M_-^{-1})v \cR_L^+ (V \cR_L^+)^3 \right)d\vec{x} \bigg\vert_{z = 0} \\
			\lesssim  \left\|(\cR_L^- V)^3 
			\cR_L^-v^*\right\|_{L^2}
			\left\| \sup_{0<z<z_1} z^{-1}|\partial_z(M_+^{-1}-M_-^{-1})|\right\|_{L^2 \to L^2}
			\left\|v \cR_L^+ (V \cR_L^+)^3 \right\|_{L^2} 
			\lesssim 1,
		\end{align*}
		uniformly in $x,y$ by Lemma~\ref{lem:abstractSpatialIntegrals}, \eqref{eqn:iter RL L2} and \eqref{eqn:Minv pm diff}. 
		
		Next, we consider the integral as a sum of terms of the form
		\begin{multline}\label{eqn:Minv diff RL case}
			\sum_{k_j} \int_0^\infty e^{-it\sqrt{z^2+m^2}}\chi(z)
			\partial_z^{k_1}\bigg(\sqrt{z^2 + m^2}  (\cR_L^- V)^3\cR_L^- v^*\bigg)\\ \partial_z^{k_2}\bigg(\frac{\partial_z^{k_3}(M_+^{-1}-M_-^{-1})}{z}\bigg) 
			\partial_{z}^{k_4}\big(v \cR_L^+ (V \cR_L^+)^3  \big)(z)(x,y)\,  dz ,
		\end{multline}
		where $k_j\in\{0,1,2\}$, $k_2,k_3\in \{0,1\}$, and $\sum k_j=2$.  The derivative is harmless if it acts on the cut off or $\sqrt{z^2+m^2}$.  We need only show that we can bound the derivative terms by a function of $z$ that is integrable in a neighborhood of zero uniformly in $x,y$.  
		
		From \eqref{eqn:Minv pm diff}, since $0\leq k_2+k_3\leq 2$, we have
		$$
			\left\|   \bigg| \sup_{0<z<z_1}z^{\alpha(k_2,k_3)} \partial_z^{k_2}\bigg(\frac{\partial_z^{k_3}(M_+^{-1}-M_-^{-1})}{z}\bigg)\bigg|
			\right\|_{L^2\to L^2}\les 1, \qquad \alpha(k_2,k_3)=\left\{\begin{array}{ll}
				0 & k_2+k_3<2\\
				1- & k_2+k_3=2
				\end{array} \right..
		$$
		Now, using \eqref{eqn:iter RL L2} and \eqref{eqn:Minv diff RL case}, we see
		\begin{multline*}
			|\eqref{eqn:Minv diff RL case}| \les \sum_{k_j} \int_0^1 \left\|\partial_z^{k_1}\bigg(\sqrt{z^2 + m^2}  (\cR_L^- V)^3\cR_L^- v^*\bigg)\right\|_{L^2} \\ \left\| \sup_{0<z<z_1}z^{\alpha(k_2,k_3)} \bigg| \partial_z^{k_2}\bigg(\frac{\partial_z^{k_3}(M_+^{-1}-M_-^{-1})}{z}\bigg)\bigg|
			\right\|_{L^2\to L^2}
			\left\|\partial_{z}^{k_4}\big(v \cR_L^+ (V \cR_L^+)^3  \big)\right\|_{L^2}\,  dz \\
			\les \int_0^1 z^{-1+}\, dz\les 1,
		\end{multline*}
		uniformly in $x,y\in \mathbb R^4$.
	\end{proof}
	
	We now extend the argument from Lemma~\ref{lem:allLow} to control the contribution to the Stone's formula when we replace the inner low resolvents $\mR_L$ with the $\mR_0$.  Roughly speaking, when any inner resolvent $\mR_0$ contributes $\mR_H$, we no longer need to utilize the difference between the `+' and `-' terms as the integrand is no longer supported in a neighborhood of $z=0$.  The argument here is less delicate, and benefits from the arguments in the proof of Lemma~\ref{lem:allLow}.
	
	\begin{proof}[Proof of Lemma~\ref{lem:NotAllLowLemma1}]
		If all inner resolvents contribute $\mR_L$ to $\Omega^\pm(z)$, then we are done by Lemma \ref{lem:allLow}. Otherwise, at least one resolvent is $\mR_H$ and after expanding the spatial variable dependence, we have we have $z|x_j-x_{j+1}|\gtrsim 1$ for some $j\in \{1,2,3,5,6,7\}$ and we no longer gain by considering the difference of the $+$ and $-$ terms, so we seek to bound
		\begin{equation*}
			\int_0^\infty  e^{-it\sqrt{z^2+m^2}}\frac{z\chi(z)}{\sqrt{z^2+m^2}}\cR_L^+(z)(V\cR_0^+(z))^3v^*M_+^{-1}v(\cR_0^+(z)V)^3 \cR_L^+(z)(x,y) \,\, dz.
		\end{equation*}
		Without loss of generality, we consider the '+` resolvents.
		Here one of the $\mR_0^+$ is $\mR_H^+$, by \eqref{eqn:RH defn} we have
		\be\label{eqn:RH in Marg}
			|\partial_z^k \mR_H(z)(x,y)|\les \frac{z^{\f12}}{|x-y|^{\f32-k}}.
		\ee
		Effectively, we may obtain further smallness in $z$ at the cost of growth in spatial variables.  Since this occurs only for an inner resolvent, this is easily controlled by the decay of the potentials on either side.  Similar to \eqref{eqn:iter RL L2}, \eqref{eqn:RL low bds}, \eqref{eqn:RH in Marg} 	
		and Lemma~\ref{lem:abstractSpatialIntegrals} show that
		\begin{align}\label{eqn:iter RLR0 L2}
			\sup_{y\in \R^4}\|[v (\cR_0^+ V)^3 \cR_L^+(z)(\cdot,y)]\|_{L^2} \les 1,\quad
			\sup_{y\in \R^4}\| \partial_z^k [v (\cR_0^+ V)^3 \cR_L^+(z)(\cdot,y)]\|_{L^2} \les z^{2-k-}.
		\end{align}
		This alone is not enough to integrate by parts twice in $z$ since we haven't yet accounted for the assumption that one resolvent is $\mR_H$.  Without loss of generality, we consider the case when the resolvent closest to $v$ is $\mR_H$, as this case constrains the amount of decay needed on the potential.  In this case we note that
		\begin{multline*}
			|\partial_z^k [v\mR_H^+ V(\cR_0^+ V)^2 \cR_L^+(z)(x,y)]|\\
			\les \sum |v(x)|\int_{\R^{12}} z^{\f12}|x-x_1|^{k_1-\f32} |V(x_1)| |\partial_z^{k_2}\big(\mR_0^+(z)V \big)^2(x_1,x_2,x_3)| |\partial_z^{k_3}\mR_L^+(z)(x_3,y)|
		\end{multline*}
		where the sum is taken of non-negative $k_j$ with $k_1+k_2+k_3=k$.  We consider only the case of $k\in\{0,1,2\}$,
		applying \eqref{eqn:R0 for bs} for the inner resolvents we have and \eqref{eqn:RL low def} and \eqref{eqn:RL low bds} to bound this by
		\begin{multline*}
			\sum \la x\ra^{-\f52-}\int_{\R^{12}} z^{\f12}|x-x_1|^{k_1-\f32} \\
			\prod_{j=1}^2 \la x_j\ra^{-5-} \bigg(\frac{\delta_{k_j}}{|x_j-x_{j+1}|^3}
			+(1-\delta_{k_j})z^{-2-k_j-}(1+|x_j-x_{j+1}|^{-2}) +\frac{z^\f12}{|x_j-x_{j+1}|^{\f32-k_2}}\bigg)\\
			\bigg(\frac{\delta_{k_3}}{|x_3-y|^3}+\frac{\delta_{k_3}}{|x_3-y|^2}+(1-\delta_{k_3})+z^{-2-k_3-}(1+|x_4-y|^{-2})\bigg) \, dx_1\, dx_2\, dx_3
		\end{multline*}
		where $\delta_{k_j}=1$ if $k_j=0$ and is zero otherwise. The case that constrains the required decay on the potential is when $k_1=2$, then using $|x-x_1|^\f12\les \la x\ra^\f12\la x_1\ra^{\f12}$ we have
		\begin{multline*}
			\la x\ra^{-2-}\int_{\R^{12}} z^{\f12}
			\prod_{j=1}^2 \la x_j\ra^{-\f72-} \bigg(\frac{1}{|x_j-x_{j+1}|^3}
		 	+\frac{z^\f12}{|x_j-x_{j+1}|^{\f32}}\bigg)
			\bigg(\frac{1}{|x_3-y|^3}+\frac{1}{|x_3-y|^2}\bigg) \, dx_1\, dx_2\, dx_3
		\end{multline*}
		By Lemma~\ref{lem:spatial estimates}, and noting that $0<z<1$, this is bounded by $z^{\f12}\la x\ra^{-2-}$ uniformly in $y$, which is just enough to be in $L^2_x(\R^4)$.  Applying Lemma~\ref{lem:abstractSpatialIntegrals} to the remaining cases, we see that		
		\begin{align}\label{eqn:iter RLRH L2}
			\sup_{y\in \R^4}\| \partial_z^k [v \mR_H^+(\cR_0^+ V)^2 \cR_L^+(z)(\cdot,y)]\|_{L^2} \les z^{\f12-}.
		\end{align}
		This, along with \eqref{eqn:Minv sup} suffices to show that we may integrate by parts twice with no boundary terms.  The resulting integrand is bounded uniformly in $x,y$ by $z^{-\f12-}$, which is integrable on the support of $\chi(z)$.  All other possibilities for the appearance of $\mR_H$ obey the same bounds and we have
		$$
			\sup_{x,y\in\R^4} \bigg|\int_0^\infty  e^{-it\sqrt{z^2+m^2}}\frac{z\chi(z)}{\sqrt{z^2+m^2}}\cR_L^+(z)(V\cR_0^+(z))^3v^*M_+^{-1}v(\cR_0^+(z)V)^3 \cR_L^+(z)(x,y) \, dz|\les \frac{1}{\la t\ra^2},
		$$
		as desired.
	\end{proof}
	
	We now consider the case when one of the leading or lagging resolvents is $\mR_H$.  Here we cannot safely integrate by parts twice as $|\partial_z^2\mR_H^\pm(z)(x,x_1)|\les z^\f12|x-x_1|^{\f12}$, whose contribution to the Stone's formula we cannot bound uniformly in $x$.  Here we need to utilize the oscillatory integral bounds in Lemma~\ref{lem:high stat phase} to obtain a time decay that is uniform in $x$.  In this regime we bound the `+' and `-' terms individually.
	\begin{lemma}\label{HighOnEdgeLemma}
		Under the assumptions of Proposition~\ref{thm:RegularCaseMainTheorem} we have the following bounds
		\begin{align*}
			\sup_{x, y \in \R^4}\bigg|\int_{0}^{\infty}e^{-it\sqrt{z^2+m^2}}\frac{z\chi(z)}{\sqrt{z^2+m^2}}\cR_H^\pm\Omega^\pm\cR_0^\pm(z)(x, y)dz\bigg| \lesssim \frac{1}{\la t\ra ^2}.
		\end{align*}
	\end{lemma}
	
	\begin{proof}
		We first consider the case when the lagging resolvent is $\mR_L$ and treat the `+' resolvents, the argument for the `-' resolvents is nearly identical.
		\begin{align*}
			\int_{0}^{\infty}e^{-it\sqrt{z^2 + m^2}}\chi(z)\cR_H^{+}(V\cR_0^{+})^3v^*M_{+}^{-1}v(\cR_0^{+}V)^3\cR_L^{+}(z)(x,y)\, dz.
		\end{align*}
		By \eqref{eqn:iter RLR0 L2}, \eqref{eqn:iter RLRH L2} and Lemma~\ref{lem:Minv reg} we can see that this is bounded uniformly in $x,y$.  Furthermore, the support of $\mR_H$ ensures that we may integrate by parts without boundary terms.
		
		After the first integration by parts, if the derivative acts on anything other than the leading $\mR_H$, the argument in Lemma~\ref{lem:NotAllLowLemma1} suffices to provide the desired bounds.  Accordingly, we need only consider
		\begin{multline*}
			\frac{1}{2it} \int_{0}^{\infty}e^{-it\sqrt{z^2 + m^2}}\chi(z)[\partial_z\cR_H^{+}](V\cR_0^{+})^3v^*M_{+}^{-1}v(\cR_0^{+}V)^3\cR_L^{+}\, dz   \\
			= \frac{1}{2it}\int_{0}^{\infty}\int_{\R^4}
			e^{-it\sqrt{z^2 + m^2} + iz|x - x_1|}\chi(z)e^{-iz|x - x_1|}\\
			([\partial_z\cR_H^{+}](V\cR_0^{+})^3v^*M_{+}^{-1}v(\cR_0^{+}V)^3\cR_L^{+} )(z,x_1,y)\, dz\, dx_1   .
		\end{multline*}
		Here we cannot integrate by parts again, but instead bound the integral with Lemma~\ref{lem:high stat phase}.  To this end, we define
		\begin{align*}
			a(z, x, x_1, y) = \chi(z)e^{-iz|x - x_1|}\partial_z\cR_H^{+}(z)(x,x_1)V(x_1)[ \cR_0^+(V\cR_0^+)^2v^*M_{+}^{-1}v(\cR_0^{+}V)^3\cR_L^{+} ](z,x_1,y).
		\end{align*}
		By \eqref{eqn:RH defn}, we have
		\begin{align*}
			\left|\partial_z^k\left(\chi(z)e^{-iz|x - x_1|}\partial_z\cR_H^{+}(z)(x,x_1)\right)\right| \lesssim \frac{z^{1 - k}\chi(z)\widetilde{\chi}(z|x - x_1|)}{(1+z|x - x_1|)^{\f12}}, \quad k = 0, 1.
		\end{align*}
		Furthermore, by \eqref{eqn:R0 bds bs helpful}, \eqref{eqn:R0 deriv in bs}, Lemmas~\ref{lem:abstractSpatialIntegrals}, Lemma~\ref{lem:Minv reg}, and \eqref{eqn:iter RLR0 L2}
		\begin{align*}
			|\partial_z^k[\cR_0^+(V\cR_0^+)^2v^*M_{+}^{-1}v(\cR_0^{+}V)^3\cR_L^{+}](z)(x_1,y)|\les 1 \qquad k=0,1
		\end{align*}
		uniformly in $z$, $x_1$ and $y$ on the support of $\chi(z)$.  Taken together, we have 
		$$
			|\partial_z^k a(z)|:=\bigg|\int_{\R^4} \partial_z^k a(z, x, x_1, y) \, dx_1\bigg|
			\les \int_{\R^4}\frac{z^{1 - k}\chi(z)\widetilde{\chi}(z|x - x_1|)}{(1+z|x - x_1|)^{\f12}} |V(x_1)|\, dx_1.
		$$
		Now, applying Lemma~\ref{lem:high stat phase}, and H\"older in the $x_1$ integral we have
		\begin{align*} 			
			\bigg|\frac{1}{t}\int_{\mathbb{R}^4}\int_{0}^{\infty}e^{-it\sqrt{z^2 + m^2} + iz|x - x_1|}a(z, x, x_1, y)dz \, dx_1\bigg| \les \frac{1}{t^2}\| V\|_1 \les \frac{1}{t^2},
		\end{align*}
		uniformly in $x,y$
		as desired.
	
		The remaining case when the last resolvent is $\mR_H$ follows a nearly identical argument.  If the derivative from the first integration by parts acts on the leading or lagging $\mR_H$, the argument above using Lemma~\ref{lem:high stat phase} provides the desired bound, while if the derivative acts anywhere in the middle, we may integrate by parts again safely as in Lemma~\ref{lem:NotAllLowLemma1}.
	\end{proof}
	
	Lemmas~\ref{lem:NotAllLowLemma1}, \ref{lem:allLow}, and \ref{HighOnEdgeLemma} suffice to prove Proposition~\ref{thm:RegularCaseMainTheorem}.  This, along with \eqref{eqn:BS} and Lemma~\ref{lem:born low} prove Theorem~\ref{thm:main} when the thresholds are regular.

	\section{Low Energy Dispersive Bounds when $m$ is not regular}\label{sec:non reg}
	
	We now consider the evolution of the solution operator in the low energy regime when a threshold is not regular.  As discussed in Definition~\ref{resondef} and proven in Section~\ref{sec:esa} below, there may be either threshold resonances or eigenvalues.  We develop a more unified inversion scheme here and show, in Proposition~\ref{prop:M inv res} below, that $M_{\pm}^{-1}(z)$ is far more singular than in the regular case as $z\to 0$.  The first bound for the low energy, when there is a resonance but no eigenvalue, paradoxically requires less decay on the potential since we don't need control on the second derivative in $z$ of the error term.  When there are eigenvalues, we need to assume further decay on the potential to ensure the boundedness of certain operators that arise in the longer expansions needed.  Throughout this section we assume that $V$ is self-adjoint and its entries satisfy $|V(x)|\les \la x\ra^{-\delta}$ from some $\delta>4$ and the threshold $\lambda=m$ is not regular.  In particular, the projection $S_1\neq 0$.

	When $\lambda=m$ is not regular, we utilize the inversion process of Jensen-Nenciu, see Lemma~2.1 in \cite{JN}.  From this, one has the formula
	\be\label{eqn:JN exp}
	M_\pm^{-1}(z)=(M^\pm(z)+S_1)^{-1}+(M^\pm(z)+S_1)^{-1}S_1 (B^\pm(z))^{-1}(M^\pm(z)+S_1)^{-1}.
	\ee
	We need to study the operator
	\be\label{eqn:Bdef}
	B^\pm(z)=S_1-S_1(M^\pm(z)+S_1)^{-1}S_1.
	\ee
	Accordingly, we need to prove that   
	\begin{align}\label{B defn}
		B^{\pm}(z)=S_1-S_1(M^\pm(z)+S_1)^{-1}S_1
	\end{align}
	has an inverse in $S_1L^2(\mathbb R^4)$. We have, using Lemma~\ref{lem:Minv reg} with $\delta>4$,  and the fact that $S_1D_0=D_0S_1=S_1$,
	\begin{align*}
		B^\pm (z)&= S_1-S_1(M^\pm (\lambda)+S_1)^{-1}S_1\\
		&=S_1-S_1 \bigg[ D_0+g_1^\pm(z)D_0T_1D_0+z^2D_0T_2D_0+O_1(z^{2+})\bigg] S_1\\
		&=-g_1^\pm(z) D_0+z^2 S_1T_1S_1+S_1O_1(z^{2+})S_1.
	\end{align*}
	To invert this operator on $S_1L^2(\R^4)$ we invoke the Feshbach inversion formula and write $B^\pm(z)$ in block form with respect to the decomposition $S_1L^2(\R^4)=QL^2(\R^4)\oplus S_2L^2(\R^4)$ where $QL^2(\R^4)$ is the finite-dimensional orthogonal complement to $S_2(L^2(\R^4))$.  By Corollary~\ref{cor:2d}, we note that $S_2L^2(\R^4)$ has dimension at most two.  Our approach is inspired by the treatment of the two dimensional massless Dirac equation, \cite{EGG2}, due to the similarity of the threshold resonance structure, though we require longer expansions for technical reasons.  As we show below, the four dimensional case has an at most two dimensional resonance space and a finite dimensional eigenspace at each threshold.  The logarithmic separation between terms in the expansion with respect to the spectral parameter complicates a repeated use of the Jensen-Nenciu inversion machinery.
	We write:
	\begin{align}\label{eqn:B defn}
		B^\pm(z)&=  A^\pm (z)-S_1O_1(z^{2+}) S_1,\\
		A^\pm(z)&=S_1v(g_1^\pm(z)\mG_{1}+z^2\mG_{2}  )v^*S_1. \label{eqn:A defn}
	\end{align}
	The exact formulation of the error term in \eqref{eqn:B defn} depends on which statement in Proposition~\ref{prop:M inv res} we use and how much decay we assume on $V$.
	We now turn to inverting $A^\pm(z)$ in a neighborhood of zero under different spectral assumptions.  Due to the need to preserve the logarithmic behavior in the spectral parameter, we denote $T(z)=O_1(f(z))$ to mean that $T$ is an absolutely bounded operator satisfying
	$$
	\big\|\sup_{0<z<z_1} |f(z)|^{-1} |T(z)|+\sup_{0<z<z_1} |f'(z)|^{-1} |\partial_z T(z)|\, \big\|_{L^2\to L^2}\les 1.
	$$
	
	\begin{prop}\label{lem:Ainverse}
		Assume that $|V(x)|\les \la x\ra^{-4-}$. For sufficiently small $z>0$, the operators $A^\pm(z)$ are invertible on $S_1L^2$.  Further, when $S_1\neq S_2$, there are $z$ independent, absolutely bounded operators $\Gamma_0,\Gamma_1$ that are independent of the choice of $\pm$ so that
		$$
		A^\pm(z)^{-1}=\frac{1}{z^2}[S_2v\mG_{2}v^*S_2]^{-1}+\frac{1}{g_1^\pm(z)}\Gamma_0+\frac{z^2}{(g_1^\pm(z))^2} \Gamma_1+ O_1\bigg(\frac{1}{z^2(\log z)^3} \bigg),	
		$$
		as an operator on $S_1L^2$.  Moreover 
		$$
		A^+(z)^{-1}-A^-(z)^{-1}=\bigg(\frac{1}{g^+(z)}-\frac{1}{g^-(z)}\bigg)\Gamma_0+ O_1\bigg(\frac{1}{z^2(\log z)^3} \bigg).
		$$
		Here $\Gamma_0$ and the error term are are finite rank operators.  We note that when $S_2=0$ the  operator corresponding to the most singular term in the spectral parameter is identically zero.	
		Furthermore, if $S_1=S_2$, we have
		$$
		A^\pm(z)^{-1}= \frac{1}{z^2}[S_2v\mG_{2}v^*S_2]^{-1},
		$$
		which is independent of the choice of $+$ or $-$.
		
	\end{prop}
	
	We note that these operators are finite rank on $L^2$ since $S_1L^2$ is a finite-dimensional subspace.  We note that the operator $S_2v\mG_2v^*S_2$ is always invertible on $S_2L^2(\R^4)$ by Lemma~\ref{lem:esa3} below.

	As above, we write projection $S_1=Q\oplus S_2$ where $Q$ is orthogonal to $S_2$. We note that by Lemmas~\ref{lem:esa1} and \ref{lem:esa2}, $Q$ is related to a projection onto the resonance space.  By Corollary~\ref{cor:2d}, $Q=S_1-S_2$  has rank at most two.  We first note that when $Q=0$, the statement follows \eqref{eqn:A defn} and the orthogonality property that $S_2v\mG_{1}=0$ in Lemma~\ref{cor:S2 orth} below.  
	The invertibility of the resulting operator is guaranteed by Lemma~\ref{lem:esa3}.  The following lemma implies the proposition when $S_2\neq0$.
	
	\begin{lemma}\label{lem:QAQ invert}
		
		When  $S_1\neq S_2$, that is $Q\neq 0$, the operators $QA^\pm(z)Q$ are invertible for sufficiently small $z>0$ and $\delta>4$.  Further, there are $z$-independent finite rank operators $\Gamma_0,\Gamma_1$ that are independent of the choice of $\pm$ so that
		$$
		(QA^\pm(z)Q)^{-1}=\frac{1}{g_1^\pm(z)} \Gamma_0+\frac{z^2}{(g_1^\pm(z))^2} \Gamma_1+ O_1\bigg(\frac{1}{z^2(\log z)^3} \bigg),	
		$$
		as an operator on $QL^2$.  Moreover 
		$$
		(QA^+(z)Q)^{-1}-(QA^-(z)Q)^{-1}
		=\bigg(\frac{1}{g^+(z)}-\frac{1}{g^-(z)}\bigg)\Gamma_0+
		O_1\bigg(\frac{1}{z^2(\log z)^3} \bigg).
		$$
		
	\end{lemma}
	
	\begin{proof}
		We begin by showing that $z^{-2}QA^\pm(z)Q$ is invertible on $QL^2$.  In the case that $Q$ has rank one,  then using \eqref{eqn:A defn} we can see that $z^{-2}QA^\pm(z)Q$ is a scalar of the form
		$$
		(c_1g_0^\pm(z)+c_2)Q, \qquad c_1\in \mathbb R\setminus \{0\},
		$$
		where $g_0(z)^\pm =a_1\log(z)+b_1^\pm$ where $a_1\in\mathbb R\setminus\{0\}$ and $b_1^\pm \in \mathbb C\setminus \R$ with $\overline{b_1^+}=b_1^-$ following from
		\eqref{g1 def}.  This suffices to show that claim.

		We now consider the case when $Q$ has rank two.
		We may select an orthonormal basis for $QL^2(\R^4)$, $\{\phi_1, \phi_2 \}$.  We note that $ \mG_{1}v^* \phi_1$ and $ \mG_{1}v^* \phi_2$ are linearly independent.  Assume   not, then by Lemma~\ref{lem:esa1}, $\psi_j=-\mG_{0} v^* \phi_j$, $j=1,2$.  
		Then for some $c$, (see \eqref{eq:g1g2} and \eqref{eq:psi2} below) we have
		$$ 
		\psi_1-c\psi_2=\mG_{0}v^*(c\phi_2-\phi_1)
		=\big(\mG_{0}+\frac{1}{ \la x\ra^2} \mG_{1}-\frac{1}{ \la x\ra^2} \mG_{1}\big)v^*(c\phi_2-\phi_1)\in L^2.
		$$
		By the proof of Lemma~\ref{lem:esa2}, $(\mG_0+\frac{1}{\la x\ra ^2}\mG_1)v^*\phi_j\in L^2$, that is $\psi_j=-\frac{1}{ \la x\ra^2} \mG_{1}v^*\phi_j(x)+O_{L^2}(1)$. Hence $\{\phi_1,\phi_2\}$ can only span a one-dimensional subspace of $QL^2$.  This proves our claim.
		
		Noting that, as a matrix, the kernel of $\mG_1$ is a constant multiple of $I_{uc}$.  Since $I_{uc}^2=I_{uc}$, writing $z^{-2}QA^\pm (z)Q$ with respect to the  basis $\{\phi_1, \phi_2 \}$:
		$$
		z^{-2}QA^\pm(z)Q=g_0^\pm(z) \left[\begin{array}
			{ll}|\mG_{1}v^*\phi_1|^2_{\mathbb C^2} & \la \mG_{1}v^*\phi_1, \mG_{1}v^*\phi_2\ra_{\mathbb C^2} \\
			\overline{\la \mG_{1}v^*\phi_1, \mG_{1}v^*\phi_2\ra_{\mathbb C^2}}
			& |\mG_{1}v^*\phi_2|^2_{\mathbb C^2} \end{array}
		\right] + A_1:=g_0^\pm(z)A_0+A_1,
		$$
		where $A_1$ is a $2\times 2$ matrix of constants given by the contributions of $\phi_i v\mG_{2}v^* \phi_j$. Since $ \mG_{1}v^* \phi_1$ and $ \mG_{1}v^* \phi_2$ are linearly independent $A_0$, the first matrix above, is invertible and hence, for sufficiently small $z>0$, $z^{-2}QA^\pm(z)Q$ is invertible. By a Neumann series expansion, one has
		$$
		[z^{-2}QA^\pm(z)Q]^{-1}=\frac{1}{g_0^\pm(z)}A_0^{-1}+\frac{1}{(g_0^\pm(z))^2}A_0^{-1}A_1A_0^{-1}+O_1\bigg(\frac{1}{(\log z)^3}\bigg).
		$$
		This allows us to deduce the bounds on $(QA^{\pm}(z)Q)^{-1}$ by dividing through by $z^2$ recalling that $g_1^\pm(z)=z^2g_0^\pm(z)$.
		The final claim follows from the observation that $(g_1^+(z)-g_1^-(z))=c z^2$, which follows from \eqref{g1 def}.
	\end{proof}
	
	We now show that $A^\pm(z)$ are invertible on the larger space $S_1L^2$.

	\begin{proof}[Proof of Proposition~\ref{lem:Ainverse}]	
		We now consider the case when both $Q,S_2\neq 0$, as the case when $S_2=0$ follows immediately from Lemma~\ref{lem:QAQ invert}.
		We  employ the Feshbach formula, see for example Lemma~2.3 in \cite{JN}.  If $A(z)=\left[\begin{array}{ll} a_{11} & a_{12}\\ a_{21} & a_{22}\end{array} \right]$, the invertibility of $A(z)$ follows if both $a_{22}$ is invertible and  $a:=(a_{11}-a_{12}a_{22}^{-1}a_{21})^{-1}$ exists.  Then, we have
		\begin{align}\label{eqn:Ainv Fesh}
			A(z)^{-1}=\left[
			\begin{array}{ll}
				a & -aa_{12}a_{22}^{-1}\\
				-a_{22}^{-1}a_{21}a & a_{22}^{-1}a_{21} a a_{12} a_{22}^{-1}+a_{22}^{-1}
			\end{array}
			\right].
		\end{align}
		In our case we write the operators $A^\pm(z)$ in block form with respect to the decomposition $S_1L^2=QL^2\oplus S_2 L^2$, then
		$a_{22}=S_2v\mG_{2}v^*S_2$ which is invertible by Lemma~\ref{lem:esa3}.  In our case we have
		$$
		a=\big(  QA^\pm(z) Q -Qv\mG_{1}v^* S_2(S_2v\mG_{2}v^* S_2)^{-1} S_2 v\mG_{1}v^* Q
		\big)^{-1},
		$$
		exists for sufficiently small $z$ since $QA^\pm(z)Q$ is invertible by Lemma~\ref{lem:QAQ invert}, while the second summand is a $z$ independent $2\times2$ matrix.  We note that any portion of \eqref{eqn:Ainv Fesh} containing $a$ has a projection $Q$; which makes its contribution to $A(z)^{-1}$ a finite rank operator.  Accordingly we have,
		$$
		A^\pm(z)^{-1}=\frac{1}{z^2}[S_2v\mG_{2}v^*S_2]^{-1}+\frac{1}{g_1^\pm(z)} \Gamma_{Q,0}+\frac{z^2}{(g_1^\pm(z))^2} \Gamma_{Q,1}+O_1\bigg(\frac{1}{z^2(\log z)^3} \bigg)\Gamma_{Q,0},
		$$
		where $\Gamma_{Q,j}$ are absolutely bounded and $z$ independent.  The claim about finite rank follows since each summand in the expansion is of finite rank.
		
	\end{proof}	
	
	We now argue that the operator $B^\pm(z)$ is invertible for small $z>0$ by considering it as a perturbation of $A^\pm(z)$.
	
	\begin{prop}\label{prop:Binv}
		
		Assume $|V(x)|\les \la x\ra^{-\delta}$ for $\delta >4$.  Then the operators $B^\pm(z)$ are invertible for sufficiently small $z>0$.  When $S_1\neq 0$ and $S_2=0$ there are finite rank, $z$-independent operators $\Gamma_0, \Gamma_1$ that are independent of the choice of $\pm$ so that
		$$
		(B^\pm(z))^{-1}= \frac{1}{g_1^\pm(z)}\Gamma_0+\frac{z^2}{(g_1^\pm(z))^2} \Gamma_1+O_1\bigg(\frac{1}{z^2(\log z)^3}\bigg)S_1.
		$$ 
		Furthermore, 
		$$
		(B^+(z))^{-1}-(B^{-}(z))^{-1}=\bigg(\frac{1}{g^+(z)}-\frac{1}{g^-(z)}\bigg)\Gamma_0+O_1\bigg(\frac{1}{z^2(\log z)^3}\bigg)S_1.
		$$
		When $\delta>8$ and $S_1\neq S_2\neq 0$, we have
		\begin{multline*}
			(B^\pm(z))^{-1}=A^{\pm}(z)^{-1}+\frac{g_2^\pm(z)+z^2g_1^\pm(z)}{z^4}\Gamma_0+\frac{g_2^\pm(z)+z^2g_1^\pm(z)}{(g_1^\pm(z))^2}\Gamma_{Q,0}+\Gamma_1\\
			+\frac{z^4}{(g_1^\pm(z))^2}\Gamma_{Q,1}+\frac{z^2}{g_1^\pm(z)}\Gamma_{Q,2}+O_1(z^{0+}) .
		\end{multline*}
		Here operators $\Gamma_{Q,j}$ are independent of $z$ and $\pm$, and are finite rank.
		Finally, when $S_1=S_2\neq 0$, we have
		$$
		(B^\pm(z))^{-1}=-\frac{D_2}{z^2}+\frac{g_2^\pm(z)}{z^4}\Gamma_0+\Gamma_1++O_1(z^{0+}),
		$$
		for some finite rank, $z$  and $\pm$ independent operators $\Gamma_j$.
		
	\end{prop}
	
	\begin{proof}
		The first claim follows from \eqref{eqn:B defn} and Lemma~\ref{lem:QAQ invert} via a Neumann series argument, noting that if $S_2=0$ the most singular $z^{-2}$ contribution is identically zero.
		For the remaining claims, we note the longer expansion \eqref{eqn:M+S inv long} in Lemma~\ref{lem:Minv reg}, from which we have
		\begin{multline*}
			B^\pm(z)=S_1\\
			-S_1[D_0+g_1^\pm(z)D_0T_1D_0+z^2 D_0T_2D_0+(g_1^\pm(z))^2\Gamma_{4,2}
			+[g_2^\pm(z)+z^2g_1^\pm(z)]\Gamma_{4,1}
			+z^4\Gamma_{4,0}+O_1(z^{4+})]S_1\\
			=A^\pm(z)+(g_1^\pm(z))^2S_1\Gamma_{4,2}S_1
			+[g_2^\pm(z)+z^2g_1^\pm(z)]S_1\Gamma_{4,1}S_1
			+z^4S_1\Gamma_{4,0}S_1+O_1(z^{4+})S_1
		\end{multline*}
		When $S_1\neq S_2\neq 0$, we note utilize the resolvent identity $B^{-1}=A^{-1}+A^{-1}(B-A)A^{-1}$ and note that
		$$
		B^\pm(z)-A^\pm(z)=(g_1^\pm(z))^2S_1\Gamma_{4,2}S_1
		+[g_2^\pm(z)+z^2g_1^\pm(z)]S_1\Gamma_{4,1}S_1
		+z^4S_1\Gamma_{4,0}S_1+O_1(z^{4+})S_1
		$$
		Here it is important to note the explicit formulation of $\Gamma_{4,2}=D_0T_1D_0T_1D_0$ from the proof of Lemma~\ref{lem:Minv reg}, in particular $S_2\Gamma_{4,2}=\Gamma_{4,2}S_2=0$.
		Hence, 
		\begin{multline*}
			(B^\pm)^{-1}(z)=A^{\pm}(z)^{-1}+\frac{g_2^\pm(z)+z^2g_1^\pm(z)}{z^4}\Gamma_0+\frac{g_2^\pm(z)+z^2g_1^\pm(z)}{(g_1^\pm(z))^2}\Gamma_{Q,0}+\Gamma_1\\
			+\frac{z^4}{(g_1^\pm(z))^2}\Gamma_{Q,1}+\frac{z^2}{g_1^\pm(z)}\Gamma_{Q,2}+O_1(z^{0+})
		\end{multline*}
		Here the subscript $Q$ indicates that the projection $Q$ appears in the composition of operators, making the resulting operator finite rank.
		
		When $S_1=S_2$ there is substantial cancellation from Corollary~\ref{cor:S2 orth} and \eqref{eqn:M+S inv long}, from which we obtain
		$$
		B^\pm(z)=-z^2 S_1T_2S_1+g_2^\pm(z)S_1\Gamma_{4,1}S_1+z^4S_1\Gamma_{4,0}S_1+O_1(z^{4+}).
		$$
		$S_1T_2S_1$ is invertible by Lemma~\ref{lem:esa3}, via a Neumann series expansion we have
		\begin{align*}
			(B^\pm(z))^{-1}=-\frac{D_2}{z^2}+\frac{g_2^\pm(z)}{z^4}\Gamma_0+S_1\Gamma_{4,0}S_1+O_1(z^{0+}).
		\end{align*}
		
	\end{proof}
	
	In terms of threshold obstructions, when $S_1\neq S_2$ we have a threshold resonance, there may or may not be an eigenvalue.  When $S_1=S_2\neq 0$, there is an eigenvalue, but no threshold resonance.
	
	Putting this all together, we arrive at
	\begin{prop}\label{prop:M inv res}
		
		If $0\neq S_1$, $S_2=0$, and $|V(x)|\les \la x\ra^{-\delta-}$ for some $\delta>4$, then the operators $M^{\pm}(z)$ are invertible for sufficiently small $z>0$.  Furthermore there are $z$-independent and $\pm$ independent operators $\Gamma_j$  such that
		\begin{multline*}
			M_\pm^{-1}(z)=\frac{1}{g_1^\pm(z)}S_1\Gamma_0S_1+\frac{z^2}{(g_1^\pm(z))^2} \Gamma_1+ \frac{z^2}{g_1^\pm(z)}\Gamma_2+\Gamma_3 +O_1\bigg(\frac{1}{z^2(\log z)^3}\bigg)S_1+O_1\bigg(\frac{1}{(\log z)^2}\bigg).
		\end{multline*}
		When $S_1\neq S_2\neq 0$ and $\delta>8$, we have
		\begin{multline*}
			M_\pm^{-1}(z)=-\frac{D_2}{z^2}+\frac{1}{g_1^\pm(z)}\Gamma_{Q,1}+\frac{z^2}{(g_1^\pm(z))^2} \Gamma_{Q,2}+ \frac{z^2}{g_1^\pm(z)}\Gamma_{Q,3} +O_1\bigg(\frac{1}{z^2(\log z)^3}\bigg)\Gamma_{Q,4}\\
			+\Gamma_1
			+\frac{g_2^\pm(z)}{z^4}\Gamma_2+
			\frac{g_2^\pm(z)+z^2g_1^\pm(z)}{z^4}\Gamma_3
			+\frac{g_2^\pm(z)+z^2g_1^\pm(z)}{(g_1^\pm(z))^2}\Gamma_{Q,5}+\Gamma_4\\
			+\frac{z^4}{(g_1^\pm(z))^2}\Gamma_{Q,6}+\frac{z^2}{g_1^\pm(z)}\Gamma_{Q,7}+O_1(z^{0+}).
		\end{multline*}
		When $0\neq S_1=S_2$, provided $\delta>8$, we have
		$$
		M_\pm^{-1}(z)=-\frac{D_2}{z^2}+\Gamma_0+\frac{g_2^\pm(z)}{z^4}\Gamma_1 +O_1(z^{0+}).
		$$
		
	\end{prop}
	
	\begin{proof}

		In the case when $S_1\neq S_2$, using Lemma~\ref{lem:Minv reg} and Proposition~\ref{prop:Binv} in \eqref{eqn:JN exp} establish the claim.  For the first term in \eqref{eqn:JN exp}, we can truncate the expansion in Lemma~\ref{lem:Minv reg} to see it is of the form $D_0+O_1(z^{0})$.  The second term requires the longer expansion due to the singularities of $(B^\pm(z))^{-1}$.
		\begin{multline*}
			(M^\pm(z)+S_1)^{-1}S_1 (B^\pm(z))^{-1}S_1(M^\pm(z)+S_1)^{-1}\\
			=
			[D_0+g_1^\pm(z)D_0T_1D_0+z^2D_0T_2D_0+O_1(z^{2+})]S_1\bigg(\frac{1}{g_1^\pm(z)}\Gamma_0+\frac{z^2}{(g_1^\pm(z))^2} \Gamma_1+O_1\bigg(\frac{1}{z^2(\log z)^3}\bigg)S_1\bigg)S_1\\
			[D_0+g_1^\pm(z)D_0T_1D_0+z^2D_0T_2D_0+O_1(z^{2+})]\\
			=\frac{S_1\Gamma_0S_1}{g_1^\pm(z)}+S_1\Gamma_0S_1T_1D_0+D_0T_1S_1\Gamma_0S_1+\frac{z^2}{g_1^\pm(z)}[ S_1\Gamma_0S_1T_2D_0+D_0T_2S\Gamma_0S_1 ]\\
			+O_1\bigg(\frac{1}{z^2(\log z)^3}\bigg)S_1.
		\end{multline*}
		When $S_1=S_2$, many of the logarithmic terms are annihilated.  In particular, by Corollary~\ref{cor:S2 orth} we have $S_1v\mathcal G_1=\mathcal G_1vS_1=0$.  Using the longer expansion in \eqref{eqn:M+S inv long}, where we utilize the exact form of the $\Gamma_{4,j}$ operators given in the proof to see that $S_1\Gamma_{4,2}=\Gamma_{4,2}S_1=0$, while the other $\Gamma_{4,j}$ aren't in the kernel of $S_1$.  From this we have
		\begin{multline*}
			S_1(M^\pm(z)+S_1)^{-1}=S_1(D_0+g_1^\pm(z)D_0T_1D_0+z^2 D_0T_2D_0+(g_1^\pm(z))^2\Gamma_{4,2}\\
			+[g_2^\pm(z)+z^2g_1^\pm(z)]\Gamma_{4,1}
			+z^4\Gamma_{4,0}+O_1(z^{4+}))\\
			=S_1+z^2 S_1T_2D_0+g_2^\pm(z)S_1\Gamma_{4,1}+z^4S_1 \Gamma_{4,2}+O_1(z^{4+}).
		\end{multline*}
		Similarly,
		\begin{align*}
			(M^\pm(z)+S_1)^{-1}S_1
			=S_1+z^2 D_0T_2S_1 +g_2^\pm(z)\Gamma_{4,1}S_1+z^4 \Gamma_{4,2}S_1+O_1(z^{4+}).
		\end{align*}
		So the second term in \eqref{eqn:JN exp} is
		\begin{multline*}
			[ S_1+z^2 S_1T_2D_0+g_2^\pm(z)S_1\Gamma_{4,1}+z^4S_1 \Gamma_{4,2}+O_1(z^{4+})] \big[ -\frac{D_2}{z^2}+\frac{g_2^\pm(z)}{z^4}\Gamma_0+\Gamma_1++O_1(z^{0+})\big]\\
			[S_1+z^2 D_0T_2S_1 +g_2^\pm(z)\Gamma_{4,1}S_1+z^4 \Gamma_{4,2}S_1+O_1(z^{4+})]\\
			=-\frac{D_2}{z^2}+\Gamma_0+\frac{g_2^\pm(z)}{z^4}\Gamma_1+O_1(z^{0+}).
		\end{multline*}
		The case when $S_1\neq S_2\neq 0$ is similar to the previous cases using the longer expansion in Proposition~\ref{prop:Binv}.
	\end{proof}
	
	For our purposes, we also need to account for the difference between the + and - limiting operators.  Specifically, we have
	\begin{corollary}\label{cor:M diff}
		
		If $0\neq S_1$ and $S_2=0$ with $|V(x)|\les \la x\ra^{-\delta-}$ for some $\delta>4$, then 
		$$
			(M_-^{-1}-M_+^{-1})(z)=\bigg(\frac{1}{g^+(z)}-\frac{1}{g^-(z)}\bigg)S_1\Gamma_0S_1+
			O_1\bigg(\frac{1}{z^2(\log z)^3}\bigg)S_1+O_1\bigg(\frac{1}{(\log z)^2}\bigg),
		$$
		When $0\neq S_1=S_2$, provided $\delta>8$, we have
		$$
			(M_-^{-1}-M_+^{-1})(z)= \Gamma_0+O_1(z^{0+}).
		$$
		
	\end{corollary}
	
	For the sake of brevity, we omit the longer expansion when $S_1\neq S_2\neq 0$ from the statement and proof of this corollary, it is implicitly in Proposition~\ref{prop:M inv res}.
	\begin{proof}
		We first consider the case when $S_1\neq S_2$.
		Similar to \eqref{eqn:Minv pm diff}, we have
		$$
		(M^+(z)+S_1)^{-1}-(M^-(z)+S_1)^{-1}=z^2 \Gamma_0+O_2(z^{2+}).
		$$
		We recall the algebraic identity \eqref{eqn:alg identity} and combine it with the last two equalities with Propositions~\ref{prop:Binv} and \ref{prop:M inv res}, there are essentially two cases to consider.  First, if the +/- difference acts on $(M(z)+S_1)^{-1}$, one has
		\begin{multline*}
			[(M^+(z)+S_1)^{-1}-(M^-(z)+S_1)^{-1}](B^{-})^{-1}(z)(M^{-}(z)+S_1)^{-1}\\
			=\frac{z^2}{g_1^+(z)}\Gamma_1+\frac{z^2}{g_1^-(z)}\Gamma_2+O_1\bigg(\frac{1}{(\log(z))^2}\bigg),
		\end{multline*}
		for some absolutely bounded operators $\Gamma_j$. The remaining case is when the difference is on the operators $B^{-1}(z)$, then by Propositions~\ref{prop:Binv} and \ref{prop:M inv res} we have
		\begin{multline*}
			(M^+(z)+S_1)^{-1}S_1[(B^{+})^{-1}-(B^{-})^{-1}](z)S_1(M^{-}(z)+S_1)^{-1}\\
			=
			[D_0+g_1^+(z)D_0T_1D_0+O_2(z^{2})]S_1\bigg[\bigg(\frac{1}{g^+(z)}-\frac{1}{g^-(z)}\bigg)\Gamma_0+O_1\bigg(\frac{1}{z^2(\log z)^3}\bigg)S_1\bigg]\\
			\qquad\qquad
			S_1[D_0+g_1^-(z)D_0T_1D_0+O_2(z^{2})]\\
			=O_1\bigg(\frac{1}{z^2(\log z)^2}\bigg)S_1\Gamma_1S_1+O_1(z^{0+}).
		\end{multline*}
		Putting these together, establishes the claim when $S_2=0$.
		
		Next, when $S_1=S_2$, we note that Corollary~\ref{cor:S2 orth} tells us that $S_2v\mG_1=S_1v\mG_1=0$ and Definition~\ref{resondef} tells us that $S_2D_0=S_1D_0=0$ in this case.  From Proposition~\ref{prop:Binv} and the computations above, we have
		\begin{multline*}
			[(M^+(z)+S_1)^{-1}-(M^-(z)+S_1)^{-1}]S_1(B^{-})^{-1}(z)S_1(M^{-}(z)+S_1)^{-1}\\
			=[z^2\Gamma_0+O_2(z^{2+}) ]S_1\bigg( \frac{-D_2}{z^2} +O_1\bigg(\frac{1}{\log(z)}\bigg) \bigg)S_1[D_0+g_1^-(z)D_0T_1D_0]=O_1(z^{2-}).
		\end{multline*}
		Here we note that the leading $D_0$ on the right is in the kernel of $S_1=S_2$.  On the other hand, 
		\begin{multline*}
			(M^+(z)+S_1)^{-1}S_1[(B^{+})^{-1}-(B^{-})^{-1}](z)(S_1M^{-}(z)+S_1)^{-1}\\
			=	[D_0+g_1^+(z)D_0T_1D_0+O_2(z^{2})]S_1(O_1(z^0))S_1[D_0+g_1^-(z)D_0T_1D_0+O_2(z^{2})]\\
			=S_1\Gamma_0S_1+O_1(z^{2-} ).
		\end{multline*}

	\end{proof}

	Having developed the appropriate expansions of $M_\pm^{-1}(z)$ in a neighborhood of the threshold, we now seek to prove the dispersive bounds in Theorem~\ref{thm:main}. We now show that the contribution of these new singular terms to the spectral measure in the Stone's formula, \eqref{eqn:stone}, may be bounded appropriately. As usual, we reduce to oscillatory integral estimates.
	
	\begin{lemma}\label{lem:osc log}
		
		Let $\mathcal E(z)$ be supported on $(0,z_1)$ for some $0<z_1\ll 1$ satisfying
		$|\mathcal E(z)|\les z^{-2}(\log z)^{-2}$ and $|\mathcal E'(z)|\les z^{-3}(\log z)^{-2}$.  Then, for $t>2$ we have the following bound
		$$
		\bigg|\int_0^\infty e^{-it\sqrt{z^2+m^2}}\frac{z}{\sqrt{z^2+m^2}}\mathcal E(z)\, dz\bigg| \les \frac{1}{\log(t)}.
		$$
		
	\end{lemma}
	
	\begin{proof}
		
		We break the integral up into two domains, when $0<z<t^{-\frac12}$ and when $z>t^{-\frac12}	$.  For the first region, the oscillation cannot be used instead the triangle inequality suffices to establish the bound
		\begin{align*}
			\bigg|\int_0^{t^{-\f12}} e^{-it\sqrt{z^2+m^2}}\frac{z}{\sqrt{z^2+m^2}}&\mathcal E(z)\, dz\bigg|\les \int_0^{t^{-\f12}} |\mathcal E(z)|\, dz\les \int_0^{t^{-\f12}} \frac{1}{z(\log z)^2}\, dz\les \frac{1}{\log t}.
		\end{align*}
		For the region supported away from zero, to utilize the oscillation we integrate by parts once
		\begin{align*}
			\int_{t^{-\f12}}^\infty e^{-it\sqrt{z^2+m^2}}\frac{z}{\sqrt{z^2+m^2}}\mathcal E(z)\, dz
			&=\frac{e^{-it\sqrt{z^2+m^2}}\mathcal E(z)}{-it}\bigg|^{\infty}_{t^{-\f12}}+\frac{1}{it}\int{t^{-\f12}}^\infty e^{-it\sqrt{z^2+m^2}} \mathcal E'(z)\, dz.
		\end{align*}
		By the assumptions on $\mathcal E(z)$, the boundary term's contribution is dominated by $t^{-1}|\mathcal E(t^{-\f12})|\les (\log t)^{-2}$.  For the second integral, we further divide the region of integration into regions where $|\log z|\approx \log t$ and where $|\log z|\gtrsim 1$.  By the assumption that $t>2$ and the support of $\mathcal E(z)$, $|\log z|$ is a decreasing function on the region considered and we have
		\begin{align*}
			\bigg|\frac{1}{it}\int_{t^{-\f12}}^\infty e^{-it\sqrt{z^2+m^2}} \mathcal E'(z)\, dz\bigg| &\les \frac{1}{t}\int_{t^{-\f12}}^{t^{-\f14}} \frac{1}{z^3(\log z)^2}\, dz+\frac{1}{t}\int_{t^{-\f14}}^{\f12}\frac{1}{z^3}\, dz\\
			&\les \frac{1}{t(\log t)^2}\int_{t^{-\f12}}^{t^{-\f14}}\frac{dz}{z^3} +\frac{1}{t}\int_{t^{-\f14}}^\infty \frac{dz}{z^3}\les \frac{1}{(\log t)^2}+\frac{1}{t^{\f12}} \les \frac{1}{\log t}.
		\end{align*}
		Extending the upper limit of integration of the second integral is harmless for our desired bound which is constrained by its behavior on $0<z<t^{-\f12}$ above.	 
	\end{proof}
	
	This estimate is needed for the time decay in  the finite-rank operator $F_t$ in Theorem~\ref{thm:main}.  For the $\la t\ra^{-1}$ error term in Theorem~\ref{thm:main}, we note that that if $\mathcal E(z)$ is supported on a small neighborhood of zero with $|\mathcal E(z)|\les \frac{1}{\log (z)}$ and $|\mathcal E'(z)|\les \frac{1}{z(\log z)^2}$, then
	$$
	\bigg|\int_0^\infty e^{-it\sqrt{z^2+m^2}}\frac{z}{\sqrt{z^2+m^2}}\mathcal E(z)\, dz\bigg| \les \frac{1}{\la t \ra}.
	$$
	The integral is clearly bounded by the triangle inequality and assumption on $\mathcal E(z)$.  The time decay follows by a single integration by parts.

	\begin{proof}[Proof of Theorem~\ref{thm:main} when $m$ is not regular]
		As usual, we employ a Born Series expansion to write
		$$
		\mR_V^\pm=\sum_{j=0}^{7}\mR_0^\pm(-V\mR_0^\pm)^j
		+(\mR_0^\pm V)^3\mR_0^\pm v^*M_\pm^{-1}(z)v(\mR_0^\pm V)^3 \mR_0^{\pm}
		$$
		From the treatment in the case when the threshold is regular, we need only consider the final term. Our goal is to control its contribution to the Stone's formula, which is
		\begin{align}\label{eqn:J goal}
			\int_0^\infty e^{-it\sqrt{z^2+m^2}} \frac{z\chi(z)}{\sqrt{z^2+m^2}} \mathcal J(z)(x,y)\, dz,
		\end{align}
		where
		\begin{align*}
			\mathcal J(z)= [(\mR_0^+ V)^3\mR_0^+ v^*M_{+}^{-1}v(\mR_0^+ V)^3 \mR_0^{+}
			-(\mR_0^- V)^3\mR_0^- v^*M_{-}^{-1}v(\mR_0^- V)^3 \mR_0^{-}](z).
		\end{align*}
		For the moment we suppress the dependence on the spatial variables.  The uniform boundedness of the spatial integrals follows as in the regular case since the iterated resolvents and their first derivative on either side are $L^2$ uniformly in $x,y$.
		
		We first consider the case when $S_1\neq 0$ and $S_2=0$, that is when there is only a threshold resonance.  
		By \eqref{eqn:alg identity}, it suffices to consider two cases.  Either the +/- difference acts on a free resolvent, or it acts on $M^{-1}$.  In the first case, we consider 
		\begin{align*}
			\int_0^\infty e^{-it\sqrt{z^2+m^2}} \frac{z\chi(z)}{\sqrt{z^2+m^2}}[\mR_0^+-\mR_0^-](V\mR_0^+)^3  v^*M_{+}^{-1}v(\mR_0^+ V)^3 \mR_0^{+}](z)(x,y)\, dz.
		\end{align*}
		The other cases, the $\mR_0^+-\mR_0^-$ is an inner resolvent are similar and follow from the argument we present here.  It suffices to show that we can integrate by parts in $z$ once here to obtain the desired time decay.  
		
		By Lemma~\ref{lem:R0exp}, we note that (in terms of the spectral parameter $z$) $\mR_0-\mG_0=O_1(z^{2-})$ with no growth in the spatial variables.  By replacing each $\mR_0^\pm(z)$ with $\mG_0+\mathcal E_0^\pm(z)$, any term of the expansion with contribution of multiple $\mathcal E_0^\pm(z)$ yields an integrand that is $O_1(z)$ with respect to the spectral parameter and, as above, can be integrated by parts once in $z$ to obtain a bound on the $z$ integral of size $\la t\ra^{-1}$.  As in the analysis of the regular case, the iterated integrals in the spatial variables is bounded uniformly in $x,y$.  Hence, it suffices to consider the leading terms. 
		
		Here we need to adjust the approach depending on whether the difference of the resolvents has a large or small argument.  By Lemma~\ref{lem:R0exp}, \eqref{eqn:RL diff bds} and the discussion directly after it,
		\begin{multline}\label{eqn:Ro dif res}
			[\mR_0^+-\mR_0^-](z)(x,x_1)=[cz^2\mG_1(x,x_1)+ O_1(z^4|x-x_1|^{2})]\chi(z|x-x_1|)\\
			+\widetilde \chi(z|x-x_1|)e^{\pm i z|x-x_1|}  O_1(z^{\f12}|x-x_1|^{-\f32}),\
		\end{multline}
		We first consider the contribution of the first summand.
		By Proposition~\ref{prop:M inv res}, the most singular term is of the form
		$$
		\bigg[\frac{cz^2 \chi(z|x-x_1|)}{g_1^+(z)}\mG_1(x,x_1)+
		O_1(z^2 |x-x_1|^2)\chi(z|x-x_1|) \bigg] (V\mG_0)^3v^*S_1\Gamma_0S_1 v(\mG_0V)^3\mG_0.
		$$
		Where we use that $(\log z)^{-1}\les 1$ on the support of $\chi(z)$ for the second term.
		Hence, from \eqref{eq:G1 def} it suffices to control
		\begin{align*}
			\int_0^\infty e^{-it\sqrt{z^2+m^2}} \frac{z\chi(z)}{\sqrt{z^2+m^2}}\bigg[ \frac{cz^2}{g_1^-(z)} + O_1(z |x-x_1|) \bigg]\chi(z|x-x_1|)\, dz.
		\end{align*}
		
		By the discussion following Lemma~\ref{lem:osc log}, we have $\mathcal E(z)= \frac{cz^2}{g_1^-(z)}$, which by \eqref{g1 def} is $\mathcal E(z)= \frac{c}{a_1\log z+b_1}$ and satisfies the size estimates $|\mathcal E(z)|\les (\log z)^{-1}$ and $|\mathcal E'(z)|\les z^{-1}(\log z)^{-2}$. Hence the integral is bounded and we may integrate by parts once to get a bound of size $\la t\ra^{-1}$. For the second term, we may integrate by parts once then use the support of the cut-off to see that
		$$
		\int_0^\infty z|x-x_1|^2\chi(z|x-x_1)|\, dz \les 1.
		$$
		The resulting spatial integrals are bounded as in the analysis in the regular case.
		The remaining terms where the +/- difference acts on a resolvent is handled similarly.
		We now consider the contribution of the terms supported on $z|x-x_1|\gtrsim 1$.  In this regime, we use that $|g_1(z)|\gtrsim z^2$ and effectively ignore the logarithm(s) in the denominator.  By the support of the cut-offs, we may integrate by parts once without boundary terms to see
		\begin{multline}\label{eqn:low diff ugly res}
			\bigg|\int_0^\infty e^{-it\sqrt{z^2+m^2}} \frac{z\chi(z)}{\sqrt{z^2+m^2}} \frac{1}{{g_1^-(z)}}\widetilde \chi(z|x-x_1|)e^{\pm i z|x-x_1|}  O_1(z^{\f12}|x-x_1|^{-\f32})\, dz \bigg|\\
			\les \frac{1}{t}\int_0^\infty \bigg|\partial_z \bigg[ \frac{1}{{g_1^-(z)}} \widetilde \chi(z|x-x_1|)e^{\pm i z|x-x_1|}  O_1(z^{\f12}|x-x_1|^{-\f32}) \bigg]\bigg| \, dz\\
			\les \frac{1}{t} \int_0^\infty \chi(z)\widetilde \chi(z|x-x_1|)\big(z^{-\f52}|x-x_1|^{-\f32}+z^{-\f32}|x-x_1|^{-\f12}
			\big)\, dz \les \frac1t.
		\end{multline}
		
		The bottleneck for the time decay is when the +/- difference acts on $M^{-1}(z)$.  
		By Corollary~\ref{cor:M diff}, we have
		$$
		(M^+(z))^{-1}-(M^+(z))^{-1}
		=\bigg(\frac{1}{g^+(z)}-\frac{1}{g^-(z)}\bigg)S_1\Gamma_0S_1+
		O_1\bigg(\frac{1}{z^2(\log z)^3}\bigg)S_1+O_1\bigg(\frac{1}{(\log z)^3}\bigg).
		$$
		In this contribution to \eqref{eqn:J goal}, we further decompose $\mR_0^\pm(z)=\mG_0+(\mR_0^\pm(z)-\mG_0)$ to obtain a sum of terms
		\begin{multline*}
			(\mR_0^- V)^3\mR_0^- v^*[M_{+}^{-1}-M_-^{-1}]v(\mR_0^+ V)^3 \mR_0^{+}
			=(\mG_0 V)^3\mG_0 v^*[M_{+}^{-1}-M_-^{-1}](z)v(\mG_0 V)^3 \mG_0]]\\
			+(\mR_0^--\mG_0) (V\mG_0)^3 v^*[M_{+}^{-1}-M_-^{-1}](z)v(\mG_0 V)^3 \mG_0\\
			+\mG_0V(\mR_0^--\mG_0) ( V\mG_0)^2 v^*[M_{+}^{-1}-M_-^{-1}](z)v(\mG_0 V)^3 \mG_0\\
			+\mG_0V\mG_0V(\mR_0^--\mG_0)V\mG_0  v^*[M_{+}^{-1}-M_-^{-1}](z)v(\mG_0 V)^3 \mG_0+\dots
		\end{multline*}
		Here the remaining terms are those with all $\mG_0$ on the left and one $\mR_0^+-\mG_0$ on the right, or ones with a contribution of multiple $\mR_0^\pm(z)-\mG_0$.
		We note that any summand with more than one $\mR_0^\pm(z)-\mG_0$ gains nearly four powers of $z$ by Lemma~\ref{lem:R0exp}, and can be bounded by $ O_1(z^{2-})$ with respect to the spectral parameter, for which we can comfortably integrate by part for the $\la t\ra^{-1}$ bound.  The spatial variable dependence is bounded uniformly in $x,y$ by an argument similar to the regular case.  We focus on the first term above, and note that by Corollary~\ref{cor:M diff}, the first two terms in the expansion of $[M_{+}^{-1}-M_-^{-1}](z)$ contribute
		\begin{align}\label{eqn:Ft def}
			F_t=\int_0^\infty e^{-it\sqrt{z^2+m^2}} \frac{z\chi(z)}{\sqrt{z^2+m^2}} (\mG_0V)^3\mG_0v^*O_1\bigg(\frac{1}{z^2(\log z)^2}\bigg)S_1v(\mG_0V)^3\mG_0 \, dz.
		\end{align}
		Lemma~\ref{lem:osc log} implies the claim about the time decay.  The claim about rank follows since $S_1$ is of rank at most two.  The contribution of the error term in Corollary~\ref{cor:M diff} allows for one integration by parts as in the discussion following Lemma~\ref{lem:osc log}, since it is of the form
		$$
			\int_0^\infty e^{-it\sqrt{z^2+m^2}} \frac{z\chi(z)}{\sqrt{z^2+m^2}} (\mG_0V)^3\mG_0v^*O_1\bigg(\frac{1}{(\log z)^2}\bigg)S_1v(\mG_0V)^3\mG_0 \, dz.
		$$
		
		Next, we must consider the case when all but one resolvent are replaced by $\mG_0$.  We consider the case of
		$$
			\int_0^\infty e^{-it\sqrt{z^2+m^2}} \frac{z\chi(z)}{\sqrt{z^2+m^2}} (\mR_0^--\mG_0)(z) V\mG_0V\mG_0 v^*[M_{+}^{-1}-M_-^{-1}](z)v(\mG_0 V)^2 \mG_0\, dz.
		$$
		We decompose $(\mR_0^--\mG_0)(z)(x,x_1)$ into two pieces; one when $z|x-x_1|\gtrsim 1$ and the second when $z|x-x_1|\les 1$.  When $z|x-x_1|\gtrsim1$, we note that for $k=0,1$ that
		$|\partial_z^k(\mR_0^--\mG_0)(z)(x,x_1)\widetilde\chi(z|x-x_1|)|\les z^{2-k}$.  This provides the required smallness in $z$ needed to integrate by parts once and bound its contribution by $\la t\ra^{-1}$.  
		
		We now consider $(\mR_0^--\mG_0)(z)(x,x_1)\chi(z|x-x_1|)$.
		Since the next terms in the expansion are of order $z^2\log(z|x-x_1|)$, the argument here is quite delicate.  We mimic the approach in \cite{EGG4} for the Schr\"odinger equation and consider the integral in the spectral parameter.   Define $\log^-r=-\log(r)\chi_{0<r<1}$ and let  $r=|x-x_1|$.  
		We need to bound an integral of the form
		$$
		\int_0^\infty e^{-it\sqrt{z^2+m^2}} \frac{z\chi(z)}{\sqrt{z^2+m^2}}\bigg(\frac{1}{g_1^+(z)}-\frac{1}{g_1^-(z)}\bigg) (z^2\log(zr)+\widetilde E(z,r))\chi(zr)\, dz 
		$$
		where $\widetilde E(z,r)=cz^2(1+r^{-2})+O_1(z^{4-}(r^{-2}+r^{2}))$, for which we may integrate by parts in $z$ to  bound by $\la t\ra^{-1}$ after considering the spatial integrals.  The dependence on $r$ is then controlled in the iterated spatial integrals, as negative powers of $r$ may be integrated while growth if $r$ is large is avoided by the support of $\chi(zr)$.
		Now we consider the leading term, we integrate by parts and  to bound the $z$ integral by (ignoring the terms when the derivative hits the cutoff functions)  
		\begin{multline*}
			\frac{1}{t} \int_0^\infty \chi(z) \Big[\frac{\chi(z r)|\log(z r)|}{z |a\log(z)+b|^3 } + \frac{ \chi(z r)z r^2|\log(z r)|}{|a\log(z)+b|^2} + \frac{1}{z |a\log(z)+b|^2}+z^{0-}\Big]    dz
			\\ \les \frac{1}{t} [1+\log^-(r)+r^{-2}].
		\end{multline*}
		To obtain the last inequality note that the last two summands are clearly integrable. The second summand can be estimated by noting that the denominator is bounded away from zero and using the change of variable $z r \to z$. Finally, the first summand can be estimated by using the inequality 
		$$\chi(z) \chi(z r)|\log(z r)|\les 1+|\log(z)|+\log^-(r).$$
		This yields the required inequality asserted in Theorem~\ref{thm:main} by noting that
		$$ 
		\sup_{x\in \R^4} v(y)(\mG_0V)^3(1+\log^-(|\cdot -x|)+|\cdot-x|^{-2})\in L^2_y(\R^4).
		$$
		This is easily verified by applying Lemma~\ref{lem:abstractSpatialIntegrals} along with the bound $\log^{-}|x_1-x|\les |x_1-x|^{-\epsilon}$.  The contribution of the error term of size $O_1(z^{-2}(\log z)^{-3})$ in $M_+^{-1}-M_{-}^{-1}$ is slightly less delicate, but is handled similarly.

		Now, if $S_1=S_2\neq 0$, by Proposition~\ref{prop:M inv res} we have
		$$
		M_\pm^{-1}(z)=-\frac{D_2}{z^2}+\Gamma_0+\frac{g_2^\pm(z)}{z^4}\Gamma_1 +O_1(z^{0+}).
		$$
		By Corollary~\ref{cor:M diff}, and the argument above, it is clear that when the +/- difference acts on the $M_\pm^{-1}(z)$ that the $\la t\ra^{-1}$ bound holds except for the contribution of $\frac{g_2^\pm(z)}{z^4}\Gamma_1$, here we need only recall \eqref{g2 def} to see that that $\frac{g_2^+(z)-g_2^-(z)}{z^4}=2\Im(b_2)\neq 0$ is a constant.

		To control the contribution when the +/- difference acts on a free resolvent.  We need only consider the contribution of most singular term, as the remaining terms clearly allow for an integration by parts.  Hence, we need to bound
		$$
		\int_0^{\infty}e^{-it\sqrt{z^2+m^2}} \frac{z\chi(z)}{\sqrt{z^2+m^2}}(\mR_0^+-\mR_0^-)(z) (V\mG_0)^3 v^*\bigg[-\frac{D_2}{z^2}\bigg]v(\mG_0 V)^3 \mG_0\, dz
		$$
		Now, the bound in \eqref{eqn:Ro dif res} allows us to rewrite this when $z|x-x_1|\les 1$, in terms of $z$, as an integral of the form
		$$
		\int_0^{\infty}e^{-it\sqrt{z^2+m^2}} \frac{z\chi(z)}{\sqrt{z^2+m^2}} \bigg[ c + O_1(z^2 |x-x_1|^2) \bigg]\chi(z|x-x_1|)\, dz.
		$$
		This suffices to allow us to integrate by parts once in $z$ and bound with $\la t\ra^{-1}$.  The contribution of the resolvents when $z|x-x_1|\gtrsim 1$ is controlled by the same argument in \eqref{eqn:low diff ugly res}. The remaining spatial variable dependence is easily seen to be bounded.
		So all pieces can be integrated by parts once in $z$ to gain the $\la t\ra^{-1}$ bound.  
		
		Finally, the case when $S_1\neq S_2\neq 0$ is controlled by the arguments above noting that the finitely many operators $\Gamma_{Q,j}$ in the expansion for $M_{\pm}^{-1}(z)$ in Proposition~\ref{prop:M inv res} have finite rank, hence their contribution is finite rank.  The remaining leading terms are controlled by the arguments above noting that
		$$
			\frac{z^2}{(g_1^\pm(z))^2}=O_1\bigg(\frac{1}{z^2(\log z)^2} \bigg), \quad 
			\frac{z^2}{(g_1^+(z))^2}-\frac{z^2}{(g_1^-(z))^2}=O_1\bigg(\frac{1}{z^2(\log z)^3} \bigg)
		$$
		and $\frac{z^2}{g_1^\pm(z)}=(a_1\log z+b_1)^{-1}$.
	\end{proof}

	\section{Classification of threshold spectral subspaces} \label{sec:esa}
	
	We now relate the spectral subspaces $S_jL^2(\R^4)$ and the invertibility of operators arising in $M^\pm(z)$ in Definition~\ref{resondef} to the existence of threshold resonances; distributional solutions to $\mathcal H\psi=\pm m\psi$ with $\mathcal H=D_m+V$.
	We note that the proofs mirror those in the three dimensional massive case, \cite{EGT} with modifications to the different structure in four dimensions.  The resonance structure mimics that of the Schr\"odinger operator, see \cite{EGG4}, having a space of `p-wave' resonances in addition to threshold eigenvalues.  One important difference from the three dimensional massive Dirac analysis is the need for a fourth $\alpha_j$ matrix in the definition of the Dirac operator, as well as the logarithmic behavior of the resolvent as the spectral parameter approaches a threshold.  We consider the characterization at the positive threshold energy $\lambda=m$.  The analysis may be adapted to the negative threshold, $-m$, with straight-forward adaptations of the argument presented here as discussed in Remark~\ref{rmk:neg branch}.  We leave the details to the interested reader.
	
	We begin by relating the subspace $S_1L^2(\R^4)$ to the existence of threshold obstructions.  
	\begin{lemma}\label{lem:esa1}
		Assume $|v(x)| \les \la x \ra ^{-{2}-} $.  $\phi \in S_1 L^2(\R^4) \setminus \{0\} $ if and only if $\phi= Uv \psi $ for some  $\psi \in L^{2,-{0}-}(\R^4)\setminus \{0\} $ which  is a distributional solution of  $ ( D_m+V-mI)\psi=0 $. Furthermore, $\psi=- \mathcal{G}_0v^{*}\phi$ and $\psi$ is a bounded function.
	\end{lemma} 
	\begin{proof} If  $\phi \in S_1 L^2(\R^4)  \setminus \{0\}, $ then by Definition~\ref{resondef}, $(U+ v \mathcal{G}_0 v^{*}) \phi =0$.   Since $U^2=I$,  
		\begin{align}\label{S1G0}
			\phi= -U v \mathcal{G}_0 v^{*} \phi = U v \psi,\,\, \text{ where }\,\,  \psi:=-  \mathcal{G}_0 v^{*} \phi.
		\end{align}
		Using \eqref{eq:G0 def}  and \eqref{eq:Dmpm} with $\lambda=m$, we obtain
		\begin{multline}\label{eq:eigeniden}
			(D_m - mI) \psi   = -(D_m - mI)\mathcal{G}_0 v^{*}\phi = -(D_m - mI)(D_m +mI)G_0v^{*}\phi \\
			= \Delta G_0 v^{*}\phi = - v^{*} \phi =-v^*Uv\psi =-V\psi.
		\end{multline}  
		Therefore,    $ ( D_m +V -mI)\psi=0 $ in the sense of distributions. In the fourth equality above, 
		we used  the fact that $\Delta G_0 v^{*}\phi=-v^{*}\phi$ holds since  $v^{*}\phi \in L^{2,2+}$, see Lemma~2.4 in \cite{Jen2} for example.  
		
		Now we prove that   $\psi \in L^{2,-{0}-}(\R^4) $. Note that 
		\begin{multline} 
			\psi =- [ - i\alpha \cdot \nabla + 2m I_{uc} ] G_0 v^{*} \phi 
			=- [ - i\alpha \cdot \nabla + 2m I_{uc} ]  \frac{1}{2 \pi^2} \int_{\R^4} \frac{v^{*}(y) \phi (y)}{|x-y|^2} \\
			=   \frac{1}{2 \pi^2} \int_{\R^4}  i \alpha \cdot(x-y) \frac{v^{*}(y) \phi (y)}{|x-y|^4} - m I_{uc}  \frac{1}{2 \pi^2} \int_{\R^4} \frac{v^{*}(y) \phi (y)}{|x-y|^2} 
			=:  \psi_1+ \psi_2. \label{eq:g1g2}
		\end{multline}
		Since the integrals in equation can be bounded by fractional integral operators, by Lemma 2.3 in \cite{Jen}, we have $\psi_1\in L^2(\R^4) \subseteq L^{2,-{0}-}(\R^4)  $ provided $|v(x)| \les \la x \ra ^{-1} $; and  $\psi_2\in L^{2,-0-}(\R^4)$ provided $|v(x)| \les \la x \ra ^{-2-} $. 
		
		Conversely, assume that $\phi= Uv \psi $ for some $\psi \in L^{2,-0-}(\R^4) $ satisfying $ (D_m+V-mI)\psi=0 $. We need only show that $T_0\phi=0$.
		By a calculation similar to \eqref{eq:eigeniden}, we have 
		$$
		(D_m - mI) \psi   = -V\psi= - v^{*} \phi= \Delta G_0 v^{*}\phi = -(D_m - mI)(D_m +mI)G_0v^{*}\phi =  -(D_m - mI)\mathcal{G}_0 v^{*}\phi.
		$$ 
		Thus, also using \eqref{eq:Dmpm} with $\lambda=-m$, we have 
		$$ \Delta (\psi +  \mathcal{G}_0 v^{*}\phi)=(D_m+mI)(D_m-mI)(\psi +  \mathcal{G}_0 v^{*}\phi)  = 0. 
		$$
		Noting that $\psi +  \mathcal{G}_0 v^{*}\phi\in L^{2,-0-}(\R^4)$ is harmonic, we conclude that (see \cite{JenKat})
		$  \psi +   \mathcal{G}_0 v^{*}\phi=0.$
		Therefore,
		$$
		(U+ v \mathcal{G}_0 v^{*}) \phi = v\psi  +v\mathcal{G}_0 v^* \phi=0,
		$$
		and hence $\phi\in S_1L^2$.
		
		Since $\phi= Uv \psi $, if $\phi\neq 0 $, then $ \psi\neq 0$. The reverse implication  follows from $\psi=- \mathcal{G}_0 v^{*} \phi$.
		Finally, using \eqref{S1G0} we have $\psi =-\mG_0V\psi $. Iterating this identity we obtain $\psi =(-\mG_0V)^k\psi $ for sufficiently large $k$. Therefore, applying integral estimates in Lemma~\ref{lem:spatial estimates}, we see that $\psi $ is bounded.
	\end{proof}

	\begin{lemma}\label{lem:esa2} Suppose $|v(x)| \les \la x \ra ^{-2-} $. Fix $\phi=Uv\psi\in S_1L^2$, where $\psi \in L^{2,-0-}(\R^4)\setminus \{0\} $  is a distributional solution of  $ ( D_m+V-mI)\psi=0 $. Then $\phi \in S_2  L^2(\R^4) $ if and only if   $\psi \in  L^2(\R^4)$.  Moreover, any threshold eigenfunction, $\psi$, is a bounded function.
	\end{lemma}
	
	\begin{proof} The boundedness of $\psi$ and the equality $ ( D_m+V-mI)\psi=0 $ were obtained in the   previous lemma.  
		First note that if $\phi  \in S_2 L^2(\R^4) $, namely $S_1T_1S_1\phi= S_1v\mathcal{G}_1v^*\phi  =0$, recalling \eqref{eq:G1 def} we have 
		\begin{align*}
			0= \la vM_{uc} v^*\phi ,\phi  \ra =  \la M_{uc} v^*\phi , M_{uc}v^*\phi  \ra_{\C^4} = \|M_{uc} v^* \phi \|^2_{\C^4}, 
		\end{align*} 
		Hence, $M_{uc} v^*\phi =0$. It is also clear that  if $M_{uc} v^*\phi =0$, then $\phi\in S_2L^2$.
		
		Also note that in the proof of Lemma~\ref{lem:esa1}, we showed that $\psi=\psi_1+\psi_2$ and  $\psi_1 \in L^2(\R^4) $. 
		Therefore it suffices to prove that $M_{uc}v^*\phi=0$ if and only if $\psi_2\in L^2$. 
		Recalling \eqref{eq:g1g2} we can write  $\psi_2$  as
		\be\label{eq:psi2}
		\psi_2(x) =    \frac{-m}{2 \pi^2} I_{uc}  \int_{\R^4} v^{*}(y) \phi (y) \bigg[ \frac{1}{|x-y|^2} - \frac{1}{\la x\ra^2} \bigg] dy
		- \frac{m}{2 \pi^2 \la x \ra^2}  [M_{uc}v^* \phi]. 
		\ee
		By the Mean Value Theorem, we have
		$$
		\bigg| \frac{1}{|x-y|^2}-\frac1{\la x\ra^2}\bigg| \les \frac{\la y\ra}{|x-y|^2\la x\ra}+\frac{\la y\ra}{|x-y|\la x\ra^2}\les \la y\ra \mathcal I_2(x,y)\la x\ra^{-1}+\la y\ra \mathcal I_1(x,y)\la x\ra^{-2},
		$$
		where $\mathcal I_j(x,y)$ is the integral kernel of the fractional integral operator $\mathcal I_j$.  Since $\mathcal I_2:L^{2,1+}\to L^{2,-1}$ and $\mathcal I_1:L^{2,1+}\to L^{2,-2}$.  The decay of $v^*$ along with the powers of $x$ and $y$  allow us to deduce that the first term in \eqref{eq:g1g2} is in $L^2(\mathbb R^4)$.  Therefore, if $M_{uc}v^*\phi=0$, then $\psi_2\in L^2$ and hence $\psi \in L^2$. 
		
		On the other hand, if $\psi\in L^2$, then we must have that $\psi_2\in L^2$.  Since $\la x\ra^{-2}\notin L^2(\mathbb R^4)$, we must have that $M_{uc}v^*\phi=0$.  As shown above, this requires that $T_1\phi=0$ i.e. $\phi\in S_2L^2(\R^4)$.
		
	\end{proof}

	One immediate consequence of the proof, which we use in the argument for the dispersive bounds and inverse expansions, is the following orthogonality condition. 
	\begin{corollary}\label{cor:S2 orth}
		We have the identities
		\begin{align*}
			S_2v\mathcal G_1=\mathcal G_1v^*S_2=\frac{m}{2\pi^2}M_{uc}v^* S_2=\frac{m}{2\pi^2}S_2vM_{uc}=0.
		\end{align*}
	\end{corollary}  
	Further, the rank of $S_1-S_2$ is at most two at the threshold $\lambda=m$.
	\begin{corollary}\label{cor:2d}
		Assume $|v(x)| \les \la x \ra ^{-2-}$.  Then the rank of $S_1-S_2$ is at most two.
	\end{corollary}  
	\begin{proof}
		We consider the representation in \eqref{eq:g1g2}.  We have already shown that $\psi_1\in L^2$. By \eqref{eq:psi2} and the discussion following it, we can write  $\psi_2$ as 
		$$
		\frac{-m}{2 \pi^2 \la x \ra^2}  [M_{uc}v^* \phi]+O_{L^2}(1)=\frac{1}{ \la x\ra^2} (a_1, a_2, 0, 0)^T+O_{L^2}(1).
		$$	
		The constants $a_j= -\frac{m}{2\pi^2} \int_{\R^4} [v^*(y)\phi(y)]_j \, dy$ are finite by the assumed decay of $v^*$.
	\end{proof}

	\begin{lemma}\label{lem:esa3}
		Assume $|v(x)| \les \la x \ra ^{-2-}$. Then $S_2v\mathcal{G}_2v^*S_2$ is invertible  as an operator in $S_2 L^2(\R^4) $.  
	\end{lemma}
	\begin{proof}
		First, we note that $S_2v\mathcal{G}_2v^*S_2$ is a compact operator.  Hence, we need only show that its kernel is trivial. 
		Assume that for some $\phi \in S_2 L^2(\R^4)$, $S_2v\mathcal{G}_2v^*S_2\phi=0$, then $ \la \mathcal{G}_2 v^*\phi , v^*\phi  \ra =0$. By Corollary~\ref{cor:S2 orth}, $\mathcal{G}_1 v^*\phi =0$, and we have
		\begin{align} \label{eq:verb}
			0= \la \mathcal{G}_2 v^*\phi , v^*\phi  \ra = - \lim_{z \rightarrow 0} { \f 1 {z^2}}\la [ \mathcal{R}_0^+(\lambda) - \mathcal{G}_0-g_1^+(z)\mathcal G_1] v^* \phi , v^*\phi  \ra 
		\end{align} 
		where $\lambda= \sqrt{z^2+m^2}$.  For $\phi\in S_2L^2$, the orthogonality condition ensures that the contribution of the $\mathcal G_1$ term is zero.  So, for $z= i\omega$ and $ 0 < \omega \ll m$, we have
		\begin{multline*}
			\frac 1 {z^2} \la [ \mathcal{R}_0(\lambda) - \mathcal{G}_0-g_1^+(z)\mathcal G_1] v^* \phi , v^*\phi  \ra =
			\f 1 {z^2} \la [ \mathcal{R}_0(\lambda) - \mathcal{G}_0] v^* \phi , v^*\phi  \ra \\
			= \int_{\R^4} \big\la K(\omega,\xi) \widehat{v^*\phi (\xi)}, \widehat{v^*\phi (\xi)} \big\ra _{\C^4} d\xi 
		\end{multline*}
		on the Fourier side.  Now, recalling \eqref{eq:G0 def} and \eqref{G0 def}, one has
		\begin{multline*}
			K(\omega,\xi) =
			\frac{1}{\omega^2 |\xi|^2} \left(\begin{array}{cccc}
				2m & 0 & \kappa & \bar{\eta}       \\
				0 & 2m & \eta &  -\bar{\kappa}      \\
				\bar{\kappa} & \bar{\eta}  & 0 & 0      \\
				\eta &  -\kappa & 0 & 0
			\end{array}\right)\\
			-\frac{1}{\omega^2(\omega^2+|\xi|^2)}
			\left(\begin{array}{cccc}
				m+\sqrt{m^2-\omega^2} & 0 & \kappa & \bar{\eta}      \\
				0 & m+\sqrt{m^2-\omega^2}  & \eta & -\bar{\kappa}       \\
				\bar{\kappa} &  \bar{\eta} & \sqrt{m^2-\omega^2}-m &0       \\
				\eta & -\kappa & 0 & \sqrt{m^2-\omega^2}-m
			\end{array}\right).
		\end{multline*}
		Here $\xi=( \xi_1, \xi_2, \xi_3, \xi_4)$, $\eta = \xi_2 + i\xi_1$ and $\kappa=\xi_3+i\xi_4$.  Notice that $|\eta|^2+|\kappa|^2=|\xi|^2$.   The first matrix corresponds to the contribution of the transform of $\mG_0$. Now, write $\tau:=\frac{|\xi|^2}{\omega^2}(m-\sqrt{m^2-\omega^2})$, then 
		\begin{align}
			K(\omega,\xi)=
			\frac{1}{|\xi|^2(\omega^2+|\xi|^2)}
			\left(\begin{array}{cccc}
				2m+\tau &0&\kappa &\eta   \\
				0 & 2m+\tau & \eta &  -\bar{\kappa}       \\
				\bar{\kappa} & \bar{\eta} & \tau & 0        \\
				\eta &  -\kappa  & 0 &\tau
			\end{array}\right).
		\end{align} 
		The eigenvalues of $K(\omega,\xi)$ are
		\begin{align*}
			&\lambda_{1,2}=\frac{m+\tau+\sqrt{m^2+|\xi|^2}}{|\xi|^2(\omega^2+|\xi|^2)},  \qquad \lambda_{3,4}= \frac{m+\tau-\sqrt{m^2+|\xi|^2}}{|\xi|^2(\omega^2+|\xi|^2)}. 
		\end{align*} 
		In particular, the eigenvalues of $K(\omega,\xi)$ are real and for any $\xi\neq0$ they are positive. Thus the matrix $K(\omega,\xi)$ is self-adjoint, and further is positive definite when $\xi\neq0$. The eigenvalues are non-increasing functions of $\omega \in (0,m)$. This allows one to invoke the monotone convergence theorem and take the limit into the integral \eqref{eq:verb} to obtain
		$$ 
		0= \lim_{\omega\rightarrow 0^+}\int_{\R^4} \la K(\omega,\xi)\widehat{v^*\phi (\xi)} , \widehat{v^*\phi (\xi)} \ra_{\C^4} d\xi \ =\int_{\R^4} \la K(0,\xi)\widehat{v^*\phi (\xi)} , \widehat{v^*\phi (\xi)} \ra_{\C^4} d\xi,
		$$
		where
		\begin{align*}
			K(0,\xi)= \frac{1}{|\xi|^4} \left(\begin{array}{cccc}
				2m+\frac{|\xi|^2}{2m} & 0 & \kappa & \bar{\eta}      \\
				0 & 2m+\frac{|\xi|^2}{2m} &  \eta& -\bar{\kappa}   \\
				\bar{\kappa} & \bar{\eta} & \frac{|\xi|^2}{2m} & 0       \\
				\eta &  -\kappa & 0 & \frac{|\xi|^2}{2m}
			\end{array}\right).
		\end{align*}
		Note that this matrix is also self-adjoint and positive definite for any $\xi\neq 0$, from which we conclude that $\widehat{v^*\phi (\xi)}=0$. To conlcude that $v^*\phi=0$, we note that $v^*\phi$ has $L^1$ entries, hence $v^*\phi (x) =0$. Recall that the fact that $\phi  \in S_1 L^2(\R^4) $ implies that $\phi = Uv^*\psi$  for $\psi= - \mathcal{G}_0v^*\phi $. Hence, we conclude that $\phi =0$.
	\end{proof}  
	In addition, we have
	\begin{align*}
		K(0,\xi)= \frac{1}{2m}  \frac{1}{|\xi|^4} \left(\begin{array}{cccc}
			2m & 0 & \kappa & \bar{\eta}      \\
			0 & 2m &  \eta& -\bar{\kappa}   \\
			\bar{\kappa} & \bar{\eta} & 0 & 0       \\
			\eta &  -\kappa & 0 & 0
		\end{array}\right)^2.
	\end{align*}
	Therefore, for any $\phi \in S_2L^2$ we have 
	\begin{align}\label{G0 to G2 ident} 
		\la \mathcal{G}_2 v^*\phi  , v^*\phi  \ra =-\frac{1}{2m} \la  \mathcal{G}_0 v^*\phi , \mathcal{G}_0 v^*\phi  \ra.
	\end{align}
	\begin{lemma}\label{lem:eproj}
		The operator $P_m = -\frac{1}{2m}\mathcal{G}_0V\mathcal{G}_0V\mathcal{G}_0v^*S_2D_2S_2v\mathcal{G}_0V\mathcal{G}_0V\mathcal{G}_0$ is the finite rank, orthogonal projection in $ L^2(\R^4) $ onto the eigenspace of $\mathcal H = D_m + V$ at threshold $m$.  
	\end{lemma}   
	
	The proof is standard, it follows as in the proof in the three dimensional case, see \cite{EGT}.  We omit the proof for the sake of brevity.

	\section{High Energy}\label{sec:hi}
	For completeness we prove frequency-localized dispersive estimates, which reflect the need for smoothness on the initial data and/or potential.  The phase has vanishing curvature as $z\to \infty$, so classic stationary phase methods fail.  However, when restricting to dyadic intervals of the from $z\approx 2^j$, the Hessian of the phase is bounded below allowing for an adapting stationary phase argument at the cost of powers of $2^j$.  We don't expect the bounds for the high energy portion of the perturbed evolution to sharp with respect to differentiability of the initial data.
	
	Let $\chi_j(z)$ be a smooth cut-off to the interval $z\approx 2^j$.
	Our main technical result is
	\begin{prop}\label{prop:high energy}
		Let $V(x)$ have continuous entries satisfying $|V(x)|\les \la x\ra^{-5-}$.  Then
		the following bound holds 
		\begin{align} \label{eqn: 2j hightail}
			\sup_{x,y\in \R^4}
			\bigg|\int_0^\infty e^{-it\sqrt{z^2+m^2}}\frac{z\chi_j(z)}{\sqrt{z^2+m^2}}  [\mR_V^+-\mR_V^-](\lambda)  (x,y)
			\, dz\bigg|    \les  	  \min(2^{4j},   2^{11j/2} |t|^{-2}).
		\end{align}
	\end{prop}
	We note that Theorem~\ref{thm:full results} follows from these bounds by summing in $j$ since the contribution of the powers of $\mathcal H$ to the Stone's formula satisfy $\la \mathcal H\ra^{-1}\approx 2^{-j}$ on the support of $\chi_j$.  We note that if we don't use the oscillation, one needs only require $|V(x)|\les \la x\ra^{-2-}$ for the bound $|\eqref{eqn: 2j hightail}|\les 2^{4j}$, this is clear from the proofs below.
	
	To prove these high energy bounds, we utilize different resolvent bounds for $z\approx 2^j$ than when $z$ approaches zero.  Recall that
	$$
	\mR_0^\pm(z)=(-i\alpha \cdot \nabla +m\beta+\sqrt{z^2+m^2}I)R_0^\pm (z^2),
	$$
	We now decompose the kernel of $\mR_0^\pm(z)(x,y)$ into two pieces based on the size of $z|x-y|$.
	$\mR_0^\pm(z)(x,y)=\chi(z|x-y|)\mR_0^\pm(z)(x,y)+\widetilde \chi(z|x-y|)\mR_0^\pm(z)(x,y):=\mR_L(z)(x,y)+\mR_H^\pm(z)(x,y)$.  Here we note that for $k=0,1,2$ and $0< \epsilon\leq 1$ we have
	\begin{align}
		\mR_H^\pm(z)(x,y)&= e^{\pm i z|x-y|} \omega_{\pm}(z|x-y|), \qquad |\partial_z^k  \omega_{\pm}(z|x-y|)|\les \frac{z^{\frac32-k}}{r^{\frac32}}.\label{eqn:RH bds}\\
		|\partial_z^k \mR_L(z)(x,y)|&=\frac{z^{-\epsilon-k}}{|x-y|^{4-\epsilon}}.
		\label{eqn:RL bds}
	\end{align}
	We note that the larger power of $z$ in the estimate for $\mR_H$ is because $\sqrt{z^2+m^2}\approx z$ when $z\approx 2^j$.
	The high energy argument won't use the expansions for $\mR_V^\pm$ developed for the low energy expansions in Section~\ref{sec:reg}.  Instead, we use the limiting absorption principle, \cite{EGG2}:
	\begin{align}\label{eqn:LAP}
		\sup_{\lambda>0} \|\partial_\lambda^k \mR_V^\pm(\lambda)\|_{L^{2,\sigma+k}\to L^{2,-\sigma-k}} \les 1, \qquad \sigma>\f12, \qquad k=0,1,2,
	\end{align}
	which requires the continuity assumption  on the potential.
	We utilize a selective iteration scheme to avoid accumulating excessive growth in the spectral parameter $z$.  From \eqref{eqn:RH bds} we can see that $\mR_H(z)(x,\cdot)$ is locally an $L^2$ function, while from \eqref{eqn:RL bds}, $\mR_L(z)(x,\cdot)$ is not.
	
	We begin by iterating the resolvent identity (and dropping the $\pm$ superscripts for convenience) to write
	\begin{align}\label{eqn:high BS}
		\mR_V=\mR_0-\mR_0V\mR_0+\mR_0V\mR_V V\mR_0
	\end{align}

	We recall Lemma~3.5 in \cite{eg2}:
	\begin{lemma}\label{lem:high stat phase2} Fix $j\in\mathbb N$, and let $\chi_j(z)$ be  a cut-off to
		$z\approx 2^j$.  If 
		$$
		|a(z)|\les 	\frac{z\chi_j(z) \widetilde \chi(zr)}{(1+zr)^{\f12}} ,
		\qquad |\partial_z a(z)|\les 
		\frac{\chi_j(z) \widetilde \chi(zr)}{(1+zr)^{\f12}},
		$$
		then we have the bound
		\begin{align*}
			\bigg|\int_0^\infty e^{-it \phi_\pm(z) } a(z)\, dz\bigg| \les  \min(2^{2j}, 2^{\frac{3j}2}|t|^{-1/2}, 2^{2j} |t|^{-1}),
		\end{align*}
		where $\phi_\pm(z)=\sqrt{z^2+m^2}\mp \frac{z r}{t}$.
	\end{lemma}
	
	We utilize this bound to show that the free Dirac operator satisfies the following dispersive bound.
	
	\begin{lemma}\label{lem:free hi}
		
		We have the bound
		
		$$
		\sup_{x,y\in \R^4} \bigg| \int_0^\infty e^{-it\sqrt{z^2+m^2}} \frac{z\chi_j(z)}{\sqrt{z^2+m^2}} [\mR_0^+-\mR_0^-](z)(x,y)\, dz \bigg|\les \min(2^{4j},2^{4j}|t|^{-2},2^{\frac72 j}|t|^{-\frac32})
		$$
		Consequently, the free Dirac operator satisfies the bounds
		$$
			\|e^{-it D_m} \la D_m\ra^{-4-} \|_{L^1\to L^\infty} \les \la t\ra^{-2},
		\quad 
			\|e^{-it D_m} \la D_m\ra^{-\frac72-} \|_{L^1\to L^\infty} \les |t|^{-\f32}.
		$$
		
	\end{lemma}
	
	\begin{proof}
		First, for $z|x-y|\les 1$, recalling \eqref{eqn:RL low bds} and \eqref{eqn:RL diff bds}, we have $\chi_j(z)[\mR_L^+-\mR_L^-](z)(x,y)=\mathcal E_j(z)$ a function that satisfies $|\partial_z^k \mathcal E_j(z)|\les z^{2-k}$ uniformly in $x,y$, and is supported on $z\approx 2^j$.  Then, we seek to bound
		$$
			\bigg|\int_0^\infty e^{-it\sqrt{z^2+m^2}} \frac{z }{\sqrt{z^2+m^2}} \mathcal E_j(z)\,dz\bigg|.
		$$
		We may bound the modulus of the integrand by $z^3$ and hence, by the triangle inequality, the integral is bounded by $2^{4j}$.  If we integrate by parts, there are no boundary terms when integrating by parts due to the support of $\mathcal E_j(z)$.
		\begin{multline*}
			\int_0^\infty e^{-it\sqrt{z^2+m^2}} \frac{z }{\sqrt{z^2+m^2}} \mathcal E_j(z)\,dz=\frac{1}{-it}\int_0^\infty e^{-it\sqrt{z^2+m^2}} \partial_z \mathcal E_j(z)\, dz\\
			=\frac{1}{t^2} \int_0^\infty e^{-it\sqrt{z^2+m^2}}\partial_z \bigg( \frac{\sqrt{z^2+m^2}}{z} \partial_z \mathcal E_j(z) \bigg)\, dz.
		\end{multline*}
		Since differentiation is comparable to division by $z$ we may bound the second integral by $|t|^{-1}2^{2j}$ and the third integral by $|t|^{-2}2^{j}$.
		From this, we see that 
		$$
		\sup_{x,y} \bigg|\int_0^\infty e^{-it\sqrt{z^2+m^2}} \frac{z \chi_j(z)}{\sqrt{z^2+m^2}} [\mR_L^+-\mR_L^-](z)(x,y)\,dz\bigg| \les 2^{j}\min(2^{3j},2^{j}|t|^{-1} ,|t|^{-2}).
		$$
		For the high energy portion the bounds in \eqref{eqn:RH bds} are valid, note that we do not utilize any cancellation between the + and - resolvents but instead bound them individually.  We need to bound:
		$$
		\int_0^\infty e^{-it\sqrt{z^2+m^2}} \frac{z \chi_j(z)}{\sqrt{z^2+m^2}} \mR_H^\pm(z)(x,y)\, dz=
		\int_0^\infty e^{-it\sqrt{z^2+m^2}} \frac{z \chi_j(z)}{\sqrt{z^2+m^2}} e^{\pm i z|x-y|}\omega_{\pm}(z|x-y|)\, dz.
		$$
		By the triangle inequality, we obtain a bound of size $2^{4j}$ using that $z|x-y|\gtrsim 1$ so that $|\omega_\pm(z|x-y|)|\les z^3$.
		We integrate by parts once to seee
		$$
		\frac{1}{-it} \int_0^\infty e^{-it\sqrt{z^2+m^2}}  \partial_z [\chi_j(z) e^{\pm i z|x-y|}\omega_{\pm}(z|x-y|) ]\, dz.
		$$
		Here
		$$
		[\chi_j(z) e^{\pm i z|x-y|}\omega_{\pm}(z|x-y|) ]=e^{\pm i z|x-y|} za(z)
		$$
		where $a(z)$ satisfies the bounds of of Lemma~\ref{lem:high stat phase2}.  The extra power of $z$ contributes an additional power of $2^j$ to the upper bound since $z\approx 2^j$.  Hence, we have 
		\begin{multline*}
			\sup_{x,y} \bigg|\int_0^\infty e^{-it\sqrt{z^2+m^2}} \frac{z \chi_j(z)}{\sqrt{z^2+m^2}} \widetilde\chi(z|x-y|)(\mR_0^+-\mR_0^-)(z)(x,y)\,dz\bigg|\\
			\les \min(2^{4j},2^{7j/2}|t|^{-3/2}, 2^{4j} |t|^{-2}).
		\end{multline*}
		Combining the bounds on $z|x-y|\les 1 $ and $z|x-y|\gtrsim 1$ proves the claim.
	\end{proof}
	
	We modify this proof to control the second term in \eqref{eqn:high BS}, namely,
	\begin{lemma}\label{lem:hi bornterm}
		
		Under the assumptions of Proposition~\ref{prop:high energy}, we have the bound
		$$
		\sup_{x,y\in \R^4} \bigg| \int_0^\infty e^{-it\sqrt{z^2+m^2}} \frac{z\chi_j(z)}{\sqrt{z^2+m^2}} [\mR_0^+V\mR_0^+-\mR_0^-V\mR_0^-](z)(x,y)\, dz \bigg|\les \min(2^{4j},2^{3j}|t|^{-\f32},2^{\frac72 j}|t|^{-2}).
		$$	
		
	\end{lemma}
	
	\begin{proof}
		
		We note that
		$$
		\mR_0V\mR_0=\mR_LV\mR_L+\mR_LV\mR_H+\mR_HV\mR_L+\mR_HV\mR_H.
		$$	
		Due to the singularities of $\mR_L$ in the spatial variables, \eqref{eqn:RL bds}, we must utilize the +/- difference on the first summand $\mR_LV\mR_L$.
		Recalling \eqref{eqn:RL diff bds} and utilizing the support condition that $z|x-y|\les 1$ we have
		$$
		|\partial_z^k[\mR_L^+-\mR_L^-](z)(x,y)|\les z^{2-k} (z|x-y|)^{-\ell}
		$$
		for our choice of $\ell\geq 0$. 
		In which case, by \eqref{eqn:alg identity}, it suffices to bound
		$$
		\int_0^\infty e^{-it\sqrt{z^2+m^2}} \frac{z\chi_j(z)}{\sqrt{z^2+m^2}} [(\mR_L^+-\mR_L^-)V\mR_L^+](z)(x,y)\, dz
		$$
		By the triangle inequality and the bounds above with $\ell=0$, and \eqref{eqn:RL bds} with $\epsilon=0+$, we may bound this  by
		$2^{3j}$ uniformly in $x,y$.
		By the support of $\chi_j$, we may integrate by parts without boundary terms and we seek to bound
		$$
		\frac{1}{(it)^2}\int_0^\infty e^{-it\sqrt{z^2+m^2}} \partial_z \bigg(
		\frac{\sqrt{z^2+m^2}}{z}\partial_z \big( \chi_j(z)[(\mR_L^+-\mR_L^-)V\mR_L^+](z)(x,y \big)
		\bigg) \,dz
		$$	
		We note that here  differentiation is comparable to division by $z$.  Choosing $\epsilon=1$ and $\ell=1-$ and using Lemma~\ref{lem:spatial estimates}, we may bound by
		\begin{align*}
			\frac{1}{|t|^2}\int_{\R^4}\int_0^\infty z^{-2+}\chi_j(z)\frac{\la x_1\ra^{-2-}}{|x-x_1|^{1-}|x_1-y|^3}\, dz\, dx_1  \les \frac{1}{|t|^2},
		\end{align*}
		uniformly in $x,y$.  Taking advantage of the $x_1$ integral allows us to negate the growth in $z$ at the cost of integral singularities in $|x-x_1|$.
		We now turn our attention to the summand containing only $\mR_H$.  By the triangle inequality, \eqref{eqn:RH bds}, one can dominate the integrand by $z^3|x-x_1|^{-\f32}|V(x_1)||x_1-y|^{-\f32}$, and apply Lemma~\ref{lem:spatial estimates} to obtain a bound of size $2^{4j}$.  Otherwise, we integrate by parts once.  By symmetry, it suffices to consider 
		\begin{multline*}
			\int_{\R^4}\frac{1}{-it}\int_0^\infty e^{-it\sqrt{z^2+m^2}}  \chi_j(z)[\partial_z\mR_H(z)]V\mR_H(z)\, dz\, dx_1\\
			=\frac{1}{-it}\int_{\R^4}\int_0^\infty e^{-it\sqrt{z^2+m^2}\mp z(|x-x_1|+|x_1-y|)/t}  a(z)\, dz\, dx_1,
		\end{multline*}
		where
		$$
		a(z)=\chi_j(z)|x-x_1| \big[ \omega(z|x-x_1|)+\omega'(z|x-x_1|)\big]V(x_1) \omega(z|x_2-y|).
		$$
		From the bounds on $\omega$ in \eqref{eqn:RH bds}, 
		and writing $r:=\max(|x-x_1|,|x_1-y|)$ and $s:=\min(|x-x_1|,|x_1-y|)$,
		we see that
		$$
		|a(z)|\les \frac{\chi_j(z)\widetilde \chi(zr) z^{\frac52}|V(x_1)|}{(1+zr)^{\frac{1}{2}} s^{\f32}}\les \frac{z\chi_j(z)\widetilde \chi(zr)}{(1+zr)^{\f12}} \frac{2^{\f32 j}|V(x_1)|}{s^\frac32} .
		$$
		Where we used that $r+s\approx r$ and $z\approx 2^j$.
		Similarly,
		$$
		|\partial_z a(z)|\les \frac{\chi_j(z)\widetilde \chi(zr)z^{\frac32}|V(x_1)|}{(1+zr)^{\frac{1}{2}} s^{\f32}} \les 
		\frac{\chi_j(z)\widetilde \chi(zr)}{(1+zr)^{\f12}} \frac{2^{\f32j}|V(x_1)| }{s^\frac32}.
		$$
		Hence, we may apply Lemma~\ref{lem:high stat phase2} to bound the contribution by
		\begin{multline*}
			\frac{2^{\f32j}}{|t|}\min(2^{2j}, 2^{\frac{3j}2}|t|^{-1/2}, 2^{2j} |t|^{-1})\int_{\R^4} |V(x_1)|(|x-x_1|^{-\f32}+|x_1-y|^{-\f32})\\
			\les \frac{2^{\f32j}}{|t|}\min(2^{2j}, 2^{\frac{3j}2}|t|^{-1/2}, 2^{2j} |t|^{-1}),
		\end{multline*}
		uniformly in $x,y$.
		When we have $\mR_LV\mR_H$ or $\mR_HV\mR_L$, the argument above applies with 
		$$
		|a(z)|\les \frac{\chi_j(z)\widetilde \chi(zr)\chi(zs) z^{2}V(x_1)}{(1+zr)^{\frac{1}{2}} s^{3}}, \quad |\partial_z a(z)|\les \frac{\chi_j(z)\widetilde \chi(zr)\chi(zs) zV(x_1)}{(1+zr)^{\frac{1}{2}} s^{3}}.
		$$
		One may then bound the contribution by 
		\begin{multline*}
			\frac{2^{j}}{|t|}\min(2^{2j}, 2^{\frac{3j}2}|t|^{-1/2}, 2^{2j} |t|^{-1})\int_{\R^4} |V(x_1)|(|x-x_1|^{-3}+|x_1-y|^{-3})\\
			\les \frac{2^{j}}{|t|}\min(2^{2j}, 2^{\frac{3j}2}|t|^{-1/2}, 2^{2j} |t|^{-1}),
		\end{multline*}
		uniformly in $x,y$.
	\end{proof}

	We now need only control the tail of \eqref{eqn:high BS}, we expand the free resolvents as $\mR_0=\mR_L+\mR_H$ to see that
	\begin{align}\label{eqn:hi iterated}
		\mR_0V\mR_V V\mR_0=\mR_HV\mR_V V \mR_H+\mR_HV\mR_V V \mR_L+\mR_LV\mR_V V \mR_H+\mR_LV\mR_V V \mR_L
	\end{align}
	The first term requires no further iteration due to the fact that $\mR_H$ is locally $L^2$ by \eqref{eqn:RH bds}.  
	To control the contribution of this piece we use the  following (slightly modified) lemma from
	\cite{Sc2}, see \cite[Lemma 3.3]{egd} or \cite{EGT2}.
	\begin{lemma}\label{stat phase general}
		
		Let $\phi'(z_0)=0$ and $1\leq \phi'' \leq C$.  Then,
		\begin{multline*}
			\bigg| \int_{-\infty}^{\infty} e^{-it\phi(z)} a(z)\, dz \bigg|
			\lesssim \int_{|z-z_0|<|t|^{-\frac{1}{2}}} |a(z)|\, dz \\
			+|t|^{-1} \int_{|z-z_0|>|t|^{-\frac{1}{2}}} \bigg( \frac{|a(z)|}{|z-z_0|^2}+
			\frac{|a'(z)|}{|z-z_0|}\bigg)\, dz.
		\end{multline*}
		
	\end{lemma}
	We cannot use Lemma~\ref{lem:high stat phase2} since derivatives of $\mR_V$ do not decay faster in $z$ than $\mR_V$ for large $z$.  The limiting absorption principle, \eqref{eqn:LAP}, only ensures that derivatives of the operator are bounded between appropriately weighted $L^2$ spaces.

	\begin{lemma}\label{lem:RhRh high}
		
		Under the assumptions of Propostion~\ref{prop:high energy}, we have
		$$
		\sup_{x,y\in\R^4} \bigg| \int_0^\infty e^{-it\sqrt{z^2+m^2}}\frac{z\chi_j(z)}{\sqrt{z^2+m^2}} \mR_H^\pm V\mR_V^\pm V\mR_H^\pm (z)(x,y)\, dz\bigg| \les \min(2^{4j}, 2^{11j/2}|t|^{-2}).
		$$
		
	\end{lemma}

	\begin{proof}[Proof of Lemma~\ref{lem:RhRh high}]
		We consider the `+' case and omit the superscripts, the `-' follows with minor changes.
		By the triangle inequality, \eqref{eqn:RH bds}, \eqref{eqn:LAP}, and Lemma~\ref{lem:spatial estimates}, we have
		\begin{multline*}
			\bigg| \int_0^\infty e^{-it\sqrt{z^2+m^2}}\frac{z\chi_j(z)}{\sqrt{z^2+m^2}} \mR_HV\mR_V V\mR_H(z)(x,y)\, dz\bigg|\\
			\les \int_{0}^\infty z^3\chi_j(z) \|V(\cdot)|x-\cdot|^{-\f32}\|_{L^{2,\frac12+}}\|\mR_V\|_{L^{2,\frac12+}\to L^{2,-\frac12-}}\|V(\cdot)|y-\cdot|^{-\f32}\|_{L^{2,\frac12+}}\, dz\les 2^{4j},
		\end{multline*}
		uniformly in $x,y$
		provided $|V(x)|\les \la x\ra^{-2-}$.
		Due to the cut-off $\chi_j(z)$, we may integrate by parts without boundary terms and we have to control
		\begin{align}\label{eqn:hi tail hihi}
			\frac{1}{-it}\int_0^\infty e^{-it\sqrt{z^2+m^2}}\partial_z \bigg(\chi_j(z) \mR_HV\mR_V V\mR_H(z)(x,y) \bigg)\, dz.
		\end{align}
		When the derivative acts on the cut-off or $\mR_V$, we may integrate by parts again.  The more delicate case is when the derivative acts on $\mR_V$, then we consider
		$$
		\frac{1}{|t|^2}\bigg|\int_0^\infty e^{-it\sqrt{z^2+m^2}}\partial_z \bigg( \frac{\chi_j(z) \sqrt{z^2+m^2}}{z} \mR_HV[\partial_z\mR_V]V\mR_H(z)(x,y) \bigg)\, dz\bigg|.
		$$
		By symmetry considerations, \eqref{eqn:RH bds} and the limiting absorption principle, this is bounded by	
		\begin{multline*}
			\frac{1}{|t|^2} \sum_{j=0}^1 \int_0^\infty z^3\chi_j(z)  \bigg\|\frac{V(\cdot)}{|x-\cdot|^{\frac{1}2+j}}\bigg\|_{L^{2,\frac32+j+}}\|\partial_z^{j+1} \mR_V\|_{L^{2,\frac32+j+}\to L^{2,-\frac32-j-}} \bigg\|\frac{V(\cdot)}{|y-\cdot|^{\frac{3}2}}\bigg\|_{L^{2,\frac32+j+}}\, dz\\
			\les \frac{2^{4j}}{|t|^2},
		\end{multline*}
		uniformly in $x,y$.  We note that for $j=0,1$ we have
		$$
			\bigg\|\frac{V(\cdot)}{|x-\cdot|^{\frac{1}2+j}}\bigg\|_{L^{2,\frac32+j+}}^2 \les 
			\int_{\R^4} \frac{\la x_1\ra^{3+2j-2\delta}}{|x-x_1|^{1+2j}} \, dx_1\les 1
		$$
		uniformly in $x$ provided $\delta>3$ by Lemma~\ref{lem:spatial estimates}.
		
		The most restrictive case is when a derivative acts on $\mR_H$ and we need to utilize the oscillation.  Without loss of generality, we consider the case when the derivative acts on the first $\mR_H$ in \eqref{eqn:hi tail hihi} and consider only the + resolvent case, the other cases follow by symmetry considerations.  In this case we seek to bound
		\begin{align*}
			\frac{1}{-it}\int_0^\infty e^{-it\sqrt{z^2+m^2}}\chi_j(z) [\partial_z \mR_H]V\mR_V V\mR_H(z)(x,y) \, dz=\frac{1}{-it} \int_0^\infty e^{-i2^{-3j}t \phi(z)}a(z,x,y) \, dz,
		\end{align*}
		where $\phi(z)=2^{3j}(\sqrt{z^2+m^2}-z(|x|+|y|)/t)$ is rescaled so that $\phi''(z)\approx 1$ on the support of $\chi_j$.  We include the phase of the non-differentiated $\mR_H$ for flexibility to apply this proof in other cases.  Here,
		\begin{multline*}
			a(z,x,y)\\
			= \int_{\R^8}\chi_j(z)  O_1(z^4) \omega(z|x-x_1|) e^{iz(|x-x_1|-|x|+|x_2-y|-|y|)}V\mR_V(z)V(x_1,x_2) \omega(z|x_2-y|)\, dx_1\, dx_2.
		\end{multline*}
		Where $|\partial_u^j\omega(u)|\les \la u\ra^{-\frac12-j}$.  Noting that when $\gamma>2$, by Lemma~\ref{lem:spatial estimates}, 
		$
		\| \la \cdot \ra^{-\gamma} |x-\cdot|^{-\frac12} \|_2\les \la x\ra^{-\frac12}, 
		$
		so that we have the bound
		$$
		|a(z,x,y)|+|\partial_z a(z,x,y)|\les 2^{3j}\chi_j(z)\la x\ra^{-\frac12} \la y\ra^{-\frac12}.
		$$
		Applying Lemma~\ref{stat phase general} we may bound this term's contribution to $\eqref{eqn:hi tail hihi}$ by
		\begin{align*}
			\frac{1}{|t|} \int_{|z-z_0|<\sqrt{2^{3j}/t}} |a(z)|\, dz+\frac{2^{3j}}{|t|^2}\int_{|z-z_0|>\sqrt{2^{3j}/t}}\bigg(\frac{|a(z)|}{|z-z_0|}+\frac{|a'(z)|}{|z-z_0|^2} \bigg)\, dz.
		\end{align*}
		When the critical point of the phase, $z_0=\frac{m(|x|+|y|)}{\sqrt{t^2-(|x|+|y|)^2}}$, is in a sufficiently small neighborhood of the the support of $a(z)$, we must have that $t\approx (|x|+|y|)$ so that $\la x\ra^{-\frac12} \la y\ra^{-\frac12}\les |t|^{-\frac12}$.  In this case we bound the integrals by
		$$
		|t|^{-1}2^{3j}\la x\ra^{-\frac12}\la y\ra^{-\frac12}\bigg( \sqrt{\frac{2^{3j}}{t}}+|t|^{-1}2^{3j} \frac{2^j}{\sqrt{2^{3j}/t}}\bigg) \les 2^{\frac{11j}{2}}|t|^{-2}.
		$$
		In the case that $t\not \approx (|x|+|y|)$, we don't rescale the phase and instead use that
		$$
		|\partial_z\big( \sqrt{z^2+m^2}-z(|x|+|y|)/t \big)|\gtrsim 1,
		$$
		to integrate by parts.  Noting that $a''(z,x,y)$ satisfies the same bounds as $a(z,x,y)$ and $a'(z,x,y)$ allows us to integrate by parts again to bound with $2^{4j}|t|^{-2}$ uniformly in $x,y$.	
	\end{proof}

	We now consider the second term in \eqref{eqn:hi iterated}, the analysis for the third term is identical up to symmetry.  Iterating the resolvent identity and decomposition $\mR_0=\mR_L+\mR_H$ we derive the identity
	\begin{multline}\label{eqn:hi bornlong}
		\mR_HV\mR_V V \mR_L=\mR_HV(\mR_0-\mR_V V\mR_0) V \mR_L\\
		=\mR_HV\mR_0V\mR_L-\mR_H V\mR_V V\mR_H V \mR_L-\mR_H V\mR_V V\mR_L V \mR_L\\
		=\mR_HV\mR_0V\mR_L-\mR_H V\mR_V V\mR_H V \mR_L-\mR_H V\mR_0 V\mR_L V \mR_L+\mR_H V\mR_V V\mR_0 V\mR_L V \mR_L
	\end{multline}
	Effectively, anytime we have $\mR_H$ on both sides we can stop the iteration process as the resulting integral kernel will be locally $L^2(\R^4)$ and $\mR_H$ brings the growth in $z$ that we seek to minimize.  For inner resolvents, combining \eqref{eqn:RH bds} and \eqref{eqn:RL bds} for $k=0,1,2$ we have
	\begin{align}\label{eqn:R0 high}
		|\partial_z^k \mR_0^\pm(z)(x,y)|\les z^{\frac32}\bigg(\frac{1}{|x-y|^3}+\frac{1}{|x-y|^{\frac32-k}}\bigg).
	\end{align}
	We select the value of $\epsilon$ in \eqref{eqn:RL bds} to ensure the iterated resolvents are locally $L^2$ and that the accumulated powers of $z$ are at most $\frac32$ on either side of the $\mR_V$.

	Any summand in the expansions above not containing $\mR_V$ can be controlled as in the proof of Lemma~\ref{lem:hi bornterm}.  Any term containing both an $\mR_H$ and $\mR_V$ may be controlled as in the proof of Lemma~\ref{lem:RhRh high}, which we formulate as the following corollary.
	
	\begin{corollary}\label{cor:high Rh stuff}
		
		Under the assumptions of Proposition~\ref{prop:high energy}, we have
		$$
		\sup_{x,y\in\R^4} \bigg| \int_0^\infty e^{-it\sqrt{z^2+m^2}}\frac{z\chi_j(z)}{\sqrt{z^2+m^2}} \mR_HV\mR_V \mathcal J(z)(x,y)\, dz\bigg| \les \min(2^{4j}, 2^{(11j/2)}|t|^{-2}).
		$$
		where $\mathcal J(z)$ is one of $\mR_HV\mR_L$, or $\mR_0V\mR_LV\mR_L$.  Furthermore, the $\mR_H$ on the left may be replaced by    $\mR_LV\mR_H$, or $\mR_LV\mR_LV\mR_H$ with the same upper bound.
		
	\end{corollary}
	
	\begin{proof}
		
		By \eqref{eqn:RH bds}, \eqref{eqn:RL bds}, and \eqref{eqn:R0 high}, and repeated use of Lemma~\ref{lem:spatial estimates} we have (for $j=0,1,2$)
		$$
		|\partial_z^j \mR_HV\mR_L(z)(x_2,x_4)|, |\partial_z^j \mR_0V\mR_LV\mR_L(z)(x_2,x_4)|\les z^{\frac32}\la x_2\ra^{\frac12}.
		$$
		These bounds, along with the limiting absorption principle, \eqref{eqn:LAP}, allow one to obtain the same bounds for $a(z,x,y)$ as in Lemma~\ref{lem:RhRh high}, from which the claim follows.  Here we need $|V(x)|\les \la x\ra^{-5-}$ to control the growth in $x_2$ if two derivative act on an $\mR_H$ closest to $\mR_V$, to ensure that $V  \la \cdot \ra^{\f12}\in L^{2,\f52+}$.
		
		In the case of $\mR_LV\mR_H$ on the left, the required bounds for $a(z,x,y)$ hold by \eqref{eqn:RH bds}, \eqref{eqn:RL bds}, and the spatial integral estimate
		$$
			|\partial_z^k (\omega(z|x_2-x_3|)V\mR_L)|\les z^{\frac32}\int_{\R^4} \frac{\la x_3\ra^{-\delta}}{ |x_2-x_3|^{\frac12}|x_3-y|^{4-}}\, dx_3\les z^{\frac32}|x_2-y|^{-\frac12}.
		$$
		These bounds suffice to repeat the argument in Lemma~\ref{lem:RhRh high}.
		Similar adjustments may be made for the final claim, we leave the details to the interested reader.
	\end{proof}
	
	\begin{proof}[Proof of Proposition~\ref{prop:high energy}]
		From the expansion \eqref{eqn:high BS}, we note that Lemmas~\ref{lem:free hi} and \ref{lem:hi bornterm} control the contribution of the first two summands to the Stone's formula.  By Lemma~\ref{lem:RhRh high}, the discussion following it and Corollary~\ref{cor:high Rh stuff}, it remains only to consider the remaining ``Born-like" terms in \eqref{eqn:hi bornlong} that do not contain $\mR_V$.  
		
		Consider the contribution of $\mR_HV\mR_0V\mR_L$, the remaining are similar.  By the triangle inequality, \eqref{eqn:RH bds} and \eqref{eqn:RL bds}, we may bound its contribution by 
		$$
		\int_{\R^8}\int_{0}^\infty \chi_j(z)z^{3-\epsilon}  \frac{|V(x_1)|\, |V(x_2)|}{|x-x_1|^{\f32}|x_2-y|^{4-\epsilon}}\bigg(\frac{1}{|x_1-x_2|^3}+\frac{1}{|x_1-x_2|^{\f32}}\bigg)\, dz\, dx_1\, dx_2.
		$$
		We may bound this uniformly in $x$ and $y$ by $2^{\frac72j}$, provided we select $\epsilon=\frac12+$, to ensure the local singularities are removed by Lemma~\ref{lem:spatial estimates}.  In all other cases, we may integrate by parts.  If the derivative acts on $\mR_H$, we may repurpose the proof of Lemma~\ref{lem:hi bornterm} with 
		$$
		|\partial_z^ka(z)|\les \frac{2^{(\frac52-\epsilon)j}\la |x_1-x_2|\ra |V(x_1)|\, |V(x_2)|}{|x_1-x_2|^3|x_2-y|^{4-\epsilon}} \frac{z^{1-k}\chi_j(z)\widetilde \chi(z|x-x_1|)}{(1+z|x-x_1|)^{\f12}}.
		$$
		To obtain a bound of the form $2^{5j+}|t|^{-2}$.  If the derivative acts on any other piece we may integrate by parts again to obtain a bound of the form
		$$
		\frac{1}{|t|^2}\int_{\R^8} \int_0^\infty z^{3-\epsilon} \frac{ |V(x_1)|(1+|x_1-x_2|^{\f72}) |V(x_2)|}{|x-x_1|^{\f12}|x_1-x_2|^{3}|x_2-y|^{4-\epsilon}}\, dz\, dx_1\, dx_2
		$$
		Here we select $\epsilon=0+$,  and invoke Lemma~\ref{lem:spatial estimates} twice, to obtain bound of size $|t|^{-2} 2^{4j}$ that is uniform in $x,y$.  We require $|V(x)|\les \la x\ra^{-5-}$ to ensure that the integrals converge when both derivatives act on the inner $\mR_0$, since this can grow like $\la x_j-x_{j+1}\ra^{\frac12}$ in the spatial variables.  The analysis of $\mR_HV\mR_0V\mR_LV\mR_L$ follows similarly.
		
		For the final term in \eqref{eqn:high BS} we expand as follows.
		\begin{multline*}
			\mR_LV\mR_V V \mR_L=\mR_LV\big[\mR_0-\mR_0 V\mR_0+\mR_0V \mR_V V \mR_0\big] V \mR_L\\
			=\mR_LV\mR_0V\mR_L-\mR_L V\mR_0 V\mR_H V \mR_L-\mR_L V\mR_0 V\mR_V V\mR_0 V \mR_L		
		\end{multline*}
		The final term requires two more selective iterations of the resolvent identity.  The only portion of it that is not controlled by the previous arguments is when we have only $\mR_L$ contributing and we need to control the contribution of
		$$
			\mR_LV\mR_L V\mR_L  V\mR_V \mR_L V\mR_L  V \mR_L
		$$
		In these contribution of these terms to the $z$ integral, we need only integrate by parts and do not need to concern ourselves with the oscillatory piece.  The Born-like terms can be handled as in the previous discussion.
		
		By the preceding discussion, any term with $\mR_H$ can be controlled.  It suffices to consider the final term in which we much iterate on both sides sufficiently to insure that the iterated $\mR_L$ terms are locally $L^2$.  It remains only to bound the contribution of
		\begin{align}\label{eqn:hi all low}
			\int_0^\infty e^{-it\sqrt{z^2+m^2}}\frac{z\chi_j(z)}{\sqrt{z^2+m^2}}
			\mR_L( V\mR_L )^2 V\mR_V V(\mR_L V)^2 V \mR_L(z)(x,y)\, dz
		\end{align}
		Here we note that, by \eqref{eqn:RL bds} and Lemma~\ref{lem:spatial estimates}  we have
		$$
		\| \mR_L( V\mR_L )^2V\|_{L^{2,\frac12+}}  \les z^{-2}\bigg(\int_{\R^4} \frac{|V(x_3)|\la x_3\ra^{1+}}{|x-x_3|^{4-3\epsilon}} \,dx_3 \bigg)^{\f12}\les z^{-2}.
		$$
		Selecting $\epsilon=\frac23+$ ensures the final bound is valid and holds uniformly in $x$.  From here, we see that
		\begin{align*}
			|	( \mR_L V)^3 \mR_V (V\mR_L V)^3 (z)|\les 	\| ( \mR_L V)^3\|_{L^{2,\frac12+}} \|\mR_{V}\|_{L^{2,\frac12+}\to L^{2,-\frac12-}}\|(V\mR_L V)^3\|_{L^{2,\frac12+}}\les z^{-4},
		\end{align*}
		uniformly in $x,y$.
		Hence, by the triangle inequality and \eqref{eqn:LAP} we have $|\eqref{eqn:hi all low}|\les 1$.  To complete the proof, we integrate by parts twice, in the worse case we have derivatives act on $\mR_V$ and no decay in $z$ occurs.  That is, we have
		$$
		\sup_{x,y\in \R^4}	|\eqref{eqn:hi all low}|\les \la t\ra^{-2}.
		$$
	\end{proof}
	Finally,  we show that the solution operator remains bounded with near minimal assumptions on both the smoothness of the initial data measured with powers of $\la \mathcal H \ra^{-1}$, and decay of the potential.  Here we show that the selective iteration scheme we used on high energy to avoid growth in $z$ may be adapted to the low energy regime to avoid the need for further decay of the potential.
	
	\begin{proof}[Proof of Theorem~\ref{thm:min bd}]
		As noted after Proposition~\ref{prop:high energy}, one needs only $|V(x)|\les \la x\ra^{-2-}$ to show that the Dirac evolution localized to dyadic energies is bounded by $2^{4j}$ uniformly in $x,y$.  Hence, we need only show that the low energy evolution is bounded with lesser decay on the potential.
		
		We first note that, if we do not require differentiability in the error terms, we have the following bounds on the free resolvents when $0<z<1$:
		\begin{align}\label{eqn:RL all}
			|[\mR_0^\pm-\mG_0](z)(x,y)|\les \frac{z^{\f12}}{|x-y|^{\f32}}+\frac{z^{\f12}}{|x-y|^{2}}, \qquad
			|[\mR_L^+-\mR_L^-](z)(x,y)|\les z^2(z|x-y|)^{-\ell}.
		\end{align}
		These bounds hold for any choice of $\ell\geq0$ and any $z\neq0$.  In particular, we may now write
		$$
		M^\pm(z)=T_0+ \mathcal M_0(z), \qquad |\mathcal M_0(z)|\les z^{\f12} v(x)[\mathcal I_{\f52}+\mathcal I_2](x,y)v^*(y).
		$$
		where $\mathcal I_{\f52}:L^{2,s}\to L^{2,-s}$ provided $s>\frac54$, and $\mathcal I_{2}:L^{2,s}\to L^{2,-s}$ provided $s>1$ by Lemma~2.3 in \cite{Jen}.  Thus, if $|v(x)|\les \la x\ra^{-\frac54-}$ we may write 
		$$
		M^{\pm}(z)=T_0+\Gamma(z), \qquad \| \sup_{0<z<z_1} |z|^{-\f12} |\Gamma(z)|\,\|_{L^2\to L^2}\les 1.
		$$
		where the error term is absolutely bounded operator on $L^2$.  By a Neumann series computation, one has that $M^\pm(z)$ are invertible in a sufficiently small neighborhood of the threshold with
		\begin{align}
			M_\pm^{-1}(z)=D_0+\Gamma(z)=O_0(z^0).
		\end{align}
		where $\Gamma(z)$ is as above.
		
		By decomposing $\mR_0^\pm=\mR_L^\pm+\mR_H^\pm$ on both low  energy as well, we can adapt the selective iteration scheme from the high energy argument to both minimize the powers of $z$ on large energy and the needed decay on $V$.  Here, the symmetric resolvent identity also guarantees that
		\begin{align}\label{eqn:symm VRV}
			V\mR_V V=V-v^* M^{-1}v.
		\end{align}
		So that one may use the usual resolvent identity to write
		\begin{align*}
			\mR_V=\mR_0-\mR_0V\mR_0+\mR_0V\mR_V V\mR_0.
		\end{align*}
		We focus on the final term and selectively iterate:
		\begin{multline*}
			\mR_0V\mR_V V\mR_0=\mR_HV\mR_V V\mR_H+\mR_HV\mR_V V\mR_L+\mR_LV\mR_V V\mR_H+\mR_LV\mR_V V\mR_L\\
			=\mR_HV\mR_V V\mR_H+\mR_HV(\mR_0-\mR_V V\mR_0) V\mR_L+\mR_LV(\mR_0-\mR_0 V\mR_V) V\mR_V\\
			+\mR_LV(\mR_0-\mR_0V\mR_0+\mR_0V\mR_V V\mR_0) V\mR_L.
		\end{multline*}
		We then selectively iterate the inner expansions whenever $\mR_L$ contributes and terminate the iteration when $\mR_H$ is used.  For example, we expand $\mR_HV\mR_VV\mR_0V\mR_L$ as 
		$$
		\mR_HV\mR_V V\mR_0 V\mR_L
		=\mR_HV\mR_V V\mR_H V\mR_L+\mR_HV\mR_V V\mR_L V\mR_L.
		$$
		Since both sides of $\mR_V$ in the first term are locally $L^2$, we invoke  \eqref{eqn:symm VRV} to see
		$$
		\mR_HV\mR_V\mR_HV\mR_L=\mR_HV\mR_HV\mR_L-\mR_Hv^*M^{-1}v\mR_HV\mR_L.
		$$
		For the second term, we first utilize the usual resolvent identity, then the symmetric resolvent identity to see
		\begin{multline*}
			\mR_HV\mR_V V\mR_L V\mR_L=\mR_HV\mR_0 V\mR_L V\mR_L-\mR_HV\mR_V V\mR_0 V\mR_L V\mR_L\\
			=\mR_HV\mR_0 V\mR_L V\mR_L-\mR_HV\mR_0 V\mR_L V\mR_L+\mR_Hv^*M^{-1}v\mR_0 V\mR_L V\mR_L.
		\end{multline*}
		The other summands are similarly iterated until either an $\mR_H$ is encountered or there is a sequence of two $\mR_L$ and one $\mR_0$.  Now,
		\begin{align*}
			|v\mR_0 V\mR_L V\mR_L(x_2,y)|\les \sum \int_{\R^{8}}\frac{\la x_2\ra^{-\frac{\delta}{2}}\la x_3\ra^{-\delta} \la x_4\ra^{-\delta}}{|x_2-x_3|^{k_1}|x_3-x_4|^{k_2}|x_4-y|^{k_3}}\, dx_3\, dx_4,
		\end{align*}
		where the sum is taken over $k_1\in \{\frac32,3\}$ and $k_2,k_3\in \{2,3\}$.  Repeated use of Lemma~\ref{lem:spatial estimates} suffices to show that $v\mR_0 V\mR_L V\mR_L(\cdot,y)\in L^2$ uniformly in $y$ when $|V(x)|\les \la x\ra^{-\delta}$ for some $\delta>\frac52$. The case of $k_1=\f32$, $k_2=k_3=2$ neccessitates the condition that $\delta>\f52$.  Similar computations show that $\mR_Hv^*, \mR_HV\mR_Hv^*, \mR_LV\mR_LV\mR_Lv^*\in L^2$.
		
		For the other `Born-like' terms, those without a perturbed resolvent, we note that
		$$
		\int_{\R^4} |\mR_H(z)(x,x_1)V(x_1)\mR_H(x_1,y)|\, dx_1 \les \int_{\R^4} \frac{\la x_1\ra^{-\f52-}}{|x-x_1|^{\f32}|x_1-y|^{\f32}}\, dx_1 \les 1,
		$$
		uniformly in $x,y$ by Lemma~\ref{lem:spatial estimates}.  Similarly
		\begin{multline*}
			\int_{\R^8}|\mR_L(z)(x,x_1)V(x_1)\mR_H(z)(x_1,x_2)V(x_2)\mR_0(z)(x_2,y)|\, dx_1\, dx_2\\
			\les \sum \int_{\R^8} \frac{\la x_1\ra^{-\f52-}\la x_2\ra^{-\f52-}}{|x-x_1|^{k_1}|x_1-x_2|^{\f32}|x_2-y|^{k_2}}\, dx_1\, dx_2,
		\end{multline*}
		where the sum is taken over $k_1\in \{2,3\}$ and $k_2\in \{2,3\}$.  Repeated use of Lemma~\ref{lem:spatial estimates} shows this is bounded uniformly in $x,y$.  When the order of the resolvents is permuted, this argument follows similarly.  
		
		Only in the case when we have all $\mR_L$ contributing do we need to account for the +/- difference in the Stone's formula.  Here, we note that
		\begin{multline*}
			\int_{\R^8}|[\mR_L^+-\mR_L^-](x,x_1)V(x_1)\mR_L(z)(x_1,x_2)V(x_2)\mR_L(z)(x_2,y)|\, dx_1\, dx_2\\
			\les \sum \int_{\R^8} \frac{\la x_1\ra^{-\f52-}\la x_2\ra^{-\f52-}}{|x_1-x_2|^{k_1}|x_2-y|^{k_2}}\, dx_1\, dx_2,
		\end{multline*}
		where the sum is over $k_1,k_2\in \{2,3\}$, here we use that $|\mR_L^+-\mR_L^-|\les z^2$.  By Lemma~\ref{lem:spatial estimates}, the $x_1$ integral is bounded by $\la x_2\ra^{-\f12}$ uniformly in $x$.  Another application of the lemma establishes the bound.  This also suffices for the case of $\mR_LV\mR_L$ where one also needs to utilize the +/- difference.
	\end{proof}

\end{document}